\documentclass[10pt]{article}

\usepackage[english]{babel}
\usepackage{inputenc}				
\usepackage[scaled]{helvet}			
\usepackage{courier}				
\usepackage{eulervm}				
\usepackage[bb=boondox,bbscaled=1.05,scr=dutchcal]{mathalfa}
\normalfont

\usepackage{indentfirst}					
\usepackage{tabularx,tabulary,booktabs}		
\usepackage{caption,subcaption}				
\usepackage{fullpage}						
\usepackage[english]{varioref}				
\usepackage[dvipsnames]{xcolor}				
\usepackage{hyperref,bookmark}				
\usepackage{textcomp}
\usepackage{rotating,afterpage,lscape}	

\usepackage{graphicx,float}			
\usepackage{tikz}					
\usepackage{tikz-cd}				
\usetikzlibrary{
	arrows,decorations.pathreplacing,decorations.markings,shapes.geometric,positioning,calc,fadings,angles,patterns
	}
\graphicspath{{/}}
	
\usepackage{amsmath,amssymb,amsthm,mathtools}
\usepackage{bm,braket,faktor,esint}
\usepackage{enumerate,enumitem}
\numberwithin{equation}{section}
\usepackage[textsize=footnotesize]{todonotes}
\presetkeys%
	{todonotes}%
	{color=Dandelion}{}%

\usepackage{cite}

\mathtoolsset{showonlyrefs=true}	

\newcommand{\de}{\partial}
\newcommand{\dd}{\mathrm{d}}

\newcommand{\ZZ}{\mathbb{Z}}
\newcommand{\QQ}{\mathbb{Q}}
\newcommand{\RR}{\mathbb{R}}

\newcommand{\GG}{\mathbb{G}}

\renewcommand{\emptyset}{\varnothing}

\DeclareMathOperator{\Mod}{Mod}
\newcommand{\Aut}{{\rm Aut}}

\DeclareMathOperator{\Li}{Li}

\theoremstyle{plain}
\newtheorem{thm}{Theorem}[section]
\newtheorem{prop}[thm]{Proposition}
\newtheorem{lem}[thm]{Lemma}
\newtheorem{cor}[thm]{Corollary}

\theoremstyle{definition}
\newtheorem{defn}[thm]{Definition}
\newtheorem{rem}[thm]{Remark}

\makeatletter
\renewenvironment{abstract}{%
		\if@twocolumn
			\section*{\abstractname}%
		\else \normalsize 
			\begin{center}%
				{\bfseries \normalsize\abstractname\vspace{\z@}}
			\end{center}%
			\quotation
		\fi}
		{\if@twocolumn\else\endquotation\fi}
\makeatother

\definecolor{webbrown}{rgb}{0.65, 0.16, 0.16}

\hypersetup{
colorlinks=true, linktocpage=true, pdfstartpage=1, pdfstartview=FitV,breaklinks=true, pdfpagemode=UseNone, pageanchor=true, pdfpagemode=UseOutlines,plainpages=false, bookmarksnumbered, bookmarksopen=true, bookmarksopenlevel=1,hypertexnames=true, pdfhighlight=/O,urlcolor=webbrown, linkcolor=RoyalBlue, citecolor=ForestGreen}

\UseRawInputEncoding

\title{\textsc{Around the combinatorial unit ball of measured foliations on bordered surfaces}\vspace{1cm}}
\date{\vspace{-5ex}}
 \author{%
	Ga\"etan Borot\,%
		\footnote{Max Planck Institut f\"ur Mathematik, Vivatsgasse 7, 53111 Bonn, Germany.}\,\,%
		\footnote{Humboldt-Universit\"at zu Berlin, Institut f\"ur Mathematik und Institut f\"ur Physik, Rudower Chaussee 25, 10247 Berlin, Germany.}\,\,,
	S\'everin Charbonnier\,%
		\footnotemark[1]\,\,%
		\footnote{Universit\'e de Paris, CNRS, IRIF, F-75006, Paris, France}\,\,,
	Vincent Delecroix\,%
		\footnote{LaBRI, UMR 5800, B\^atiment A30, 351 cours de la Lib\'eration, 33405 Talence Cedex, France.}\,\,,\\
	Alessandro Giacchetto\,%
		\footnotemark[1]\,\,%
		\footnote{Universit\'e Paris-Saclay, CNRS, CEA, Institut de Physique Th\'eorique (IPhT), 91191 Gif-sur-Yvette, France}\,\,,
	Campbell Wheeler\,%
		\footnotemark[1]%
}

\begin{document}

\maketitle 

\vspace{2cm}

\begin{abstract}
	\noindent
	The volume $\mathscr{B}_{\Sigma}^{{\rm comb}}(\GG)$ of the unit ball --- with respect to the combinatorial length function $\ell_{\GG}$ --- of the space of measured foliations on a stable bordered surface $\Sigma$ appears as the prefactor of the polynomial growth of the number of multicurves on $\Sigma$. We find the range of $s \in \RR$ for which $(\mathscr{B}_{\Sigma}^{{\rm comb}})^{s}$, as a function over the combinatorial moduli spaces, is integrable with respect to the Kontsevich measure. The results depends on the topology of $\Sigma$, in contrast with the situation for hyperbolic surfaces where \cite{AHA19} recently proved an optimal square-integrability.
\end{abstract}

\thispagestyle{empty}
\bigskip

\newpage
\tableofcontents

\newpage
\section{Introduction}
\label{sec:intro}

\subsection{Measured foliations and Teichm\"uller spaces}

Consider a smooth connected oriented surface $\Sigma$ of genus $g \geq 0$ with $n > 0$ labelled boundaries which is stable (\textit{i.e.} $2g - 2 + n > 0$), and denote by $\Mod_{\Sigma}^{\de}$ its pure mapping class group. A key role in this work is played by the space ${\rm MF}_{\Sigma}$ of measured foliations on $\Sigma$ (considered up to Whitehead equivalence), where we require that $\partial \Sigma$ is a union of singular leaves. For later convenience, we also include the empty foliation. From the work of Thurston, ${\rm MF}_{\Sigma}$ is a topological space of dimension $6g - 6 + 2n$ equipped with a piecewise linear integral structure. The set of integral points in ${\rm MF}_{\Sigma}$ is identified with the set of multicurves $M_{\Sigma}$ on $\Sigma$, and in fact ${\rm MF}_{\Sigma}$ is the completion of the set of $\QQ_{+}$-weighted multicurves. The corresponding volume form $\mu_{\textup{Th}}$, called the Thurston measure, can be defined by asymptotics of lattice point counting.

\medskip

There are two other natural spaces attached to $\Sigma$: for a fixed $L = (L_1,\dots,L_n) \in \RR_{+}^{n}$, we consider the ordinary Teichm\"uller space $\mathcal{T}_{\Sigma}(L)$ and the combinatorial one $\mathcal{T}_{\Sigma}^{\textup{comb}}(L)$. The former is identified with the set of isotopy classes of hyperbolic structures on $\Sigma$ making the boundaries geodesics of length $L$ (we may allow $L_1 = \cdots = L_n = 0$, meaning that each boundary component is replaced by a puncture and we consider complete hyperbolic structures). The latter is the set of isotopy classes of embedded metric ribbon graphs on $\Sigma$ with fixed boundary length $L$, onto which $\Sigma$ retracts. By definition the associated moduli spaces are
\[
	\mathcal{M}_{g,n}(L) = \mathcal{T}_{\Sigma}(L)/{\rm Mod}_{\Sigma}^{\partial},
	\qquad
	\mathcal{M}_{g,n}^{\textup{comb}}(L) = \mathcal{T}_{\Sigma}^{\textup{comb}}(L)/{\rm Mod}_{\Sigma}^{\partial}.
\]
Such Teichm\"uller spaces are equipped with a natural ${\rm Mod}_{\Sigma}^{\partial}$-invariant symplectic form: the Weil--Petersson form $\omega_{{\rm WP}}$ in the hyperbolic setting \cite{Goldman}, and the Kontsevich form $\omega_{\textup{K}}$ in the combinatorial one \cite{Kontsevich}. Both measures $\mu_{{\rm WP}}$ and $\mu_{\textup{K}}$ assign a finite volume to the respective moduli spaces.

\medskip

$\mathcal{T}_{\Sigma}(L)$ and $\mathcal{T}_{\Sigma}^{\textup{comb}}(L) $ are topologically the same space but carry different geometries; the ordinary Teichm\"uller space has a natural smooth structure, while the combinatorial one is a polytopal complex. Nevertheless, the two geometries share many interesting properties: they posses global coordinates that are Darboux for the associated symplectic forms \cite{WKarticle,WolpertWP}, and they admit a recursive partition of unity (Mirzakhani--McShane identities) that integrate to a recursion for the associated symplectic volumes \cite{WKarticle,Mirza1}. In this article we shall examine another aspect of this parallelism regarding the asymptotic count of multicurves.

\subsection{Random geometry of multicurves}

Since the Weil--Petersson and the Kontsevich measures assign a finite volume to the respective moduli spaces, normalising them defines a probability measure and thus the ensemble of random hyperbolic surfaces and the ensemble of random combinatorial surfaces. We shall study the behavior of the length spectrum of multicurves in these two ensembles, initiated by Mirzakhani \cite{Mir08} in the hyperbolic setting and in \cite{WKarticle} in the combinatorial setting. Concretely, the data of a hyperbolic metric $\sigma \in \mathcal{T}_{\Sigma}(L)$ or of an embedded metric ribbon graph $\GG \in \mathcal{T}_{\Sigma}^{\textup{comb}}(L)$ induces a length function
\[
	\begin{aligned}
		{\rm MF}_{\Sigma} & \longrightarrow \RR_{+} \\
		\mathcal{F} & \longmapsto \ell_{\sigma}(\mathcal{F})
	\end{aligned},
	\qquad\qquad
	\begin{aligned}
		{\rm MF}_{\Sigma} & \longrightarrow \RR_{+} \\
		\mathcal{F} & \longmapsto \ell_{\GG}(\mathcal{F})
	\end{aligned}.
\]
We want to study the Thurston volume of the unit balls --- with respect to these lengths functions --- in the space of measured foliations:
\[
	\mathscr{B}_{\Sigma}(\sigma) = \mu_{\textup{Th}}\big(\Set{\mathcal{F} \in {\rm MF}_{\Sigma} | \ell_{\sigma}(\mathcal{F}) \leq 1}\big),
	\qquad\qquad
	\mathscr{B}_{\Sigma}^{\textup{comb}}(\GG) = \mu_{\textup{Th}}\big(\Set{\mathcal{F} \in {\rm MF}_{\Sigma} | \ell_{\GG}(\mathcal{F}) \leq 1}\big).
\]

The function $\mathscr{B}_\Sigma$ of $\sigma \in \mathcal{T}_{\Sigma}(L)$ (resp. $\mathscr{B}_{\Sigma}^{\textup{comb}}$ of $\mathbb{G} \in \mathcal{T}_{\Sigma}^{{\rm comb}}(L)$) is mapping class group invariant, therefore descends to a function $\mathscr{B}_{g,n}$ (resp. $\mathscr{B}_{g,n}^{\textup{comb}}$) on the moduli spaces $\mathcal{M}_{g,n}(L)$ (resp. $\mathcal{M}_{g,n}^{\textup{comb}}(L)$). They naturally appear in the study of the asymptotic number of multicurves with bounded length:
\[
	\mathscr{B}_{\Sigma}(\sigma) = \lim_{r \rightarrow \infty} \frac{\#\Set{\gamma \in M_{\Sigma} | \ell_{\sigma}(\gamma) \leq r}}{r^{6g - 6 + 2n}},
	\qquad
	\mathscr{B}_{\Sigma}^{\textup{comb}}(\GG) = \lim_{r \rightarrow \infty} \frac{\#\Set{\gamma \in M_{\Sigma} | \ell_{\GG}(\gamma) \leq r }}{r^{6g - 6 + 2n}}.
\]
Because the function $\ell_\sigma$ on ${\rm MF}$ is not very explicit it is delicate to extract properties of $\mathscr{B}_\Sigma$. In~\cite{Mir08} Mirzakhani initiated the study of $\mathscr{B}_{\Sigma}(\sigma)$, and she established the following properties for punctured surfaces --- \textit{i.e.} over $\mathcal{T}_{\Sigma}(0)$. Her proof can be extended to bordered surfaces and more generally to lengths measured with respect to a filling current~\cite{EPS}.

\begin{thm}\cite{Mir08}
	For any $L \in \RR_{\geq 0}^{n}$, the function $\mathscr{B}_{\Sigma}$ is continuous on $\mathcal{T}_{\Sigma}(L)$, and induces a proper function on $\mathcal{M}_{g,n}(L)$ whose $s$-th power is integrable  with respect to $\mu_{{\rm WP}}$ for any $s < 2$, and not integrable for $s > 2$.
\end{thm}

Arana-Herrera and Athreya~\cite{AHA19} recently proved integrability for the limit case $s = 2$ in the case of punctured surfaces.

\medskip

The $L^1$-norm of $\mathscr{B}_{g,n}$ is well-understood. It is in fact the same in the hyperbolic and combinatorial setting irrespectively of boundary lengths and coincides, up to normalisation, with the Masur--Veech volume ${\rm MV}_{g,n}$ of the top stratum of the moduli space of meromorphic quadratic differentials on punctured surfaces with simple poles at the punctures:
\begin{equation}
	\label{MVnorm}
	\forall L \in \mathbb{R}_{\geq 0}^n,\qquad \frac{{\rm MV}_{g,n}}{2^{4g - 2 + n}(4g - 4 + n)!(6g - 6 + 2n)} =  \int_{\mathcal{M}_{g,n}(L)} \mathscr{B}_{g,n}\dd\mu_{{\rm WP}} = \int_{\mathcal{M}_{g,n}^{{\rm comb}}(L)} \mathscr{B}_{g,n}^{{\rm comb}} \dd \mu_{{\rm K}}.
\end{equation}
We refer to \cite{ABCDGLW19,WKarticle,DGZZ19,Mirzaergo} for the justification of the various parts of this statement. Besides, the values of ${\rm MV}_{g,n}$ can be computed in many ways \cite{ABCDGLW19,CMSapp,DGZZ19,KazarianMV,YZZMV} and its large genus asymptotics are known \cite{Aggarwalquad,ADGZZ}.

\medskip

In contrast, the computation of the $L^2$-norm of $\mathscr{B}_{g,n}$ is still an open problem. In this article, we study the combinatorial analogue of the above quantities. We find that the computations are much simpler, due to the polytopal nature of both ${\rm MF}_{\Sigma}$ and $\mathcal{M}_{\Sigma}^{\textup{comb}}(L)$, that allows us to explicitly describe the function $\mathscr{B}_{\Sigma}^{{\rm comb}}$ (see Proposition~\ref{prop:volume:rational:fnct})
 and have a good understanding of its domain of integration.

\medskip

Consider, for example, a torus with one boundary component. The associated moduli space $\mathcal{M}_{1,1}^{{\rm comb}}(L)$ has a single top-dimensional cell given by
\[
	\Set{
		(\ell_A,\ell_B,\ell_C) \in \RR_{+}^3 | \ell_A + \ell_B + \ell_C = \tfrac{L}{2}
	}
	\big/ \ZZ_{6}.
\]
Here $\ZZ_{3} \subset \ZZ_6$ is cyclically permuting the three components, while $\ZZ_2 \subset \ZZ_6$ is the elliptic involution stabilising every point. Moreover, the Kontsevich measure on such cell is $\dd\mu_{\textup{K}} = \dd \ell_A \dd \ell_B$. We will see that
\[
	\mathscr{B}_{1,1}^{{\rm comb}}(\ell_A,\ell_B,\ell_C)
	=
	\frac{L}{2} \frac{1}{(\ell_A + \ell_B)(\ell_B + \ell_C)(\ell_C + \ell_A)},
\]
and after integration
\[
	\int_{\mathcal{M}_{1,1}^{{\rm comb}}(L)} \big(\mathscr{B}_{1,1}^{{\rm comb}}\big)^{s}\dd\mu_{\textup{K}}
	=
	\frac{L^{1 - s}}{3} \int_{(0,1)^2} \dd x\,\dd y\,(1+y)^{3(s - 1)} y^{1 - s}(1 - y^2x^2)^{-s}. \]
In particular, we find integrability if and only if $s < 2$ and
\[
	\int_{\mathcal{M}_{1,1}^{{\rm comb}}(L)} \mathscr{B}_{1,1}^{{\rm comb}}\,\dd\mu_{\textup{K}} = \frac{\pi^2}{24},
\]
which is in agreement with the Masur--Veech volume ${\rm MV}_{1,1} = \frac{2\pi^2}{3}$. 

\medskip

More generally, the explicit description of $\mathscr{B}_{\Sigma}^{{\rm comb}}$ allows us to characterise integrability, which surprisingly depends on the topology of $\Sigma$.

\begin{thm}\label{main:thm}
	For any $L \in \RR_{+}^n$, the function $\mathscr{B}_{\Sigma}^{{\rm comb}}$ is continuous on $\mathcal{T}_{\Sigma}^{{\rm comb}}(L)$. It induces on $\mathcal{M}_{g,n}^{{\rm comb}}(L)$ a proper function $\mathscr{B}_{g,n}^{{\rm comb}}$ whose $s$-th power is integrable if and only if $s < s_{g,n}^*$, where assuming that $L$ is non-resonant according to Definition~\ref{nonresdef}:
\[
		s^*_{g,n} =
		\begin{dcases}
			+\infty & \text{if $g = 0$ and $n = 3$,} \\[2pt]
			2 & \text{if $g = 0$ and $n \in \set{4,5}$, or $g = 1$ and $n = 1$,}  \\[2pt]
			\frac{4}{3} + \frac{2}{3}\frac{1}{\lfloor n/2 \rfloor - 2} & \text{if $g = 0$ and $n \geq 6$,} \\[2pt]
			\frac{4}{3} & \text{if $g = 1$ and $n \geq 2$,} \\[2pt]
			1 + \frac{1}{3(2g - 3)} & \text{if $g \geq 2$ and $n = 1$,} \\[2pt]
			1 + \frac{1}{3(2g - 1)} & \text{if $g \geq 2$ and $n \geq 2$.}
		\end{dcases}
	\]
\end{thm}

Note that generic $L$ are non-resonant. The $(0,3)$ case is trivial, since $\mathcal{M}_{0,3}^{{\rm comb}}(L)$ is a point. The cases $(0,4)$, $(0,5)$, and $(g,1)$ for $g \ge 1$ are also special. Theorem~\ref{main:thm} is the central result of the article. It is proved in Section~\ref{sec:integrability}, with three main ingredients:
\begin{itemize}
	\item a study of the geometry of the cells in the combinatorial moduli space (Section~\ref{sec:cells:tcomb});
	\item an independent characterization of integrability for inverse powers of products of linear forms with positive coefficients via convex geometry (Appendix~\ref{app:integrability:lemma});
	\item the identification of the regions of worst divergence in the integrals of $(\mathscr{B}_{g,n}^{{\rm comb}})^s$, which reduce to questions involving the combinatorics of ribbon graphs and their subgraphs (Section~\ref{sec:worst:divergence}).
\end{itemize}

\medskip

The origin of the difference in integrability between the two settings can be explained as follows. In the hyperbolic case, $\mathscr{B}_{\Sigma}$ is bounded from above by the product of inverse of lengths of short curves \cite[Proposition 3.6]{Mir08}. By the collar lemma such curves cannot intersect each other, so we can include them in a pair of pants decomposition. This is sufficient to show that $\mathscr{B}_{g,n}^{s}$ is integrable for $s < 2$. The integrability for $s = 2$ is proved via a finer upper bound in \cite{AHA19}. In the combinatorial case, there is a similar bound but no collar lemma, so there can be more short curves and this results in less integrability.

\subsection{Consequences for hyperbolic surfaces with large boundaries}

The two Teichm\"uller spaces do not just sit apart from each other. From the works of Penner \cite{Pen87}, Bowditch--Epstein \cite{BE88} and Luo \cite{Luo07} on the spine construction, there is a ${\rm Mod}_{\Sigma}^{\partial}$-equivariant homeomorphism between the Teichm\"uller space $\mathcal{T}_{\Sigma}$ and its combinatorial counterpart
\[
	{\rm sp} \colon \mathcal{T}_{\Sigma}(L) \longrightarrow \mathcal{T}_{\Sigma}^{{\rm comb}}(L),\qquad L \in \RR_{+}^n.
\]
The rescaling flow acts for $\beta > 0$ by taking $\sigma \in \mathcal{T}_{\Sigma}(L)$ and sending it to
\[
	\sigma^{\beta} = ({\rm sp}^{-1} \circ \rho_{\beta} \circ {\rm sp})(\sigma) \in \mathcal{T}_{\Sigma}(\beta L),
\]
where $\rho_{\beta} \colon \mathcal{T}_{\Sigma}^{{\rm comb}}(L) \rightarrow \mathcal{T}_{\Sigma}^{{\rm comb}}(\beta L)$ is the operation of dilating the metric on the ribbon graph by a factor $\beta$. In many ways \cite{WKarticle,Do10,Luo07,Mon09}, the asymptotic geometry of hyperbolic surfaces with metric $\sigma^{\beta}$ when $\beta \rightarrow \infty$ is described by the combinatorial geometry ${\rm sp}(\sigma) \in \mathcal{T}_{\Sigma}^{{\rm comb}}$. In particular, \cite{Mon09} proves that the Weil--Petersson measure on $\mathcal{T}_{\Sigma}(\beta L)$ converges to the Kontsevich measure on $\mathcal{T}_{\Sigma}^{{\rm comb}}(L)$, meaning that the Jacobian
\[
	{\rm Jac}_{\beta} = \frac{1}{\beta^{6g - 6 + 2n}}\,\frac{({\rm sp}^{-1} \circ \rho_{\beta})^{\ast} \dd \mu_{{\rm WP}}}{\dd \mu_{\textup{K}}}
\]
converges pointwise on $\mathcal{T}_{\Sigma}^{{\rm comb}}(L)$ to $1$.

\medskip

The non-integrability of $(\mathscr{B}_{g,n}^{{\rm comb}})^{s}$ implies an anomalous scaling of the integral of $\mathscr{B}_{g,n}^{s}$ over the moduli space of bordered Riemann surfaces when the boundary lengths tend to $+\infty$. Indeed, the combinatorial function describes the large time limit of the hyperbolic one under the rescaling flow, that is
\begin{equation}
\label{conver}	\lim_{\beta \rightarrow \infty} \beta^{6g - 6 + 2n} ({\rm sp}^{-1} \circ \rho_{\beta})^{\ast}\mathscr{B}_{\Sigma} = \mathscr{B}_{\Sigma}^{{\rm comb}}
\end{equation}
uniformly on compacts of $\mathcal{T}_{\Sigma}^{{\rm comb}}$. But, by change of variable, we have for any $L \in \RR_{+}^n$
\[
	\beta^{(6g - 6 + 2n)(1 - s)} \int_{\mathcal{M}_{g,n}(\beta L)} \mathscr{B}_{g,n}^{s}\,\dd\mu_{{\rm WP}}
	=
	\int_{\mathcal{M}_{g,n}^{{\rm comb}}(L)}
		{\rm Jac}_{\beta}
		\cdot
		\beta^{6g - 6 + 2n} \big(({\rm sp}^{-1} \circ \rho_{\beta})^{\ast} \mathscr{B}_{g,n}^{s}\big)
		\, \dd\mu_{\textup{K}}.
\]
Then, the Fatou lemma and the pointwise convergence of the integrand as $\beta \to +\infty$ imply that
\begin{equation}
\label{equnugn}	\int_{\mathcal{M}_{g,n}^{{\rm comb}}(L)} (\mathscr{B}_{g,n}^{{\rm comb}})^{s}\dd\mu_{\textup{K}}
	\leq
	\liminf_{\beta \rightarrow \infty} \beta^{(6g - 6 + 2n)(1 - s)} \int_{\mathcal{M}_{g,n}(\beta L)} \mathscr{B}_{g,n}^{s}\,\dd\mu_{{\rm WP}}.
\end{equation}

Theorem~\ref{main:thm} then implies
\begin{cor}
For $s \geq s_{g,n}^{\ast}$, we have for any $L \in \mathbb{R}_{> 0}^n$:
\[
	\lim_{\beta \rightarrow \infty}  \beta^{(6g - 6 + 2n)(1 - s)} \int_{\mathcal{M}_{g,n}(\beta L)} \mathscr{B}_{g,n}^s\,\dd\mu_{{\rm WP}} = + \infty.
\]
\end{cor}
It would be interesting to obtain an asymptotic equivalent of this integral for all values of $s$. When $s < s_{g,n}^{\ast}$, we cannot currently conclude whether there is equality in \eqref{equnugn}. This could be proved by dominated convergence only if one could describe a sufficiently integrable and uniform bound for the Jacobian ${\rm Jac}_{\beta}$ over $\mathcal{T}_{\Sigma}^{{\rm comb}}$. This would require careful estimates in the arguments by which the convergence of the Weil--Petersson Poisson structure to the Kontsevich Poisson structure were proved in \cite{Mon09}, which we do not currently have.

\medskip

For $s = 1$, we already mentioned in \eqref{MVnorm} that:
\[
	\lim_{\beta \rightarrow \infty} \int_{\mathcal{M}_{g,n}(\beta L)} \mathscr{B}_{g,n}\dd\mu_{{\rm WP}} = \int_{\mathcal{M}_{g,n}(L)} \mathscr{B}_{g,n}\dd\mu_{{\rm WP}}= \int_{\mathcal{M}_{g,n}^{{\rm comb}}(L)} \mathscr{B}_{g,n}^{{\rm comb}} \dd \mu_{{\rm K}}.
\]
which is shown in \cite{WKarticle} by a direct evaluation of the integrals. It would be more satisfactory if the equality could be proved using the convergence property stated in \eqref{conver}.

\medskip

In Appendix~\ref{appB}, we discuss various discretisations of $\int_{\mathcal{M}_{g,n}^{{\rm comb}}(L)} (\mathscr{B}_{g,n}^{{\rm comb}})^s \dd \mu_{{\rm K}}$ which can be naturally defined using the piecewise-linear integral structures on ${\rm MF}_{\Sigma}$ and on $\mathcal{T}_{\Sigma}^{{\rm comb}}$. They lead to interesting arithmetic questions and give another possible way to study the behaviour of multicurve counting on surfaces with large boundaries.

\subsection{Organisation of the paper}

The paper is organised as follows. In Subsection \ref{sec:tcomb} we recall definitions and facts about the combinatorial Teichm\"uller space $\mathcal{T}^{\textup{comb}}_{\Sigma}$, as well as recall the definition of the volume of the unit ball of measured foliations through the statistics of length of multicurves. Subsection \ref{sec:param:mf} shows how the combinatorial structures in $\mathcal{T}^{\textup{comb}}_{\Sigma}$ allows the parametrisation of the set of measured foliations ${\rm MF}_{\Sigma}$ and makes explicit the polytopal structure of the latter. Building on this parametrisation of ${\rm MF}_{\Sigma}$, Subsection \ref{sec:explicit:bcomb} is dedicated to the explicit description of the volume of the unit ball $\mathscr{B}^{\textup{comb}}_{\Sigma}$ in terms of rational functions. This is the content of Proposition \ref{prop:volume:rational:fnct}. As a direct application of the proposition, and as a preliminary result for the rest of the paper, the integrability of $\left(\mathscr{B}^{\textup{comb}}_{1,1}\right)^s$ is then extensively studied in Subsection \ref{sec:torus}. 

Section \ref{sec:integrability} is dedicated to the proof of the main result of the paper --- Theorem \ref{main:thm}. As a preliminary study, we start with Subsection \ref{sec:cells:tcomb} by giving a precise characterisation of the vertices of the cells of the combinatorial Teichm\"uller space. Then, in Subsection \ref{sec:main:result}, we state the propositions that lead to the main result: Proposition \ref{p:integrability:formula} turns the study of integrability of $\left(\mathscr{B}^{\textup{comb}}_{1,1}\right)^s$ into a local integrability result; and Proposition \ref{p:worse:divergent:subgraph} identifies the range of integrability as $g$ and $n$ vary. Those propositions are proved in Subsections \ref{sec:local:integrability} and \ref{sec:worst:divergence} respectively.

The paper is supplemented with three appendices: the theorem of Appendix \ref{app:integrability:lemma} is used in the course of the proof of Proposition \ref{p:integrability:formula} in subsection \ref{sec:local:integrability}; Appendix \ref{app:discrete:integration} deals with the discrete approach of the integrability, coming from the integral structure of $\mathcal{T}^{\textup{comb}}_{\Sigma}$; Appendix \ref{sec:notation} contains an index of notation.

\subsubsection*{Funding}

This work benefited from the support of the Max-Planck-Gesellschaft. It has been supported in part by the ERC-SyG project, ``Recursive and Exact New Quantum Theory'' (ReNewQuantum) which received funding from the European Research Council (ERC) under the European Union's Horizon 2020 research and innovation programme under grant agreement No 810573. It has also received funding from the European Research Council (ERC) under the European Union's Horizon 2020 research and innovation programme (grant agreement No. ERC-2016-STG 716083 ``CombiTop").

\subsubsection*{Acknowledgments}

We thank Don Zagier for a remark on the apparition of truncations of $\zeta(2)$ in relation with dilogarithms, Federico Zerbini for pointing out an error in the computations in Appendix \ref{app:discrete:integration} and useful conversations about polylogarithms, and Michael Borinsky and Francis Brown for bringing to our attention the general work of Berkesch, Forsg\aa{}rd and Passare concerning integrability thresholds (Appendix~\ref{app:integrability:lemma}).

\section{Counting multicurves}
\label{sec:counting:multicurves}

\subsection{Combinatorial geometry background}
\label{sec:tcomb}

Let us recall some facts about the combinatorial moduli space and the combinatorial Teichm\"uller space (we refer to \cite{WKarticle} for further readings).

\paragraph{The combinatorial moduli space.}
A \emph{ribbon graph} is a finite graph $G$ together with a cyclic order of the edges at each vertex.  Replacing edges by oriented closed ribbons and glueing them at each vertex according to the cyclic order, we obtain a topological, oriented, compact surface $|G|$, called the geometric realisation of $G$, with the graph embedded into it and onto which the surface retracts. The $n$ boundary components of $|G|$ are also called \emph{faces}, and we always assume they are labelled as $\partial_1 G,\ldots,\partial_n G$. We denote by $V_G, E_G, F_G$ the set of vertices, edges and faces respectively. For connected ribbon graphs, we define the genus $g \ge 0$ of the ribbon graph to be the genus of $|G|$, and so $\#V_G - \#E_G + \#F_G = 2 - 2g$. The datum $(g,n)$ is called the type of $G$. A ribbon graph is \emph{reduced} if its vertices have valency $\geq 3$. We denote by $\mathcal{R}_{g,n}$ the set of reduced and connected ribbon graphs of type $(g,n)$, and by $\mathcal{R}_{g,n}^{\textup{triv}}$ its subset consisting of trivalent ribbon graphs only. For $2g - 2 + n > 0$, these sets are non-empty and finite. Non-reduced or non-connected ribbon graphs will only appear in Sections~\ref{sec:main:result}-\ref{sec:worst:divergence}.

\medskip

A \emph{metric ribbon graph} $\bm{G}$ is the data of a ribbon graph $G$, together with the assignment of a positive real number for each edge, that is $\ell_{\bm{G}} \in \RR_{+}^{E_{G}}$. Notice that, for a metric ribbon graph $\bm{G}$ of genus $g$ with $n$ faces and any non-trivial edgepath $\gamma$, we can define its length  $\ell_{\bm{G}}(\gamma) \in \RR_{+}$ as the sum of the length of edges (with multiplicity) which $\gamma$ travels along. In particular, we can talk about length of the boundary components $\ell_{\bm{G}}(\de_i G)$ of the ribbon graph, and for a fixed $L \in \RR_{+}^n$ we define the polytope
\begin{equation}\label{eqn:cell:moduli}
	\mathfrak{Z}_G(L) = \Set{ \ell \in \RR_{+}^{E_{G}} | \ell_{\bm{G}}(\de_i G) = L_i} \subset \RR_{+}^{E_{G}}.
\end{equation}
It has dimension $\#E_{G} - n$. The automorphism group of $G$ is acting on $\mathfrak{Z}_G(L)$, and we define the \emph{moduli space} of metric ribbon graphs as
\begin{equation}\label{eqn:comb:moduli}
	\mathcal{M}_{g,n}^{\textup{comb}}(L) = \bigcup_{G \in \mathcal{R}_{g,n}} \frac{\mathfrak{Z}_G(L)}{\Aut(G)},
\end{equation}
where the orbicells $\mathfrak{Z}_G(L)/\Aut(G)$ are glued together through degeneration of edges. This endows $\mathcal{M}_{g,n}^{\textup{comb}}(L)$ with the structure of a polytopal orbicomplex of dimension $6g - 6 + 2n$, parametrising metric ribbon graphs of genus $g$ with $n$ boundary components of length $L \in \RR_{+}^n$. Note that the top-dimensional cells correspond to trivalent ribbon graphs.

\paragraph{The combinatorial Teichm\"uller space.}
Fix now a smooth connected oriented stable surface $\Sigma$ of genus $g \geq 0$ with $n > 0$ labelled boundaries, denoted $\de_1 \Sigma, \dots, \de_n \Sigma$. An \emph{embedded ribbon graph} on $\Sigma$ is the data $[G,f]$ of an isotopy class of proper embedding $f \colon G \hookrightarrow \Sigma$ of a ribbon graph $G$ in $\Sigma$ onto which $\Sigma$ retracts, respecting the labelling of the boundary components. As a consequence of the retraction condition, $G$ has the same genus and number of boundary components as $\Sigma$. We denote by $\mathcal{ER}_{\Sigma}$ the set of embedded ribbon graphs on $\Sigma$. The pure mapping class group of $\Sigma$ acts on $\mathcal{ER}_{\Sigma}$, and the quotient $\mathcal{ER}_{\Sigma}/\Mod_{\Sigma}^{\de}$ is in natural bijection with $\mathcal{R}_{g,n}$.

\medskip

An \emph{embedded metric ribbon graph} $\GG$ on $\Sigma$ is the data $[G,f]$ of an embedded ribbon graph on $\Sigma$, together with the assignment of a positive real number for each edge: $\ell_{\mathbb{G}} \in \RR_{+}^{E_{G}}$. The polytopes
\begin{equation}\label{eqn:cell:Teich}
	\mathfrak{Z}_{G}(L) = \Set{ \ell \in \RR_{+}^{E_{G}} | \ell_{\mathbb{G}}(\de_i G) = L_i} \subset \RR_{+}^{E_{G}}
\end{equation}
parametrise metrics on $[G,f]$ with boundary perimeters $L \in \RR_{+}^n$, and we define the combinatorial Teichmüller space of $\Sigma$ as
\begin{equation}\label{eqn:comb:Teich}
	\mathcal{T}_{\Sigma}^{\textup{comb}}(L) = \bigcup_{[G,f] \in \mathcal{ER}_{\Sigma}} \mathfrak{Z}_{G}(L),
\end{equation}
where the cells are glued together through degeneration of embedded edges. This endows $\mathcal{T}_{\Sigma}^{\textup{comb}}(L)$ with the structure of a polytopal complex of dimension $6g - 6 + 2n$, parametrising embedded metric ribbon graphs on $\Sigma$ with boundary components of lengths $L \in \RR_{+}^n$. The pure mapping class group of $\Sigma$ acts on $\mathcal{T}_{\Sigma}^{\textup{comb}}(L)$, and we have a natural isomorphism $\mathcal{T}_{\Sigma}^{\textup{comb}}(L)/\Mod_{\Sigma}^{\de} \cong \mathcal{M}_{g,n}^{\textup{comb}}(L)$. 

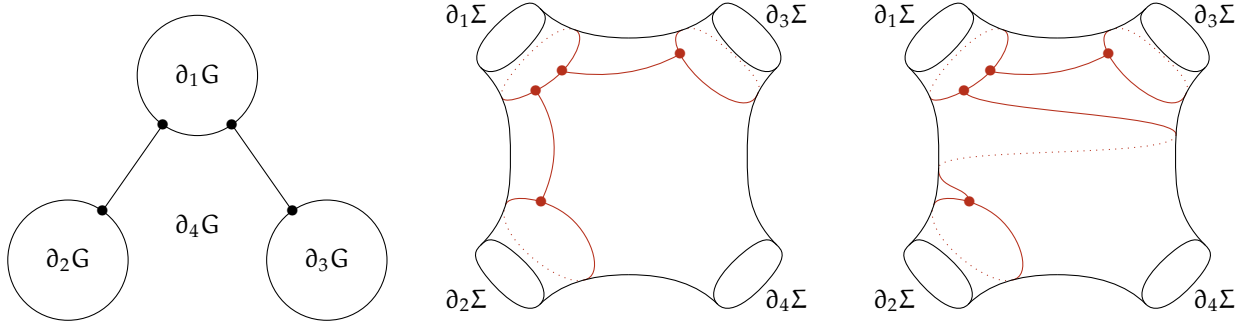
\begin{figure}
	\begin{subfigure}[t]{.31\textwidth}
	\centering
		\begin{tikzpicture}
			\draw(0,0) circle [radius=.8cm];
			\draw(-55:3) circle [radius=.8cm];
			\draw(-125:3) circle [radius=.8cm];
			\node at (-55:.8) {$\bullet$};
			\node at (-125:.8) {$\bullet$};
			\node at (-55:2.2) {$\bullet$};
			\node at (-125:2.2) {$\bullet$};
			\draw (-55:.8) -- (-55:2.2);
			\draw (-125:.8) -- (-125:2.2);

			\node at (0,0) {$\de_1 G$};
			\node at (-125:3) {$\de_2 G$};
			\node at (-55:3) {$\de_3 G$};
			\node at (-90:2) {$\de_4 G$};
		\end{tikzpicture}
	\end{subfigure}
	\hfill
	\begin{subfigure}[t]{.31\textwidth}
	\centering
		\begin{tikzpicture}[x=1pt,y=1pt,scale=.35]
			\draw[BrickRed, dotted](86.6889, 671.614) .. controls (83.0254, 679.054) and (97.5127, 699.527) .. (114.9354, 717.3408) .. controls (132.358, 735.1545) and (152.716, 750.309) .. (165.163, 744.984);
			\draw[BrickRed, dotted](363.59, 676.431) .. controls (369.769, 686.905) and (352.8845, 703.4525) .. (334.9815, 719.6782) .. controls (317.0785, 735.904) and (298.157, 751.808) .. (282.103, 744.722);
			\draw[BrickRed, dotted](175.606, 474.281) .. controls (162.369, 469.97) and (137.1845, 482.985) .. (117.599, 500.1985) .. controls (98.0134, 517.412) and (84.0268, 538.824) .. (90.6797, 555.212);
			\draw[BrickRed](86.6889, 671.614) .. controls (90.022, 662.879) and (125.011, 675.4395) .. (146.9318, 695.0133) .. controls (168.8525, 714.587) and (177.705, 741.174) .. (165.163, 744.984);
			\draw[BrickRed](90.6797, 555.212) .. controls (93.7038, 568.108) and (134.8519, 564.054) .. (160.611, 541.6195) .. controls (186.37, 519.185) and (196.74, 478.37) .. (175.606, 474.281);
			\draw[BrickRed](282.103, 744.722) .. controls (266.254, 740.218) and (277.127, 714.109) .. (299.4403, 692.54) .. controls (321.7535, 670.971) and (355.507, 653.942) .. (363.59, 676.431);
			\draw[BrickRed](129.153, 558.817) .. controls (149.7177, 602.2723) and (147.8583, 642.3233) .. (123.575, 678.97);
			\draw[BrickRed](151.928, 699.82) .. controls (199.976, 691.94) and (242.4393, 698.1903) .. (279.318, 718.571);
			\node[BrickRed] at (151.928, 699.82) {$\bullet$};
			\node[BrickRed] at (123.575, 678.97) {$\bullet$};
			\node[BrickRed] at (129.153, 558.817) {$\bullet$};
			\node[BrickRed] at (279.318, 718.571) {$\bullet$};
			\draw(64, 704) .. controls (48, 720) and (80, 752) .. (98.6667, 765.3333) .. controls (117.3333, 778.6667) and (122.6667, 773.3333) .. (128, 768) .. controls (133.3333, 762.6667) and (138.6667, 757.3333) .. (125.3333, 738.6667) .. controls (112, 720) and (80, 688) .. (64, 704);
			\draw(320, 448) .. controls (336, 432) and (368, 464) .. (381.3333, 482.6667) .. controls (394.6667, 501.3333) and (389.3333, 506.6667) .. (384, 512) .. controls (378.6667, 517.3333) and (373.3333, 522.6667) .. (354.6667, 509.3333) .. controls (336, 496) and (304, 464) .. (320, 448);
			\draw(128, 448) .. controls (112, 432) and (80, 464) .. (66.6667, 482.6667) .. controls (53.3333, 501.3333) and (58.6667, 506.6667) .. (64, 512) .. controls (69.3333, 517.3333) and (74.6667, 522.6667) .. (93.3333, 509.3333) .. controls (112, 496) and (144, 464) .. (128, 448);
			\draw(384, 704) .. controls (368, 688) and (336, 720) .. (322.6667, 738.6667) .. controls (309.3333, 757.3333) and (314.6667, 762.6667) .. (320, 768) .. controls (325.3333, 773.3333) and (330.6667, 778.6667) .. (349.3333, 765.3333) .. controls (368, 752) and (400, 720) .. (384, 704);
			\draw(384, 704) .. controls (352, 672) and (352, 640) .. (352, 608) .. controls (352, 576) and (352, 544) .. (384, 512);
			\draw(64, 512) .. controls (96, 544) and (96, 576) .. (96, 608) .. controls (96, 640) and (96, 672) .. (64, 704);
			\draw(128, 768) .. controls (149.3333, 746.6667) and (181.3333, 736) .. (224, 736);
			\draw(320, 768) .. controls (298.6667, 746.6667) and (266.6667, 736) .. (224, 736);
			\draw(320, 448) .. controls (298.6667, 469.3333) and (266.6667, 480) .. (224, 480);
			\draw(128, 448) .. controls (149.3333, 469.3333) and (181.3333, 480) .. (224, 480);

			\node at (48, 760) {$\de_1 \Sigma$};
			\node at (48, 452) {$\de_2 \Sigma$};
			\node at (394, 760) {$\de_3 \Sigma$};
			\node at (394, 452) {$\de_4 \Sigma$};
		\end{tikzpicture}
		\end{subfigure}
		\hfill
		\begin{subfigure}[t]{.31\textwidth}
		\centering
		\begin{tikzpicture}[x=1pt,y=1pt,scale=.35]
			\draw[BrickRed, dotted](352.455, 631.265) .. controls (352.001, 605.081) and (95.9407, 619.792) .. (95.932, 595.655);
			\draw[BrickRed, dotted](86.6889, 671.614) .. controls (83.0254, 679.054) and (97.5127, 699.527) .. (114.9354, 717.3408) .. controls (132.358, 735.1545) and (152.716, 750.309) .. (165.163, 744.984);
			\draw[BrickRed, dotted](363.59, 676.431) .. controls (369.769, 686.905) and (352.8845, 703.4525) .. (334.9815, 719.6782) .. controls (317.0785, 735.904) and (298.157, 751.808) .. (282.103, 744.722);
			\draw[BrickRed, dotted](175.606, 474.281) .. controls (162.369, 469.97) and (137.1845, 482.985) .. (117.599, 500.1985) .. controls (98.0134, 517.412) and (84.0268, 538.824) .. (90.6797, 555.212);
			\draw[BrickRed](86.6889, 671.614) .. controls (90.022, 662.879) and (125.011, 675.4395) .. (146.9318, 695.0133) .. controls (168.8525, 714.587) and (177.705, 741.174) .. (165.163, 744.984);
			\draw[BrickRed](90.6797, 555.212) .. controls (93.7038, 568.108) and (134.8519, 564.054) .. (160.611, 541.6195) .. controls (186.37, 519.185) and (196.74, 478.37) .. (175.606, 474.281);
			\draw[BrickRed](282.103, 744.722) .. controls (266.254, 740.218) and (277.127, 714.109) .. (299.4403, 692.54) .. controls (321.7535, 670.971) and (355.507, 653.942) .. (363.59, 676.431);
			\draw[BrickRed](151.928, 699.82) .. controls (199.976, 691.94) and (242.4393, 698.1903) .. (279.318, 718.571);
			\draw[BrickRed](123.575, 678.97) .. controls (128, 656) and (355.653, 654.571) .. (352.455, 631.265);
			\draw[BrickRed](95.932, 595.655) .. controls (94.2791, 571.764) and (128, 576) .. (129.153, 558.817);
			\node[BrickRed] at (151.928, 699.82) {$\bullet$};
			\node[BrickRed] at (123.575, 678.97) {$\bullet$};
			\node[BrickRed] at (129.153, 558.817) {$\bullet$};
			\node[BrickRed] at (279.318, 718.571) {$\bullet$};
			\draw(64, 704) .. controls (48, 720) and (80, 752) .. (98.6667, 765.3333) .. controls (117.3333, 778.6667) and (122.6667, 773.3333) .. (128, 768) .. controls (133.3333, 762.6667) and (138.6667, 757.3333) .. (125.3333, 738.6667) .. controls (112, 720) and (80, 688) .. (64, 704);
			\draw(320, 448) .. controls (336, 432) and (368, 464) .. (381.3333, 482.6667) .. controls (394.6667, 501.3333) and (389.3333, 506.6667) .. (384, 512) .. controls (378.6667, 517.3333) and (373.3333, 522.6667) .. (354.6667, 509.3333) .. controls (336, 496) and (304, 464) .. (320, 448);
			\draw(128, 448) .. controls (112, 432) and (80, 464) .. (66.6667, 482.6667) .. controls (53.3333, 501.3333) and (58.6667, 506.6667) .. (64, 512) .. controls (69.3333, 517.3333) and (74.6667, 522.6667) .. (93.3333, 509.3333) .. controls (112, 496) and (144, 464) .. (128, 448);
			\draw(384, 704) .. controls (368, 688) and (336, 720) .. (322.6667, 738.6667) .. controls (309.3333, 757.3333) and (314.6667, 762.6667) .. (320, 768) .. controls (325.3333, 773.3333) and (330.6667, 778.6667) .. (349.3333, 765.3333) .. controls (368, 752) and (400, 720) .. (384, 704);
			\draw(384, 704) .. controls (352, 672) and (352, 640) .. (352, 608) .. controls (352, 576) and (352, 544) .. (384, 512);
			\draw(64, 512) .. controls (96, 544) and (96, 576) .. (96, 608) .. controls (96, 640) and (96, 672) .. (64, 704);
			\draw(128, 768) .. controls (149.3333, 746.6667) and (181.3333, 736) .. (224, 736);
			\draw(320, 768) .. controls (298.6667, 746.6667) and (266.6667, 736) .. (224, 736);
			\draw(320, 448) .. controls (298.6667, 469.3333) and (266.6667, 480) .. (224, 480);
			\draw(128, 448) .. controls (149.3333, 469.3333) and (181.3333, 480) .. (224, 480);

			\node at (48, 760) {$\de_1 \Sigma$};
			\node at (48, 452) {$\de_2 \Sigma$};
			\node at (394, 760) {$\de_3 \Sigma$};
			\node at (394, 452) {$\de_4 \Sigma$};
		\end{tikzpicture}
		\end{subfigure}
	\caption{A ribbon graph $G$ of type $(0,4)$, and two embedded ribbon graphs $[G,f]$ and $[G,f']$ on a sphere with $4$ boundary components $\Sigma$, with the same underlying graph $G$ but different embeddings.}
\end{figure}

\paragraph{Integrating functions.}

In \cite{Kontsevich} Kontsevich defined a $2$-form $\omega_{\textup{K}}$ on the moduli space $\mathcal{M}_{g,n}^{\textup{comb}}(L)$ that is symplectic on the top-dimensional stratum. The associated symplectic volume form defines a measure $\mu_{\textup{K}}$ on $\mathcal{M}_{g,n}^{\textup{comb}}(L)$. In particular, for every measurable function $f \colon \mathcal{M}_{g,n}^{\textup{comb}}(L) \to \RR$, we can consider its integral against the Kontsevich measure, defined as
\begin{equation}\label{eqn:integral:functions}
	\int_{\mathcal{M}_{g,n}^{\textup{comb}}(L)} f \, \dd\mu_{\textup{K}}
	=
	\sum_{G \in \mathcal{R}_{g,n}^{\textup{triv}}} \frac{1}{\#\Aut(G)} \int_{\mathfrak{Z}_{G}(L)} f \, \dd\mu_{\textup{K}}.
\end{equation}
Here, by abuse of notation, we denoted with the same symbols objects on the orbicells $\mathfrak{Z}_{G}(L)/\Aut(G)$ and on the unfolded cells $\mathfrak{Z}_{G}(L)$.

\paragraph{Combinatorial length of curves.}

If $\mathbb{G} \in \mathcal{T}_{\Sigma}^{{\rm comb}}(L)$, the homotopy class $\gamma$ of a simple closed curve admits a unique non-backtracking edgepath representative on the embedded graph underlying $\mathbb{G}$, and we can define the length $\ell_{\mathbb{G}}(\gamma)$ as the length of this representative. $\mathcal{T}_{\Sigma}^{{\rm comb}}$ can also be described in terms of measured foliations transverse to $\partial \Sigma$, and this notion of length coincides with the intersection number of $\gamma$ with the measured foliation associated to $\mathbb{G}$. More generally, we can talk about the length with respect to $\GG$ of any multicurve $c \in M_{\Sigma}$ by adding lengths of the components of $c$. We can then introduce the function:
\begin{equation}\label{eqn:Bcomb}
	\mathscr{B}_{\Sigma}^{\textup{comb}}(\GG)
	=
	\lim_{r \to \infty} \frac{\#\Set{c \in M_{\Sigma} | \ell_{\GG}(c) \le r}}{r^{6g-6+2n}}.
\end{equation} 
Its basic properties have been studied in \cite{WKarticle}.

\begin{prop} \label{prop:Bcomb:properties} \cite{WKarticle}
	For any $L \in \mathbb{R}_{+}^n$, the function $\mathscr{B}_{\Sigma}^{\textup{comb}}$ takes values in $\RR_{+}$, is continuous on $\mathcal{T}_{\Sigma}^{\textup{comb}}(L)$, and the induced function $\mathscr{B}_{g,n}^{\textup{comb}}$ on $\mathcal{M}_{g,n}^{\textup{comb}}(L)$ is integrable with respect to $\mu_{\textup{K}}$.
\end{prop}

\subsection{Parametrisation of measured foliations}
\label{sec:param:mf}

In this paragraph, we shall describe a parametrisation of the space of measured foliations ${\rm MF}_{\Sigma}$ that depends on a chosen embedded ribbon graph $[G,f]$. It is dual to the parametrisation of \cite{Mosher2} --- which considers triangulations instead of ribbon graphs. This will allow us to effectively describe the function $\mathscr{B}_{g,n}^{\textup{comb}}$ on the orbicell of the moduli space $\mathcal{M}_{g,n}^{\textup{comb}}(L)$ determined by the ribbon graph $G$.

\medskip

In what follows, it is useful to introduce a larger space ${\rm MF}_{\Sigma}^{\bullet}$ of measured foliations, where now $\de \Sigma$ can be a union of smooth and singular leaves (and we still include the empty foliation). It is a piecewise linear manifold of dimension $6g - 6 + 3n$, with a piecewise integral structure whose integral points are the multicurves $M_{\Sigma}^{\bullet}$ on $\Sigma$ where the components are allowed to be homotopic to boundary components. In particular, we can consider the associated Thurston measure $\mu_{\textup{Th}}^{\bullet}$ by lattice point count, and the function
\begin{equation}
	\mathscr{B}_{\Sigma}^{\textup{comb},\bullet}(\GG)
	=
	\mu_{\textup{Th}}^{\bullet}\big(\Set{\mathcal{F} \in {\rm MF}_{\Sigma}^{\bullet} | \ell_{\GG}(\mathcal{F}) \leq 1}\big),
	\qquad
	\GG \in \mathcal{T}_{\Sigma}^{\textup{comb}}(L).
\end{equation}
As usual \cite{FLP12} the function $\ell_{\GG}$ is continuously extended from multicurves to measured foliations.

\medskip

We have a homeomorphism
\begin{equation}\label{eqn:homeo:MF:MFbullet}
	\Phi \colon {\rm MF}_{\Sigma} \times \RR_{\ge 0}^n \xrightarrow{\;\;\cong\;\;} {\rm MF}_{\Sigma}^{\bullet}
\end{equation}
which assign to a measured foliation $\mathcal{F}$ and a tuple $(x_1,\dots,x_n)$ the foliation $\mathcal{F}^{\bullet}$ obtained by adding a cylinder of boundary-parallel leaves around $\de_i \Sigma$ of total height $x_i$. The map $\Phi$ also respects the piecewise linear structure: $\Phi(M_{\Sigma} \times \ZZ_{\ge 0}^n) = M_{\Sigma}^{\bullet}$. Thus, it respects the measures, when ${\rm MF}_{\Sigma}^{\bullet}$ and ${\rm MF}_{\Sigma}$ are equipped with their respective Thurston measures and $\RR_{\geq 0}^n$ with the Lebesgue measure. We also notice that ${\rm MF}_{\Sigma}$ and $\RR_{\ge 0}^n$ naturally sit inside ${\rm MF}_{\Sigma}^{\bullet}$ as $\Phi(\cdot,0)$ and $\Phi(\varnothing,\cdot)$ respectively.

\medskip

There is an elementary relation between the enumeration of multicurves with or without components homotopic to boundaries. 

\begin{lem}\label{lem:MF:MFbullet}
	For any $\GG \in \mathcal{T}_{\Sigma}^{\textup{comb}}(L)$, we have
	\begin{equation}
		\mathscr{B}_{\Sigma}^{\textup{comb},\bullet}(\GG)
		=
		\frac{(6g - 6 + 2n)!}{(6g - 6 + 3n)!} \, \frac{\mathscr{B}_{\Sigma}^{\textup{comb}}(\GG)}{\prod_{i = 1}^n L_i}.
	\end{equation}
\end{lem}

\begin{proof}
	Since $\ell_{\GG}$ is homogeneous and additive under disjoint union of multicurves, we have
	\[
		\forall (\mathcal{F},x) \in {\rm MF}_{\Sigma} \times \RR_{\geq 0}^n,
		\quad
		\ell_{\GG}(\mathcal{F}) + \ell_{\GG}(x) = \ell_{\GG}(\Phi(\mathcal{F},x)),
		\qquad
		\text{with }\ell_{\GG}(x) = \sum_{i = 1}^n x_i L_i.
	\]
	Therefore, using homogeneity of the Thurston and Lebesgue measure, we find
	\[
	\begin{split}
		\mathscr{B}_{\Sigma}^{\textup{comb},\bullet}(\GG)
		& =
		\int_{0}^{1} \dd t \,
			\mu_{\textup{Th}}\big(\Set{ \mathcal{F} | \ell_{\GG}(\mathcal{F}) \leq t }\big)
			\cdot
			\mu_{\textup{Leb}}\big(\Set{ x | \ell_{\GG}(x) \leq 1 - t }\big) \\
		& =
		\bigg( \int_{0}^{1} \dd t \, t^{6g - 6 + 2n} (1 - t)^n \bigg)
		\cdot
		\mu_{\textup{Th}}\big(\Set{ \mathcal{F} | \ell_{\GG}(\mathcal{F}) \leq 1 }\big)
		\cdot
		\mu_{\textup{Leb}}\big(\Set{ x | \ell_{\GG}(x) \leq 1 }\big) \\
		& =
		\frac{n! (6g - 6 + 2n)!}{(6g - 6 + 3n)!} \cdot \mathscr{B}_{\Sigma}^{\textup{comb}}(\GG) \cdot \frac{1}{n! \prod_{i = 1}^n L_i} \\
		& =
		\frac{(6g - 6 + 2n)!}{(6g - 6 + 3n)!} \,
		\frac{\mathscr{B}_{\Sigma}^{\textup{comb}}(\GG)}{\prod_{i = 1}^n L_i}.
	\end{split}
	\]
\end{proof}

\begin{rem}
	The above statement can be generalised to any notion of length as follows. Let $l \colon {M}_{\Sigma}^{\bullet} \to \RR_+$ be a locally convex function, that is additive under disjoint union of multicurves. It uniquely extends to a continuous function on ${\rm MF}_{\Sigma}^{\bullet}$, and it induces a function still denoted $l$ on ${\rm MF}_{\Sigma}$. Furthermore, we have
	\[
		\mu_{\textup{Th}}^{\bullet}\big( \Set{l \leq 1 }\big)
		=
		\frac{(6g - 6 + 2n)!}{(6g - 6 + 3n)!}\,\frac{\mu_{\textup{Th}} \big(\Set{l \leq 1 }\big)}{\prod_{i = 1}^n l(\partial_i\Sigma)}. 
	\]
\end{rem}

Fix now an embedded ribbon graph $[G,f]$ in $\Sigma$. Each edge $e$ of the embedded graph $G \hookrightarrow \Sigma$ is dual to a unique --- up to homotopy of proper embeddings\footnote{
	If $X$ and $Y$ are topological manifolds with boundaries, a continuous map $f \colon X \rightarrow Y$ is called a proper embedding if $f^{-1}(\partial Y) = \partial X$ and we use the natural notion of homotopies among such.
} 
-- arc $\alpha_e$ between two (possibly the same) boundaries of $\Sigma$, and these arcs are pairwise disjoint. To a measured foliation, we associate the set of intersection numbers\footnote{
	We recall that the intersection number is defined as follows (\emph{cf.} \cite[Section~5.3]{FLP12}). For a fixed isotopy class of measured foliation $\mathcal{F}$ in $\Sigma$, and an arc $a$ in $\Sigma$ between two boundary components (or a simple closed curve), we have the notion of measure of $a$:
	\[
		\mu_{\mathcal{F}}(a) = \sup\bigg( \sum_{j=1}^k \mu_{\mathcal{F}}(a_j)\bigg),
	\]
	where $a_1,\dots,a_k$ are arcs of $a$, mutually disjoint and transverse to $\mathcal{F}$, and where the sup is taken over all sums of this type. If $\alpha$ is now a homotopy class of arc in $\Sigma$ between two boundary components (or a homotopy class of simple closed curve), we set
	\[
		\iota(\mathcal{F},\alpha) = \inf_{a \in \alpha} \mu_{\mathcal{F}}(a),
	\]
	where the inf is taken over representatives of $\alpha$. Such quantity is invariant under isotopy of $\mathcal{F}$ and Whitehead moves.
} 
with these arcs
\begin{equation}\label{eqn:int:numb:map}
	\mathfrak{m}_{[G,f]} \colon \;
	\begin{aligned}
		{\rm MF}_{\Sigma}^{\bullet} & \longrightarrow \RR_{\geq 0}^{E_G} \\
		\mathcal{F} \quad & \longmapsto \bigl( \iota(\mathcal{F},\alpha_e) \bigr)_{e \in E_G}
	\end{aligned}.
\end{equation}
By definition, $\mathfrak{m}_{[G,f]}$ preserves the piecewise linear integral structures of ${\rm MF}_\Sigma^{\bullet}$ and $\RR^{E_G}_{\geq 0}$. 

\medskip

The map $\mathfrak{m}_{[G,f]}$ gives a description of ${\rm MF}_\Sigma^{\bullet}$ and ${\rm MF}_\Sigma$. We will show that it in fact gives a parametrisation of ${\rm MF}_\Sigma^{\bullet}$ and ${\rm MF}_\Sigma$, after we introduce notations to describe the image.

\begin{defn}
	Let $G$ be a ribbon graph. A \emph{simple loop} is a non-empty, closed, non-backtracking edgepath on $G$ that does not pass twice through the same edge. A \emph{dumbbell} is a closed, non-backtracking edgepath $\gamma$ on $G$ that passes at most twice through each edge and such that the union of edges that are visited twice forms a non-empty edgepath $p$ for which we have a decomposition $\gamma = \gamma_1 \!\cdot\! p \!\cdot\! \gamma_2 \!\cdot\! p^{-1}$, where $\gamma_1$ and $\gamma_2$ are simple loops. A simple loop or a dumbbell is called \emph{essential} if it does not coincide with a boundary component of the ribbon graph $G$.

	\medskip

	If $[G,f]$ is an embedded ribbon graph in $\Sigma$, we call (essential) simple loop or dumbbell of $[G,f]$ the homotopy class of the image of any (essential) simple loop or dumbbell of $G$ via $f$. (See Figure~\ref{fig:simple:loops:dumbbells} for an example.)
\end{defn}

\begin{defn}
	A \emph{corner} in a trivalent ribbon graph $G$ is an ordered triple $\Delta = (e,e',e'')$ where $e,e',e''$ are edges incident to a vertex in the cyclic order. Equivalently, a corner consists of a vertex $v$ together with the choice of an incident edge $e$. We say that a corner belongs to a face $\mathfrak{f} \in F_G$ if $e'$ and $e''$ are edges around that face. We denote ${\rm C}(\mathfrak{f})$ the set of corners belonging to $\mathfrak{f}$ and ${\rm C}_{G}$ the set of all corners of $G$. If we have an assignment of real numbers $(x_e)_{e \in E_G}$ and $\Delta = (e,e',e'')$ is a corner, we denote $x_{\Delta} = x_{e'} + x_{e''} - x_{e}$.
\end{defn}

\begin{lem}\label{lem:MF:parametrisation}
	Fix an embedded ribbon graph $[G,f]$ in $\Sigma$, with $G$ trivalent. The map $\mathfrak{m}_{[G,f]}$ is a homeomorphism onto its image, which is the convex polyhedral cone
	\begin{equation}\label{eqn:Z:G:bullet}
		Z_{G}^{\bullet} = \Set{ x \in \RR_{\geq 0}^{E_G} | \forall \Delta \in {\rm C}_G \quad x_{\Delta} \geq 0}.
	\end{equation}
	The image of ${\rm MF}_\Sigma$, denoted $Z_{G}$, is the union ranging over the set $\mathfrak{D}_{G} = \Set{ \Delta \colon F_{G} \rightarrow {\rm C}_{G} | \Delta(\mathfrak{f}) \in {\rm C}(\mathfrak{f})}$ of the convex polyhedral cones
	\begin{equation}\label{eqn:Z:G:fan}
		Z_{G,\Delta}
		=
		\Set{x \in Z_{G}^{\bullet} | \forall \mathfrak{f} \in F_{G}\,\quad x_{\Delta(\mathfrak{f})} = 0}.
	\end{equation}
	The rays of the polyhedral cones $Z_{G,\Delta}$ are images of essential simple loops and essential dumbbells. And conversely, the image of an essential simple loop or an essential dumbbell generates a ray which is extremal in any $Z_{G,\Delta}$ it belongs to.
\end{lem}

Lemma~\ref{lem:MF:parametrisation} shows that each trivalent ribbon graph $G$ presents ${\rm MF}_\Sigma$ as a union of polyhedral cone. This structure is finer than the piecewise linear integral structure of ${\rm MF}_\Sigma$.

\begin{rem}
	It is possible to extend Lemma~\ref{lem:MF:parametrisation} to non-trivalent ribbon graphs $G$ (or equivalently not top-dimensional cells of the combinatorial Teichm\"uller space). More precisely, if $G$ is not trivalent, one can obtain a similar description by resolving the non-trivalent vertices of the underlying ribbon graph (in some arbitrary way) into trivalent vertices.
\end{rem}

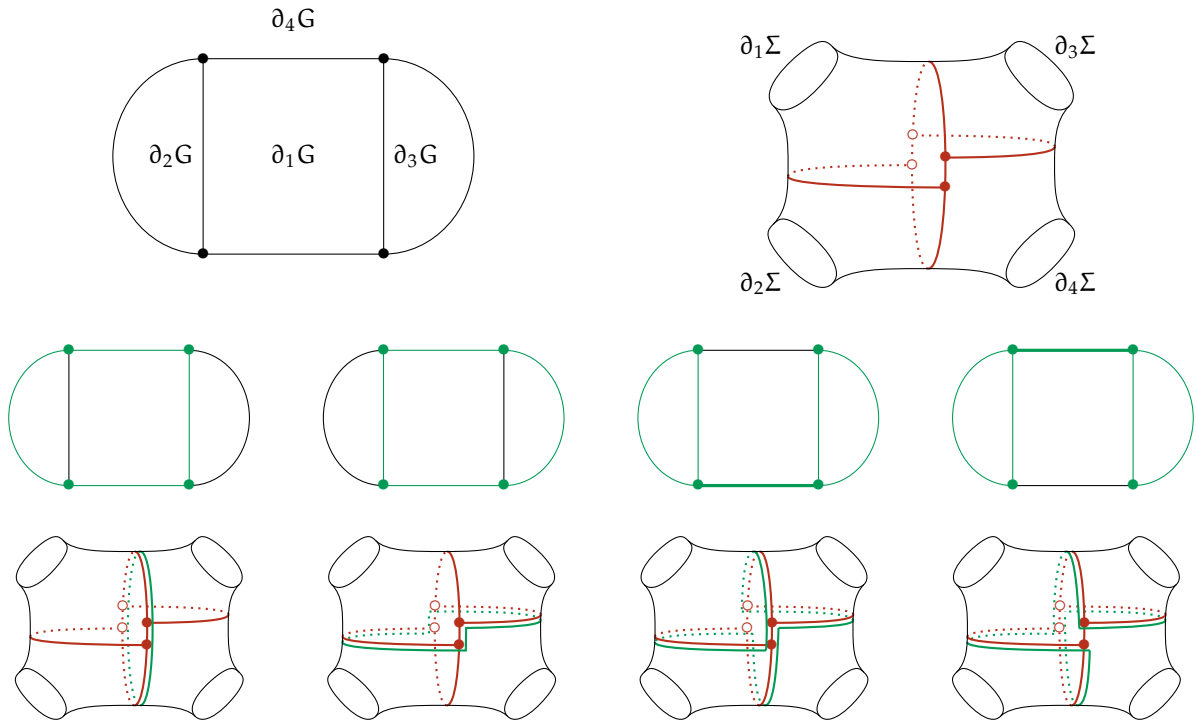
\begin{figure}
\centering
	\hfill
	\begin{subfigure}[t]{.49\textwidth}
	\centering
	\begin{tikzpicture}[xscale=1.2,yscale=1.3]
		\draw (1,1) -- (-1,1);
		\draw (-1,1) -- (-1,-1);
		\draw (-1,-1) -- (1,-1);
		\draw (1,-1) -- (1,1);

		\draw (1,1) arc (90:-90:1);
		\draw (-1,1) arc (90:270:1);

		\node at (1,1) {$\bullet$};
		\node at (-1,1) {$\bullet$};
		\node at (-1,-1) {$\bullet$};
		\node at (1,-1) {$\bullet$};

		\node at (0,0) {$\de_1 G$};
		\node at (-1,0) [left] {$\de_2 G$};
		\node at (1,0) [right] {$\de_3 G$};
		\node at (0,1.4) {$\de_4 G$};

		\node at (0,-1.4) {$\phantom{.}$};
	\end{tikzpicture}
	\end{subfigure}
	\hfill
	\begin{subfigure}[t]{.49\textwidth}
	\centering
	\begin{tikzpicture}[x=1pt,y=1pt,scale=.35]
		\draw[thick,BrickRed, dotted](368.353, 660.022) .. controls (369.1197, 667.9907) and (317.9167, 672.0083) .. (214.744, 672.075);
		\draw[thick,BrickRed, dotted](231.339, 528.004) .. controls (208.171, 528.114) and (207.903, 750.949) .. (232, 752);
		\draw[thick,BrickRed, dotted](79.9237, 627.92) .. controls (79.9746, 635.944) and (124.6483, 639.9333) .. (213.945, 639.888);
		\node[white] at (213.9455, 639.8883) {$\bullet$};
		\node[white] at (214.7436, 672.0747) {$\bullet$};
		\node[BrickRed] at (213.9455, 639.8883) {$\circ$};
		\node[BrickRed] at (214.7436, 672.0747) {$\circ$};
		\node[BrickRed] at (249.5961, 615.7038) {$\bullet$};
		\node[BrickRed] at (250.134, 648.1133) {$\bullet$};
		\draw[thick,BrickRed](250.134, 648.113) .. controls (328.746, 648.3937) and (368.1523, 652.3633) .. (368.353, 660.022);
		\draw[thick,BrickRed](231.339, 528.004) .. controls (256.033, 527.091) and (256.467, 752.001) .. (232, 752);
		\draw[thick,BrickRed](249.596, 615.704) .. controls (136.1545, 616.33) and (79.597, 620.402) .. (79.9237, 627.92);
		\draw(64, 704) .. controls (48, 720) and (80, 752) .. (98.6667, 765.3333) .. controls (117.3333, 778.6667) and (122.6667, 773.3333) .. (128, 768) .. controls (133.3333, 762.6667) and (138.6667, 757.3333) .. (125.3333, 738.6667) .. controls (112, 720) and (80, 688) .. (64, 704);
		\draw(384, 704) .. controls (368, 688) and (336, 720) .. (322.6667, 738.6667) .. controls (309.3333, 757.3333) and (314.6667, 762.6667) .. (320, 768) .. controls (325.3333, 773.3333) and (330.6667, 778.6667) .. (349.3333, 765.3333) .. controls (368, 752) and (400, 720) .. (384, 704);
		\draw(128, 512) .. controls (112, 496) and (80, 528) .. (66.6667, 546.6667) .. controls (53.3333, 565.3333) and (58.6667, 570.6667) .. (64, 576) .. controls (69.3333, 581.3333) and (74.6667, 586.6667) .. (93.3333, 573.3333) .. controls (112, 560) and (144, 528) .. (128, 512);
		\draw(320, 512) .. controls (336, 496) and (368, 528) .. (381.3333, 546.6667) .. controls (394.6667, 565.3333) and (389.3333, 570.6667) .. (384, 576) .. controls (378.6667, 581.3333) and (373.3333, 586.6667) .. (354.6667, 573.3333) .. controls (336, 560) and (304, 528) .. (320, 512);
		\draw(384, 576) .. controls (368, 592) and (368, 616) .. (368, 640) .. controls (368, 664) and (368, 688) .. (384, 704);
		\draw(64, 704) .. controls (80, 688) and (80, 664) .. (80, 640) .. controls (80, 616) and (80, 592) .. (64, 576);
		\draw(128, 768) .. controls (144, 752) and (184, 752) .. (224, 752) .. controls (264, 752) and (304, 752) .. (320, 768);
		\draw(128, 512) .. controls (144, 528) and (184, 528) .. (224, 528) .. controls (264, 528) and (304, 528) .. (320, 512);
		\node at (50, 768) {$\partial_1 \Sigma$};
		\node at (390, 768) {$\partial_3 \Sigma$};
		\node at (50, 512) {$\partial_2 \Sigma$};
		\node at (390, 512) {$\partial_4 \Sigma$};
	\end{tikzpicture}
	\end{subfigure}
	\hfill
	\begin{subfigure}[t]{.24\textwidth}
	\centering
	\begin{tikzpicture}[xscale=.8,yscale=.9]
		\draw[ForestGreen] (1,1) -- (-1,1);
		\draw (-1,1) -- (-1,-1);
		\draw[ForestGreen] (-1,-1) -- (1,-1);
		\draw[ForestGreen] (1,-1) -- (1,1);

		\draw (1,1) arc (90:-90:1);
		\draw[ForestGreen] (-1,1) arc (90:270:1);

		\node[ForestGreen] at (1,1) {$\bullet$};
		\node[ForestGreen] at (-1,1) {$\bullet$};
		\node[ForestGreen] at (-1,-1) {$\bullet$};
		\node[ForestGreen] at (1,-1) {$\bullet$};

		\node at (0,-1.5) {$\phantom{.}$};
		\node at (0,1.5) {$\phantom{.}$};
	\end{tikzpicture}
	\end{subfigure}
	\hfill
	\begin{subfigure}[t]{.24\textwidth}
	\centering
	\begin{tikzpicture}[xscale=.8,yscale=.9]
		\draw[ForestGreen] (1,1) -- (-1,1);
		\draw[ForestGreen] (-1,1) -- (-1,-1);
		\draw[ForestGreen] (-1,-1) -- (1,-1);
		\draw (1,-1) -- (1,1);

		\draw[ForestGreen] (1,1) arc (90:-90:1);
		\draw (-1,1) arc (90:270:1);

		\node[ForestGreen] at (1,1) {$\bullet$};
		\node[ForestGreen] at (-1,1) {$\bullet$};
		\node[ForestGreen] at (-1,-1) {$\bullet$};
		\node[ForestGreen] at (1,-1) {$\bullet$};

		\node at (0,-1.5) {$\phantom{.}$};
		\node at (0,1.5) {$\phantom{.}$};
	\end{tikzpicture}
	\end{subfigure}
	\hfill
	\begin{subfigure}[t]{.24\textwidth}
	\centering
	\begin{tikzpicture}[xscale=.8,yscale=.9]
		\draw (1,1) -- (-1,1);
		\draw[ForestGreen] (-1,1) -- (-1,-1);
		\draw[ForestGreen,very thick] (-1,-1) -- (1,-1);
		\draw[ForestGreen] (1,-1) -- (1,1);

		\draw[ForestGreen] (1,1) arc (90:-90:1);
		\draw[ForestGreen] (-1,1) arc (90:270:1);

		\node[ForestGreen] at (1,1) {$\bullet$};
		\node[ForestGreen] at (-1,1) {$\bullet$};
		\node[ForestGreen] at (-1,-1) {$\bullet$};
		\node[ForestGreen] at (1,-1) {$\bullet$};

		\node at (0,-1.5) {$\phantom{.}$};
		\node at (0,1.5) {$\phantom{.}$};
	\end{tikzpicture}
	\end{subfigure}
\hfill
	\begin{subfigure}[t]{.24\textwidth}
	\centering
	\begin{tikzpicture}[xscale=.8,yscale=.9]
		\draw[ForestGreen,very thick] (1,1) -- (-1,1);
		\draw[ForestGreen] (-1,1) -- (-1,-1);
		\draw (-1,-1) -- (1,-1);
		\draw[ForestGreen] (1,-1) -- (1,1);

		\draw[ForestGreen] (1,1) arc (90:-90:1);
		\draw[ForestGreen] (-1,1) arc (90:270:1);

		\node[ForestGreen] at (1,1) {$\bullet$};
		\node[ForestGreen] at (-1,1) {$\bullet$};
		\node[ForestGreen] at (-1,-1) {$\bullet$};
		\node[ForestGreen] at (1,-1) {$\bullet$};

		\node at (0,-1.5) {$\phantom{.}$};
		\node at (0,1.5) {$\phantom{.}$};
	\end{tikzpicture}
	\end{subfigure}
	\hfill
	\begin{subfigure}[t]{.24\textwidth}
	\centering
	\begin{tikzpicture}[x=1pt,y=1pt,scale=.26]
		\draw[thick,ForestGreen, dotted](239.907, 752.038) .. controls (216.009, 752.005) and (216.05, 527.995) .. (239.841, 527.963);
		\draw[thick,BrickRed, dotted](368.353, 660.022) .. controls (369.1197, 667.9907) and (317.9167, 672.0083) .. (214.744, 672.075);
		\draw[thick,BrickRed, dotted](231.339, 528.004) .. controls (208.171, 528.114) and (207.903, 750.949) .. (232, 752);
		\draw[thick,BrickRed, dotted](79.9237, 627.92) .. controls (79.9746, 635.944) and (124.6483, 639.9333) .. (213.945, 639.888);
		\node[white] at (213.9455, 639.8883) {$\bullet$};
		\node[white] at (214.7436, 672.0747) {$\bullet$};
		\node[BrickRed] at (213.9455, 639.8883) {$\circ$};
		\node[BrickRed] at (214.7436, 672.0747) {$\circ$};
		\node[BrickRed] at (249.5961, 615.7038) {$\bullet$};
		\node[BrickRed] at (250.134, 648.1133) {$\bullet$};
		\draw[thick,BrickRed](250.134, 648.113) .. controls (328.746, 648.3937) and (368.1523, 652.3633) .. (368.353, 660.022);
		\draw[thick,BrickRed](231.339, 528.004) .. controls (256.033, 527.091) and (256.467, 752.001) .. (232, 752);
		\draw[thick,BrickRed](249.596, 615.704) .. controls (136.1545, 616.33) and (79.597, 620.402) .. (79.9237, 627.92);
		\draw[thick,ForestGreen](239.841, 527.963) .. controls (263.822, 527.373) and (264.271, 752.65) .. (239.907, 752.038);
		\draw(64, 704) .. controls (48, 720) and (80, 752) .. (98.6667, 765.3333) .. controls (117.3333, 778.6667) and (122.6667, 773.3333) .. (128, 768) .. controls (133.3333, 762.6667) and (138.6667, 757.3333) .. (125.3333, 738.6667) .. controls (112, 720) and (80, 688) .. (64, 704);
		\draw(384, 704) .. controls (368, 688) and (336, 720) .. (322.6667, 738.6667) .. controls (309.3333, 757.3333) and (314.6667, 762.6667) .. (320, 768) .. controls (325.3333, 773.3333) and (330.6667, 778.6667) .. (349.3333, 765.3333) .. controls (368, 752) and (400, 720) .. (384, 704);
		\draw(128, 512) .. controls (112, 496) and (80, 528) .. (66.6667, 546.6667) .. controls (53.3333, 565.3333) and (58.6667, 570.6667) .. (64, 576) .. controls (69.3333, 581.3333) and (74.6667, 586.6667) .. (93.3333, 573.3333) .. controls (112, 560) and (144, 528) .. (128, 512);
		\draw(320, 512) .. controls (336, 496) and (368, 528) .. (381.3333, 546.6667) .. controls (394.6667, 565.3333) and (389.3333, 570.6667) .. (384, 576) .. controls (378.6667, 581.3333) and (373.3333, 586.6667) .. (354.6667, 573.3333) .. controls (336, 560) and (304, 528) .. (320, 512);
		\draw(384, 576) .. controls (368, 592) and (368, 616) .. (368, 640) .. controls (368, 664) and (368, 688) .. (384, 704);
		\draw(64, 704) .. controls (80, 688) and (80, 664) .. (80, 640) .. controls (80, 616) and (80, 592) .. (64, 576);
		\draw(128, 768) .. controls (144, 752) and (184, 752) .. (224, 752) .. controls (264, 752) and (304, 752) .. (320, 768);
		\draw(128, 512) .. controls (144, 528) and (184, 528) .. (224, 528) .. controls (264, 528) and (304, 528) .. (320, 512);
	\end{tikzpicture}
	\end{subfigure}
	\hfill
	\begin{subfigure}[t]{.24\textwidth}
	\centering
	\begin{tikzpicture}[x=1pt,y=1pt,scale=.26]
		\draw[thick,ForestGreen, dotted](368.071, 651.77) .. controls (368.5543, 661.3087) and (314.3888, 665.5485) .. (205.5743, 664.4896) -- (205.6193, 632.3469) .. controls (121.866, 633.2576) and (79.8764, 629.1593) .. (79.6507, 620.052);
		\draw[thick,BrickRed, dotted](368.353, 660.022) .. controls (369.1197, 667.9907) and (317.9167, 672.0083) .. (214.744, 672.075);
		\draw[thick,BrickRed, dotted](231.339, 528.004) .. controls (208.171, 528.114) and (207.903, 750.949) .. (232, 752);
		\draw[thick,BrickRed, dotted](79.9237, 627.92) .. controls (79.9746, 635.944) and (124.6483, 639.9333) .. (213.945, 639.888);
		\node[white] at (213.9455, 639.8883) {$\bullet$};
		\node[white] at (214.7436, 672.0747) {$\bullet$};
		\node[BrickRed] at (213.9455, 639.8883) {$\circ$};
		\node[BrickRed] at (214.7436, 672.0747) {$\circ$};
		\node[BrickRed] at (249.5961, 615.7038) {$\bullet$};
		\node[BrickRed] at (250.134, 648.1133) {$\bullet$};
		\draw[thick,BrickRed](250.134, 648.113) .. controls (328.746, 648.3937) and (368.1523, 652.3633) .. (368.353, 660.022);
		\draw[thick,BrickRed](231.339, 528.004) .. controls (256.033, 527.091) and (256.467, 752.001) .. (232, 752);
		\draw[thick,BrickRed](249.596, 615.704) .. controls (136.1545, 616.33) and (79.597, 620.402) .. (79.9237, 627.92);
		\draw[thick,ForestGreen](368.071, 651.77) .. controls (368.0237, 645.1793) and (331.8917, 641.4863) .. (259.675, 640.691) -- (259.1957, 607.8658) .. controls (138.7176, 607.8866) and (78.8692, 611.9487) .. (79.6507, 620.052);
		\draw(64, 704) .. controls (48, 720) and (80, 752) .. (98.6667, 765.3333) .. controls (117.3333, 778.6667) and (122.6667, 773.3333) .. (128, 768) .. controls (133.3333, 762.6667) and (138.6667, 757.3333) .. (125.3333, 738.6667) .. controls (112, 720) and (80, 688) .. (64, 704);
		\draw(384, 704) .. controls (368, 688) and (336, 720) .. (322.6667, 738.6667) .. controls (309.3333, 757.3333) and (314.6667, 762.6667) .. (320, 768) .. controls (325.3333, 773.3333) and (330.6667, 778.6667) .. (349.3333, 765.3333) .. controls (368, 752) and (400, 720) .. (384, 704);
		\draw(128, 512) .. controls (112, 496) and (80, 528) .. (66.6667, 546.6667) .. controls (53.3333, 565.3333) and (58.6667, 570.6667) .. (64, 576) .. controls (69.3333, 581.3333) and (74.6667, 586.6667) .. (93.3333, 573.3333) .. controls (112, 560) and (144, 528) .. (128, 512);
		\draw(320, 512) .. controls (336, 496) and (368, 528) .. (381.3333, 546.6667) .. controls (394.6667, 565.3333) and (389.3333, 570.6667) .. (384, 576) .. controls (378.6667, 581.3333) and (373.3333, 586.6667) .. (354.6667, 573.3333) .. controls (336, 560) and (304, 528) .. (320, 512);
		\draw(384, 576) .. controls (368, 592) and (368, 616) .. (368, 640) .. controls (368, 664) and (368, 688) .. (384, 704);
		\draw(64, 704) .. controls (80, 688) and (80, 664) .. (80, 640) .. controls (80, 616) and (80, 592) .. (64, 576);
		\draw(128, 768) .. controls (144, 752) and (184, 752) .. (224, 752) .. controls (264, 752) and (304, 752) .. (320, 768);
		\draw(128, 512) .. controls (144, 528) and (184, 528) .. (224, 528) .. controls (264, 528) and (304, 528) .. (320, 512);
	\end{tikzpicture}
	\end{subfigure}
	\hfill
	\begin{subfigure}[t]{.24\textwidth}
	\centering
	\begin{tikzpicture}[x=1pt,y=1pt,scale=.26]
		\draw[thick,ForestGreen, dotted](368.071, 651.77) .. controls (368.5543, 661.3087) and (314.5803, 665.5427) .. (206.149, 664.472);
		\draw[thick,ForestGreen, dotted](240, 528) .. controls (229.2393, 528) and (223.3657, 562.717) .. (222.379, 632.151);
		\draw[thick,ForestGreen, dotted](222.379, 632.151) .. controls (127.4525, 633.1923) and (79.8764, 629.1593) .. (79.6507, 620.052);
		\draw[thick,ForestGreen, dotted](223.919, 752) .. controls (213.1743, 752.0267) and (207.251, 722.8507) .. (206.149, 664.472);
		\draw[thick,BrickRed, dotted](368.353, 660.022) .. controls (369.1197, 667.9907) and (317.9167, 672.0083) .. (214.744, 672.075);
		\draw[thick,BrickRed, dotted](231.339, 528.004) .. controls (208.171, 528.114) and (207.903, 750.949) .. (232, 752);
		\draw[thick,BrickRed, dotted](79.9237, 627.92) .. controls (79.9746, 635.944) and (124.6483, 639.9333) .. (213.945, 639.888);
		\node[white] at (213.9455, 639.8883) {$\bullet$};
		\node[white] at (214.7436, 672.0747) {$\bullet$};
		\node[BrickRed] at (213.9455, 639.8883) {$\circ$};
		\node[BrickRed] at (214.7436, 672.0747) {$\circ$};
		\node[BrickRed] at (249.5961, 615.7038) {$\bullet$};
		\node[BrickRed] at (250.134, 648.1133) {$\bullet$};
		\draw[thick,BrickRed](250.134, 648.113) .. controls (328.746, 648.3937) and (368.1523, 652.3633) .. (368.353, 660.022);
		\draw[thick,BrickRed](231.339, 528.004) .. controls (256.033, 527.091) and (256.467, 752.001) .. (232, 752);
		\draw[thick,BrickRed](249.596, 615.704) .. controls (136.1545, 616.33) and (79.597, 620.402) .. (79.9237, 627.92);
		\draw[thick,ForestGreen](240, 528) .. controls (250.7453, 527.786) and (257.3037, 565.3497) .. (259.675, 640.691);
		\draw[thick,ForestGreen](259.675, 640.691) .. controls (331.8917, 641.4863) and (368.0237, 645.1793) .. (368.071, 651.77);
		\draw[thick,ForestGreen](79.6507, 620.052) .. controls (78.8692, 611.9487) and (132.7783, 607.9017) .. (241.378, 607.911);
		\draw[thick,ForestGreen](241.378, 607.911) .. controls (245.546, 647.934) and (238.241, 752.033) .. (223.919, 752);
		\draw(64, 704) .. controls (48, 720) and (80, 752) .. (98.6667, 765.3333) .. controls (117.3333, 778.6667) and (122.6667, 773.3333) .. (128, 768) .. controls (133.3333, 762.6667) and (138.6667, 757.3333) .. (125.3333, 738.6667) .. controls (112, 720) and (80, 688) .. (64, 704);
		\draw(384, 704) .. controls (368, 688) and (336, 720) .. (322.6667, 738.6667) .. controls (309.3333, 757.3333) and (314.6667, 762.6667) .. (320, 768) .. controls (325.3333, 773.3333) and (330.6667, 778.6667) .. (349.3333, 765.3333) .. controls (368, 752) and (400, 720) .. (384, 704);
		\draw(128, 512) .. controls (112, 496) and (80, 528) .. (66.6667, 546.6667) .. controls (53.3333, 565.3333) and (58.6667, 570.6667) .. (64, 576) .. controls (69.3333, 581.3333) and (74.6667, 586.6667) .. (93.3333, 573.3333) .. controls (112, 560) and (144, 528) .. (128, 512);
		\draw(320, 512) .. controls (336, 496) and (368, 528) .. (381.3333, 546.6667) .. controls (394.6667, 565.3333) and (389.3333, 570.6667) .. (384, 576) .. controls (378.6667, 581.3333) and (373.3333, 586.6667) .. (354.6667, 573.3333) .. controls (336, 560) and (304, 528) .. (320, 512);
		\draw(384, 576) .. controls (368, 592) and (368, 616) .. (368, 640) .. controls (368, 664) and (368, 688) .. (384, 704);
		\draw(64, 704) .. controls (80, 688) and (80, 664) .. (80, 640) .. controls (80, 616) and (80, 592) .. (64, 576);
		\draw(128, 768) .. controls (144, 752) and (184, 752) .. (224, 752) .. controls (264, 752) and (304, 752) .. (320, 768);
		\draw(128, 512) .. controls (144, 528) and (184, 528) .. (224, 528) .. controls (264, 528) and (304, 528) .. (320, 512);
	\end{tikzpicture}
	\end{subfigure}
	\hfill
	\begin{subfigure}[t]{.24\textwidth}
	\centering
	\begin{tikzpicture}[x=1pt,y=1pt,scale=.26]
		\draw[thick,ForestGreen, dotted](368.071, 651.77) .. controls (368.5543, 661.3087) and (319.968, 665.5833) .. (222.312, 664.594);
		\draw[thick,ForestGreen, dotted](240, 528) .. controls (223.979, 528) and (220.83, 639.789) .. (222.312, 664.594);
		\draw[thick,ForestGreen, dotted](206.211, 632.107) .. controls (122.0632, 633.1777) and (79.8764, 629.1593) .. (79.6507, 620.052);
		\draw[thick,ForestGreen, dotted](223.919, 752) .. controls (209.414, 752.036) and (203.935, 672.27) .. (206.212, 632.024);
		\draw[thick,BrickRed, dotted](368.353, 660.022) .. controls (369.1197, 667.9907) and (317.9167, 672.0083) .. (214.744, 672.075);
		\draw[thick,BrickRed, dotted](231.339, 528.004) .. controls (208.171, 528.114) and (207.903, 750.949) .. (232, 752);
		\draw[thick,BrickRed, dotted](79.9237, 627.92) .. controls (79.9746, 635.944) and (124.6483, 639.9333) .. (213.945, 639.888);
		\node[white] at (213.9455, 639.8883) {$\bullet$};
		\node[white] at (214.7436, 672.0747) {$\bullet$};
		\node[BrickRed] at (213.9455, 639.8883) {$\circ$};
		\node[BrickRed] at (214.7436, 672.0747) {$\circ$};
		\node[BrickRed] at (249.5961, 615.7038) {$\bullet$};
		\node[BrickRed] at (250.134, 648.1133) {$\bullet$};
		\draw[thick,BrickRed](250.134, 648.113) .. controls (328.746, 648.3937) and (368.1523, 652.3633) .. (368.353, 660.022);
		\draw[thick,BrickRed](231.339, 528.004) .. controls (256.033, 527.091) and (256.467, 752.001) .. (232, 752);
		\draw[thick,BrickRed](249.596, 615.704) .. controls (136.1545, 616.33) and (79.597, 620.402) .. (79.9237, 627.92);
		\draw[thick,ForestGreen](240, 528) .. controls (250.7453, 527.786) and (256.7383, 554.4493) .. (257.979, 607.99);
		\draw[thick,ForestGreen](242.817, 640.011) .. controls (326.2723, 641.2597) and (368.0237, 645.1793) .. (368.071, 651.77);
		\draw[thick,ForestGreen](79.6507, 620.052) .. controls (78.8692, 611.9487) and (138.4063, 607.9273) .. (258.262, 607.988);
		\draw[thick,ForestGreen](242.817, 640.011) .. controls (239.7663, 714.6923) and (233.467, 752.022) .. (223.919, 752);
		\draw(64, 704) .. controls (48, 720) and (80, 752) .. (98.6667, 765.3333) .. controls (117.3333, 778.6667) and (122.6667, 773.3333) .. (128, 768) .. controls (133.3333, 762.6667) and (138.6667, 757.3333) .. (125.3333, 738.6667) .. controls (112, 720) and (80, 688) .. (64, 704);
		\draw(384, 704) .. controls (368, 688) and (336, 720) .. (322.6667, 738.6667) .. controls (309.3333, 757.3333) and (314.6667, 762.6667) .. (320, 768) .. controls (325.3333, 773.3333) and (330.6667, 778.6667) .. (349.3333, 765.3333) .. controls (368, 752) and (400, 720) .. (384, 704);
		\draw(128, 512) .. controls (112, 496) and (80, 528) .. (66.6667, 546.6667) .. controls (53.3333, 565.3333) and (58.6667, 570.6667) .. (64, 576) .. controls (69.3333, 581.3333) and (74.6667, 586.6667) .. (93.3333, 573.3333) .. controls (112, 560) and (144, 528) .. (128, 512);
		\draw(320, 512) .. controls (336, 496) and (368, 528) .. (381.3333, 546.6667) .. controls (394.6667, 565.3333) and (389.3333, 570.6667) .. (384, 576) .. controls (378.6667, 581.3333) and (373.3333, 586.6667) .. (354.6667, 573.3333) .. controls (336, 560) and (304, 528) .. (320, 512);
		\draw(384, 576) .. controls (368, 592) and (368, 616) .. (368, 640) .. controls (368, 664) and (368, 688) .. (384, 704);
		\draw(64, 704) .. controls (80, 688) and (80, 664) .. (80, 640) .. controls (80, 616) and (80, 592) .. (64, 576);
		\draw(128, 768) .. controls (144, 752) and (184, 752) .. (224, 752) .. controls (264, 752) and (304, 752) .. (320, 768);
		\draw(128, 512) .. controls (144, 528) and (184, 528) .. (224, 528) .. controls (264, 528) and (304, 528) .. (320, 512);
	\end{tikzpicture}
	\end{subfigure}
	\hfill
	\caption{A ribbon graph $G$ and an embedded ribbon graph $[G,f]$ on a sphere with $4$ boundary components $\Sigma$, and all essential simple loops and dumbbells on them.}
	\label{fig:simple:loops:dumbbells}
\end{figure}

\begin{proof}
	Let $x  \in \RR_{\geq 0}^{E_G}$ be in the image of $\mathfrak{m}_{[G,f]}$, \textit{i.e.} there exists $\mathcal{F} \in {\rm MF}^{\bullet}_{\Sigma}$ such that $\mathfrak{m}_{[G,f]}(\mathcal{F}) = x$. For a vertex $v$ of $G$, let us denote by $e,e',e''$ the adjacent edges, respecting the cyclic order. Then there must be a switch at $v$ and one should specify the weights of this switch. These are three numbers $y_{e},y_{e'},y_{e''} \in \RR_{\geq 0}$ such that
	\[
		x_e = y_{e'} + y_{e''},
		\qquad
		x_{e'} = y_e + y_{e''},
		\qquad
		x_{e''} = y_{e} + y_{e'}.
	\]
	See Figure~\ref{fig:decomposition:SN} for an example in the case of $\mathcal{F}$ being a multicurve. This linear system of equations admits a solution in non-negative real numbers if and only if the three corners conditions are satisfied, namely
	\[
		x_{e} \leq x_{e'} + x_{e''},
		\qquad
		x_{e'} \leq x_{e''} + x_e,
		\qquad
		x_{e''} \leq x_e + x_{e'}.
	\]
	When the solution exists, it is unique and given by the formulas
	\[
		y_{e} = \frac{x_{\Delta}}{2},
		\qquad
		x_{\Delta} = x_{e'} + x_{e''} - x_{e} \text{ for each corner }\Delta = (e,e',e'').
	\]
	This gives the first part of the lemma. By definition, a measured foliation $\mathcal{F} \in {\rm MF}_\Sigma^{\bullet}$ belongs to ${\rm MF}_\Sigma$ if and only if none of its leaves is homotopic to a boundary component of $\Sigma$. This is the case when there is a stop around each face $\mathfrak{f}$, \textit{i.e.} if and only if there exists a corner $\Delta = (e,e',e'')$ around $\mathfrak{f}$ such that $y_{e} = 0$, or equivalently $x_{\Delta} = 0$. This justifies~\eqref{eqn:Z:G:fan}, which is written as a finite union of convex polyhedral cones indexed by the location of the stops, \textit{i.e.} maps $\Delta \colon F_{G} \rightarrow {\rm C}_{G}$ such that $\Delta(\mathfrak{f}) \in {\rm C}(\mathfrak{f})$.

	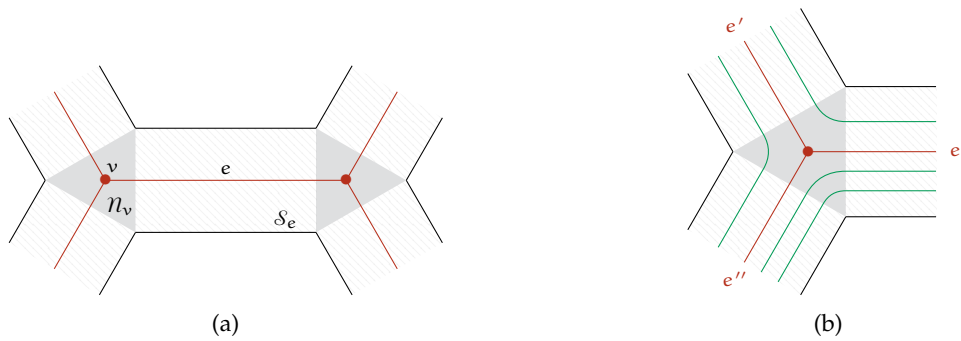
\begin{figure}
		\centering
		\begin{subfigure}[t]{.48\textwidth}
			\centering
			\begin{tikzpicture}[scale=.8]
				\fill [Gray,opacity=.3] ($(-2,0) + (60:1)$) -- ($(-2,0) + (180:1)$) -- ($(-2,0) + (-60:1)$) -- cycle;
				\fill [Gray,opacity=.3] ($(2,0) + (120:1)$) -- ($(2,0) + (0:1)$) -- ($(2,0) + (-120:1)$) -- cycle;

				\fill [pattern=north west lines, pattern color=Gray,opacity=.3] ($(-2,0) + (60:1)$) -- ($(-2,0) + (-60:1)$) -- ($(2,0) + (-120:1)$) -- ($(2,0) + (120:1)$) -- cycle;

				\fill [pattern=north west lines, pattern color=Gray,opacity=.3] ($(-2,0) + (60:1)$) -- ($(-2,0) + (60:1) + (120:1.2)$) -- ($(-2,0) + (180:1) + (120:1.2)$) -- ($(-2,0) + (180:1)$) -- cycle;
				\fill [pattern=north west lines, pattern color=Gray,opacity=.3] ($(-2,0) + (-60:1)$) -- ($(-2,0) + (-60:1) + (-120:1.2)$) -- ($(-2,0) + (180:1) + (-120:1.2)$) -- ($(-2,0) + (180:1)$);
				\fill [pattern=north west lines, pattern color=Gray,opacity=.3] ($(2,0) + (120:1)$) -- ($(2,0) + (120:1) + (60:1.2)$) -- ($(2,0) + (0:1) + (60:1.2)$) -- ($(2,0) + (0:1)$);
				\fill [pattern=north west lines, pattern color=Gray,opacity=.3] ($(2,0) + (-120:1)$) -- ($(2,0) + (-120:1) + (-60:1.2)$) -- ($(2,0) + (0:1) + (-60:1.2)$) -- ($(2,0) + (0:1)$);

				\draw ($(-2,0) + (60:1)$) -- ($(-2,0) + (60:1) + (120:1.2)$);
				\draw ($(-2,0) + (180:1)$) -- ($(-2,0) + (180:1) + (120:1.2)$);

				\draw ($(-2,0) + (-60:1)$) -- ($(-2,0) + (-60:1) + (-120:1.2)$);
				\draw ($(-2,0) + (180:1)$) -- ($(-2,0) + (180:1) + (-120:1.2)$);

				\draw ($(-2,0) + (60:1)$) -- ($(2,0) + (120:1)$);
				\draw ($(-2,0) + (-60:1)$) -- ($(2,0) + (-120:1)$);

				\draw ($(2,0) + (120:1)$) -- ($(2,0) + (120:1) + (60:1.2)$);
				\draw ($(2,0) + (0:1)$) -- ($(2,0) + (0:1) + (60:1.2)$);

				\draw ($(2,0) + (-120:1)$) -- ($(2,0) + (-120:1) + (-60:1.2)$);
				\draw ($(2,0) + (0:1)$) -- ($(2,0) + (0:1) + (-60:1.2)$);

				\draw[BrickRed] (-2,0) -- (2,0);
				\draw[BrickRed] ($(-2,0) + (120:1.7)$) -- (-2,0) -- ($(-2,0) + (-120:1.7)$);
				\draw[BrickRed] ($(2,0) + (60:1.7)$) -- (2,0) -- ($(2,0) + (-60:1.7)$);
				\node[BrickRed] at (-2,0) {$\bullet$};
				\node[BrickRed] at (2,0) {$\bullet$};

				\node at ($(-2,0) + (60:.24)$) {\scriptsize$v$};
				\node at ($(-2,0) + (-60:.5)$) {\footnotesize$\mathscr{N}_v$};
				\node at (0,.18) {\scriptsize$e$};
				\node at (1,-.65) {\footnotesize$\mathscr{S}_e$};
			\end{tikzpicture}
			\caption{}
			\label{fig:decomposition:SN:a}
		\end{subfigure}
		\begin{subfigure}[t]{.48\textwidth}
			\centering
			\begin{tikzpicture}[scale=1]
				\fill [Gray,opacity=.3] (60:1) -- (180:1) -- (-60:1) -- cycle;

				\fill [pattern=north west lines, pattern color=Gray,opacity=.3] (60:1) -- (-60:1) -- ($(-60:1) + (0:1.2)$) -- ($(60:1) + (0:1.2)$) -- cycle;
				\fill [pattern=north west lines, pattern color=Gray,opacity=.3] (60:1) -- ($(60:1) + (120:1.2)$) -- ($(180:1) + (120:1.2)$) -- (-1,0) -- cycle;
				\fill [pattern=north west lines, pattern color=Gray,opacity=.3] (-60:1) -- ($(-60:1) + (-120:1.2)$) -- ($(180:1) + (-120:1.2)$) -- (-1,0);

				\draw (60:1) -- ($(60:1) + (120:1.2)$);
				\draw (180:1) -- ($(180:1) + (120:1.2)$);

				\draw (-60:1) -- ($(-60:1) + (-120:1.2)$);
				\draw (180:1) -- ($(180:1) + (-120:1.2)$);

				\draw (60:1) -- ($(60:1) + (0:1.2)$);
				\draw (-60:1) -- ($(-60:1) + (0:1.2)$);

				\draw[BrickRed] (0,0) -- (1.7,0);
				\draw[BrickRed] (120:1.7) -- (0,0) -- (-120:1.7);
				\node[BrickRed] at (0,0) {$\bullet$};
				\node[BrickRed] at (1.95,0) {\scriptsize$e$};
				\node[BrickRed] at (120:1.95) {\scriptsize$e'$};
				\node[BrickRed] at (-120:1.95) {\scriptsize$e''$};

				\draw[ForestGreen] (1.7,.4) -- (.5,.4) to[out=180,in=-60] ($(120:.5) + (30:.4)$) -- ($(120:1.7) + (30:.4)$);
				\draw[ForestGreen] (1.7,-.26) -- (.5,-.26) to[out=180,in=60] ($(-120:.5) + (-30:.26)$) -- ($(-120:1.7) + (-30:.26)$);
				\draw[ForestGreen] (1.7,-.52) -- (.5,-.52) to[out=180,in=60] ($(-120:.5) + (-30:.52)$) -- ($(-120:1.7) + (-30:.52)$);
				\draw[ForestGreen] ($(-120:1.7) + (150:.4)$) -- ($(-120:.5) + (150:.4)$) to[out=60,in=-60] ($(120:.5) + (-150:.4)$) -- ($(120:1.7) + (-150:.4)$);

			\end{tikzpicture}
			\caption{}
			\label{fig:decomposition:SN:b}
		\end{subfigure}
		\caption{If $\mathcal{F}$ is a measured foliation in $\Sigma$ associated to a multicurve $c$, then the values $x_e$ and $y_e$ have the following interpretation. First, decompose the embedded ribbon graph into strips $\mathscr{S}_e$ associated to each edge, and triangular neighbourhoods $\mathscr{N}_v$ associated to each vertex (Figure~\ref{fig:decomposition:SN:a}). Then isotope $c$ to a non-backtracking simple representative that has $x_{e}$ parallel paths in the strip $\mathscr{S}_e$, and $y_e$ paths in the corner of $\mathscr{N}_v$ opposite to $e$. (Figure~\ref{fig:decomposition:SN:b}).}
		\label{fig:decomposition:SN}
	\end{figure}

	\medskip

	The identification of the rays essentially follows from \cite[Proof of Proposition~3.11.3]{Mos03}. For the reader's convenience, we spell out the argument.

	\medskip

        Let us first prove that each image of an essential simple loop or an essential dumbbell generates an extremal ray in any polyhedral cone $Z_{G,\Delta}$ that contains it. Let $\Delta \in \mathfrak{D}_G$ and let $\mathcal{F}$ be an essential loop or an essential dumbbell such that $\mathfrak{m}_{[G,f]}(\mathcal{F}) \in Z_{G,\Delta}$. Assume that one can write a sum $u = \mathfrak{m}_{[G,f]}(\mathcal{F}) = v + w$ where $v, w \in Z_G^\bullet$. We denote $\varsigma$ the support of $\mathcal{F}$, \textit{i.e.} the set of edges of $G$ whose dual arcs intersect positively $\mathcal{F}$. Because $v,\, w$ are non-negative vectors, their supports must be contained in $\varsigma$. If $\mathcal{F}$ is an essential simple curve then the only elements of ${\rm MF}_\Sigma$ with support contained in $\varsigma$ are multiple of $\mathcal{F}$. We conclude that $v, w \in \mathbb{R}_{\geq 0} \cdot u$ and this shows that $u$ generates an extremal ray. Let us now consider the case where $\mathcal{F}$ is an essential dumbbell. Let $e \in E_G$ be the edge such that $\mathcal{F}$ passes in both directions along it. Then there exist $s,\, s_1,\, s_2,\, t,\, t_1,\, t_2 \in \mathbb{Q}_{\geq 0}$ such that $v = s \mathcal{F} \sqcup s_1 \gamma_1 \sqcup s_2 \gamma_2$ and $w = t \mathcal{F} \sqcup t_1 \gamma_1 \sqcup t_2 \gamma_2$ where $\gamma_1$ and $\gamma_2$ are the two (not necessarily essential) loops at the extremities of the dumbbell. We obtain that $u = (s + t) u + (s_1 + t_1) y_1 + (s_2 + t_2) y_2$ where $y_i = \mathfrak{m}_{[G,f]}(\gamma_i)$. Since the supports of $\gamma_1$ and $\gamma_2$ do not contain $e$ we must have $s + t = 1$. Hence $s_1 = t_1 = s_2 = t_2 = 0$. This shows that $v, w \in \mathbb{R}_{\geq 0} \cdot u$, in other words $u = \mathfrak{m}_{[G,f]}(\mathcal{F})$ generates an extremal ray.

	\medskip

	Assume that $\mathfrak{m}_{[G,f]}(\mathcal{F}) = x$ belongs to a ray of $Z_{G,\Delta}$. As above, we denote $\varsigma$ the support of $\mathcal{F}$. By following the leaves of $\mathcal{F}$, we conclude that $\varsigma$ is a union of closed curves on $G$. Moreover, $\varsigma$ is connected, for otherwise we could write $x$ as a non-trivial sum over the connected components contradicting that $x$ belongs to a ray.

	\medskip

	Choose arbitrarily an orientation on $\varsigma$. We claim that $\varsigma$ passes through each edge at most once in each direction. If this were not the case, one could choose an origin on $\varsigma$ so that it takes the form $\varsigma = a \!\cdot\! e \!\cdot\! b \!\cdot\! e$ where $a$ and $b$ are non-empty paths. Then, $\varsigma_1 = a \!\cdot\! e$ and $\varsigma_2 = b \!\cdot\! e$ are closed curves, and there is a natural decomposition of the weights of $\mathcal{F}$ into two measured foliations $\mathcal{F}_1$, $\mathcal{F}_2$ with respective supports $\varsigma_1$, $\varsigma_2$ such that $x = \mathfrak{m}_{[G,f]}(\mathcal{F}_1) + \mathfrak{m}_{[G,f]}(\mathcal{F}_2)$ contradicting that $x$ belongs to a ray.
	
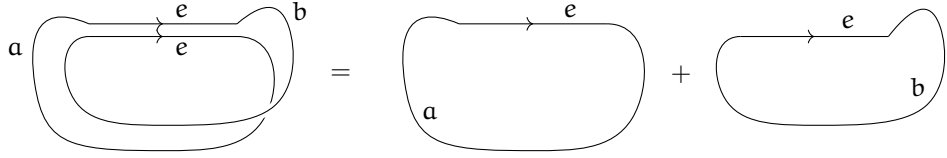
\begin{center}
		\begin{tikzpicture}[x=1pt,y=1pt,xscale=.7,yscale=.6]
			\draw(80, 720) .. controls (48, 736) and (48, 704) .. (50.6667, 682.6667) .. controls (53.3333, 661.3333) and (58.6667, 650.6667) .. (70.6667, 645.3333) .. controls (82.6667, 640) and (101.3333, 640) .. (121.3333, 640) .. controls (141.3333, 640) and (162.6667, 640) .. (173.3333, 658) .. controls (184, 676) and (184, 712) .. (160, 712);
			\draw[white,line width=1.2mm](80, 704) .. controls (64, 704) and (64, 680) .. (73.3333, 668) .. controls (82.6667, 656) and (101.3333, 656) .. (122.6667, 656) .. controls (144, 656) and (168, 656) .. (180, 668) .. controls (192, 680) and (192, 704) .. (188, 718) .. controls (184, 732) and (176, 736) .. (160, 720);
			\draw[decoration={markings,mark=at position 0.5 with {\arrow{>}}},postaction={decorate}] (80, 720) -- (160, 720);
			\draw[decoration={markings,mark=at position 0.5 with {\arrow{>}}},postaction={decorate}](80, 712) -- (160, 712);
			\draw(80, 712) .. controls (64, 712) and (64, 684) .. (73.3333, 670) .. controls (82.6667, 656) and (101.3333, 656) .. (122.6667, 656) .. controls (144, 656) and (168, 656) .. (180, 668) .. controls (192, 680) and (192, 704) .. (188, 718) .. controls (184, 732) and (176, 736) .. (160, 720);
			\node at (130, 728) {$e$};
			\node at (130, 704) {$e$};
			\node at (40, 704) {$a$};
			\node at (194, 728) {$b$};
			\node at (216, 688) {$=$};
			\draw(280, 720) .. controls (248, 736) and (248, 704) .. (250.6667, 682.6667) .. controls (253.3333, 661.3333) and (258.6667, 650.6667) .. (270.6667, 645.3333) .. controls (282.6667, 640) and (301.3333, 640) .. (321.3333, 640) .. controls (341.3333, 640) and (362.6667, 640) .. (373.3333, 660) .. controls (384, 680) and (384, 720) .. (360, 720);
			\draw[decoration={markings,mark=at position 0.5 with {\arrow{>}}},postaction={decorate}](280, 720) -- (360, 720);
			\draw[decoration={markings,mark=at position 0.5 with {\arrow{>}}},postaction={decorate}](432, 712) -- (512, 712);
			\draw(432, 712) .. controls (416, 712) and (416, 684) .. (425.3333, 670) .. controls (434.6667, 656) and (453.3333, 656) .. (474.6667, 656) .. controls (496, 656) and (520, 656) .. (532, 668) .. controls (544, 680) and (544, 704) .. (540, 718) .. controls (536, 732) and (528, 736) .. (512, 712);
			\node at (340, 728) {$e$};
			\node at (490, 721) {$e$};
			\node at (264, 664) {$a$};
			\node at (528, 680) {$b$};
			\node at (400, 688) {$+$};
		\end{tikzpicture}
	\end{center}

	Since $G$ is trivalent, if $\varsigma$ passes through each edge at most once (in any direction), it must be an essential simple loop. Now assume that $\varsigma$ passes through certain edges in both directions. If $e$ is an oriented edge, we use the notation $\bar{e}$ for the edge with opposite orientation. If $\varsigma$ were not an essential dumbbell, there would exist oriented edges $e \neq e'$ with $e \neq \bar{e}'$, and paths $a,b,c,d$ such that one of the following cases holds.
	\begin{itemize}
		\item
		$\varsigma = a \!\cdot\! e \!\cdot\! b \!\cdot\! e' \!\cdot\! c \!\cdot\! \bar{e} \!\cdot\! d \!\cdot\! \bar{e}'$. Then, there exists a natural decomposition $x = \mathfrak{m}_{[G,f]}(\mathcal{F}_1) + \mathfrak{m}_{[G,f]}(\mathcal{F}_2)$ with measured foliations $\mathcal{F}_1,\mathcal{F}_2$ of respective supports $\varsigma_1 = a \!\cdot\! e \!\cdot\! \bar{c} \!\cdot\! \bar{e}'$ and $\varsigma_2 = d \!\cdot\! \bar{e}' \!\cdot\! b \!\cdot\! \bar{e}$.
		
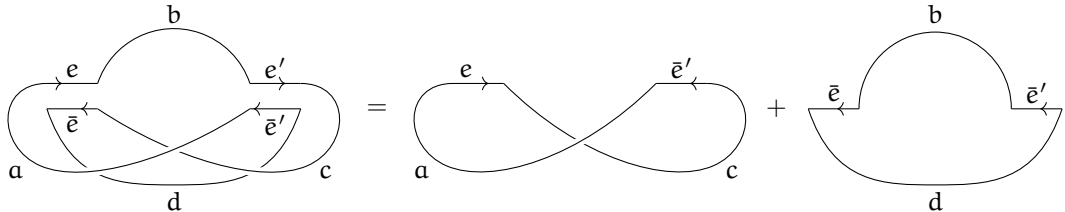
\vspace{-.2cm}
		\begin{center}
			\begin{tikzpicture}[x=1pt,y=1pt,scale=1.2]
				\draw(72, 744) arc[start angle=-161.5651, end angle=-18.4349, x radius=25.2982, y radius=-25.2982];
				\draw[decoration={markings,mark=at position 0.3 with {\arrow{>}}},postaction={decorate}] (56, 744) -- (72, 744);
				\draw[decoration={markings,mark=at position 0.8 with {\arrow{<}}},postaction={decorate}] (56, 736) -- (72, 736);
				\draw[decoration={markings,mark=at position 0.2 with {\arrow{<}}},postaction={decorate}](120, 736) -- (136, 736);
				\draw[decoration={markings,mark=at position 0.7 with {\arrow{>}}},postaction={decorate}](120, 744) -- (136, 744);
				\draw(56, 736) .. controls (64, 712) and (80, 712) .. (96, 712) .. controls (112, 712) and (128, 712) .. (136, 736);
				\draw[white,line width=1.2mm](72, 736) .. controls (96, 720) and (124, 712) .. (138, 718) .. controls (152, 724) and (152, 744) .. (136, 744);
				\draw(72, 736) .. controls (96, 720) and (124, 712) .. (138, 718) .. controls (152, 724) and (152, 744) .. (136, 744);
				\draw[white,line width=1.2mm](56, 744) .. controls (40, 744) and (40, 724) .. (54, 718) .. controls (68, 712) and (96, 720) .. (120, 736);
				\draw(56, 744) .. controls (40, 744) and (40, 724) .. (54, 718) .. controls (68, 712) and (96, 720) .. (120, 736);
				\draw[decoration={markings,mark=at position 0.7 with {\arrow{>}}},postaction={decorate}](184, 744) -- (200, 744);
				\draw[decoration={markings,mark=at position 0.8 with {\arrow{<}}},postaction={decorate}](248, 744) -- (264, 744);
				\draw(200, 744) .. controls (224, 720) and (252, 712) .. (266, 718) .. controls (280, 724) and (280, 744) .. (264, 744);
				\draw[white,line width=1.2mm](184, 744) .. controls (168, 744) and (168, 724) .. (182, 718) .. controls (196, 712) and (224, 720) .. (248, 744);
				\draw(184, 744) .. controls (168, 744) and (168, 724) .. (182, 718) .. controls (196, 712) and (224, 720) .. (248, 744);
				\draw(312, 736) arc[start angle=179.5106, end angle=360.4894, x radius=24.0009, y radius=-24.0009];
				\draw[decoration={markings,mark=at position 0.7 with {\arrow{<}}},postaction={decorate}](296, 736) -- (312, 736);
				\draw[decoration={markings,mark=at position 0.7 with {\arrow{<}}},postaction={decorate}](360, 736) -- (376, 736);
				\draw(296, 736) .. controls (304, 712) and (320, 712) .. (336, 712) .. controls (352, 712) and (368, 712) .. (376, 736);
				\node at (160, 736) {$=$};
				\node at (286, 736) {$+$};
				\node at (64, 748) {$e$};
				\node at (64, 731) {$\bar{e}$};
				\node at (128, 749) {$e'$};
				\node at (128, 731) {$\bar{e}'$};
				\node at (96, 766) {$b$};
				\node at (96, 707) {$d$};
				\node at (46, 716) {$a$};
				\node at (144, 716) {$c$};
				\node at (188, 748) {$e$};
				\node at (256, 749) {$\bar{e}'$};
				\node at (174, 716) {$a$};
				\node at (272, 716) {$c$};
				\node at (304, 741) {$\bar{e}$};
				\node at (368, 741) {$\bar{e}'$};
				\node at (336, 766) {$b$};
				\node at (336, 707) {$d$};
			\end{tikzpicture}
		\end{center}
		\item
		$\varsigma = a \!\cdot\! e \!\cdot\! b \!\cdot\! \bar{e} \!\cdot\! c \!\cdot\! e' \!\cdot\! d \!\cdot\! \bar{e}'$ where $b$ and $d$ are non-empty. Then, there exists a natural decomposition $x = \mathfrak{m}_{[G,f]}(\mathcal{F}_1) + \mathfrak{m}_{[G,f]}(\mathcal{F}_2)$ with measured foliations $\mathcal{F}_1,\mathcal{F}_2$ of respective supports $\varsigma_1 = a \!\cdot\! e \!\cdot\! b \!\cdot\! \bar{e} \!\cdot\! \bar{a} \!\cdot\! e' \!\cdot\! d \!\cdot\! \bar{e}'$ and $\varsigma_2 = \bar{c} \!\cdot\! e \!\cdot\! b \!\cdot\! \bar{e} \!\cdot\! c \!\cdot\! e' \!\cdot\! d \!\cdot\! \bar{e}'$.
		
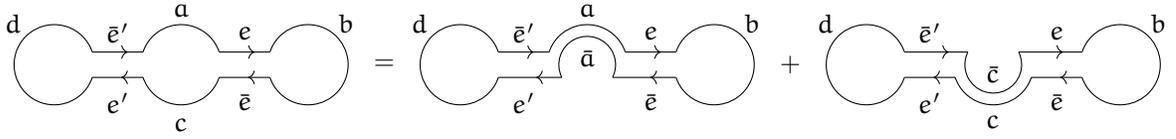
\vspace{-.4cm}
		\begin{center}
			\begin{tikzpicture}[x=1pt,y=1pt,scale=1.2]
				\draw(84, 660) arc[start angle=-161.5651, end angle=-18.4349, x radius=12.6491, y radius=-12.6491];
				\draw(108, 652) arc[start angle=18.4349, end angle=161.5651, x radius=12.6491, y radius=-12.6491];
				\draw(124, 660) arc[start angle=-161.5651, end angle=161.5651, x radius=12.6491, y radius=-12.6491];
				\draw(68, 652) arc[start angle=18.4349, end angle=341.5651, x radius=12.6491, y radius=-12.6491];
				\draw[decoration={markings,mark=at position 0.7 with {\arrow{>}}},postaction={decorate}](68, 660) -- (84, 660);
				\draw[decoration={markings,mark=at position 0.7 with {\arrow{<}}},postaction={decorate}](68, 652) -- (84, 652);
				\draw[decoration={markings,mark=at position 0.7 with {\arrow{<}}},postaction={decorate}](108, 652) -- (124, 652);
				\draw[decoration={markings,mark=at position 0.7 with {\arrow{>}}},postaction={decorate}](108, 660) -- (124, 660);
				\draw(212, 660) arc[start angle=-161.5651, end angle=-18.4349, x radius=12.6491, y radius=-12.6491];
				\draw(252, 660) arc[start angle=-161.5651, end angle=161.5651, x radius=12.6491, y radius=-12.6491];
				\draw(196, 652) arc[start angle=18.4349, end angle=341.5651, x radius=12.6491, y radius=-12.6491];
				\draw[decoration={markings,mark=at position 0.7 with {\arrow{>}}},postaction={decorate}](196, 660) -- (212, 660);
				\draw[decoration={markings,mark=at position 0.7 with {\arrow{<}}},postaction={decorate}](196, 652) -- (216, 652);
				\draw[decoration={markings,mark=at position 0.7 with {\arrow{<}}},postaction={decorate}](232, 652) -- (252, 652);
				\draw[decoration={markings,mark=at position 0.7 with {\arrow{>}}},postaction={decorate}](236, 660) -- (252, 660);
				\draw(364, 652) arc[start angle=18.4349, end angle=161.5651, x radius=12.6491, y radius=-12.6491];
				\draw(380, 660) arc[start angle=-161.5651, end angle=161.5651, x radius=12.6491, y radius=-12.6491];
				\draw(324, 652) arc[start angle=18.4349, end angle=341.5651, x radius=12.6491, y radius=-12.6491];
				\draw[decoration={markings,mark=at position 0.7 with {\arrow{>}}},postaction={decorate}](324, 660) -- (344, 660);
				\draw[decoration={markings,mark=at position 0.7 with {\arrow{<}}},postaction={decorate}](324, 652) -- (340, 652);
				\draw[decoration={markings,mark=at position 0.7 with {\arrow{<}}},postaction={decorate}](364, 652) -- (380, 652);
				\draw[decoration={markings,mark=at position 0.7 with {\arrow{>}}},postaction={decorate}](360, 660) -- (380, 660);
				\draw(216, 652) arc[start angle=153.4349, end angle=386.5651, x radius=8.9443, y radius=-8.9443];
				\draw(360, 660) arc[start angle=-26.5651, end angle=206.5651, x radius=8.9443, y radius=-8.9443];
				\node at (96, 673) {$a$};
				\node at (148, 669) {$b$};
				\node at (96, 637) {$c$};
				\node at (43, 669) {$d$};
				\node at (76, 666) {$\bar{e}'$};
				\node at (76, 644) {$e'$};
				\node at (116, 666) {$e$};
				\node at (116, 644) {$\bar{e}$};
				\node at (224, 673) {$a$};
				\node at (276, 669) {$b$};
				\node at (171, 669) {$d$};
				\node at (224, 658) {$\bar{a}$};
				\node at (352, 653) {$\bar{c}$};
				\node at (352, 638) {$c$};
				\node at (404, 669) {$b$};
				\node at (299, 669) {$d$};
				\node at (244, 666) {$e$};
				\node at (244, 644) {$\bar{e}$};
				\node at (204, 666) {$\bar{e}'$};
				\node at (332, 666) {$\bar{e}'$};
				\node at (372, 644) {$\bar{e}$};
				\node at (372, 666) {$e$};
				\node at (332, 644) {$e'$};
				\node at (204, 644) {$e'$};
				\node at (160, 656) {$=$};
				\node at (288, 656) {$+$};
			\end{tikzpicture}
		\end{center}
	\end{itemize}
	In both cases this contradicts the assumption that $x$ belongs to a ray.
\end{proof}

\subsection{Volume of combinatorial unit balls}
\label{sec:explicit:bcomb}

If $\GG \in \mathcal{T}_{\Sigma}^{{\rm comb}}$, the description in Lemma~\ref{lem:MF:parametrisation} reduces the computation of the Thurston measure of the combinatorial unit ball $\set{\ell_{\GG} \leq 1}$ to the computation of volumes of truncations of polyhedral cones. This can be carried out explicitly on a computer, but at a qualitative level, the result always takes the following form.

\medskip

Let $G$ be a trivalent ribbon graph on a surface $\Sigma$ of type $(g,n)$. We recall that $G$ induces a decomposition of the space of measured foliations ${\rm MF}_\Sigma$ into polyhedral cones $Z_{G,\Delta}$ where $\Delta:  F_G \to {\rm C}_G$ is a choice of a corner in each face, and their union over $\Delta$ is denoted $Z_{G}$. An \emph{elementary simplex} of $Z_{G}$ is a simplex of dimension $6g-6+2n$ in $Z_{G}$ whose extremal rays are linearly independent in $\RR^{E_G}$ and are either essential simple loops or essential dumbbells. A \emph{simplicial decomposition} of $Z_{G}$ is a collection $T_G$ of simplicial cones with disjoint interior and whose union is $Z_G$. Each simplicial cone $t \in T_G$ has $6g-6+2n$ extremal rays generated by an essential simple loop or dumbbell. We denote $R(t) \subset \RR_{\geq 0}^{E_G}$ this set of generators. We define $\det(t)$ to be the volume with respect to the Thurston measure $\mu_{\textup{Th}}$ of the elementary simplex issued from the origin and sides being $R(t)$. The number $\det(t)$ is a positive integer and is also the number of integral points in the semi-open simplex.

\begin{prop}\label{prop:volume:rational:fnct}
        Let $G$ be a trivalent ribbon graph of type $(g,n)$. For any $\bm{G} \in \mathfrak{Z}_G(L)$, that is any metric on the underlying graph $G$, $\mathscr{B}_{g,n}^{{\rm comb}}(\bm{G})$ is a rational function of the edge lengths. More precisely, for any simplicial decomposition $T_G$ of $Z_{G}$ we have
	\begin{equation}
		\mathscr{B}_{g,n}^{{\rm comb}}(\bm{G}) = \frac{1}{(6g-6+2n)!} \sum_{t \in T_G} \frac{1}{\det(t) \cdot \prod_{\rho \in R(t)} \ell_{\bm{G}}(\rho)}.
	\end{equation}
\end{prop}

\begin{proof}
	By definition of a simplicial decomposition: $\mathscr{B}_{\Sigma}^{\textup{comb}}(\GG) = \sum_{t \in T_G} \mu_{\textup{Th}}(t \cap \{\ell_{\bm{G}} \leq 1\})$. From the definition of the Thurston measure
	\begin{align*}
		\mu_{\textup{Th}}(t \cap \{\ell_{\bm{G}} \leq 1\}) &= \lim_{r \to +\infty} \frac{\#\Set{x \in t \cap \ZZ_{\geq 0}^{E_{G}} | \sum_{e \in E_{G}} x_e \, \ell_{\GG}(e) \leq r}}{r^{6g - 6 + 2n}} \\
		& = \frac{1}{\det(t)} \lim_{r \to +\infty} \frac{\#\Set{ z \in \ZZ_{\geq 0}^{R(t)} | \sum_{\rho \in R(t)} z_{\rho}\,\ell_{\GG}(\rho) \leq r}}{r^{6g - 6 + 2n}} \\ 
		& = \frac{1}{\det(t)}\,\frac{1}{(6g - 6 + 2n)! \prod_{\rho \in R(t)} \ell_{\GG}(\rho)}.
	\end{align*}
\end{proof}

\begin{rem}
	Proposition~\ref{prop:volume:rational:fnct} extends to graphs $G$ with higher valencies by choosing any resolution into a trivalent graph with some edges of zero length.
\end{rem}

\subsection{How to use the formula: the \texorpdfstring{$(1,1)$}{(1,1)} case.}
\label{sec:torus}
There is a single trivalent ribbon graph $G$ of genus $1$ with one boundary component (\emph{cf.} Figure~\ref{fig:cell:11}). For a fixed $L \in \RR_+$, the associated polytope is simply
\[
	\mathfrak{Z}_{G}(L) = \Set{ (\ell_A,\ell_B,\ell_C) \in \RR_{+}^3 | \ell_A + \ell_B + \ell_C = \tfrac{L}{2} }.
\]
The automorphism group of $G$ is $\ZZ_6$, where the subgroup $\ZZ_3 \subset \ZZ_6$ is cyclically permuting the three edges, while $\ZZ_2 \subset \ZZ_6$ is the elliptic involution stabilising every point and is the automorphism group of $\mathbf{G}$ for which the lengths of the edges are not equal.

\medskip

$G$ has a unique face $\mathfrak{f}$ and six corners; from the elliptic involution acting on $\mathbf{G}$, $\mathscr{B}_{1,1}^{{\rm comb}}$ reduces to the sum of three contributions. The first one corresponds to the corner $\Delta(f) = (A,B,C)$. The polytope $Z_{G,\Delta}$ is a simplicial cone, with extremal rays $\rho_1 = (1,1,0)$ and $\rho_2 = (1,0,1)$ corresponding to the essential simple loops of Figure~\ref{fig:cell:11}, and with determinant $1$. The two contributions are obtained by cyclic permutation of the role of $(A,B,C)$. For a point $\bm{G} = (\ell_A,\ell_B,\ell_C) \in \mathfrak{Z}_{G}(L)$, we find $\ell_{\bm{G}}(\rho_1) = \ell_A + \ell_B$, $\ell_{\bm{G}}(\rho_2) = \ell_A + \ell_C$, and $\det(t) = 1$. Similarly for the other polyhedral cones, so that
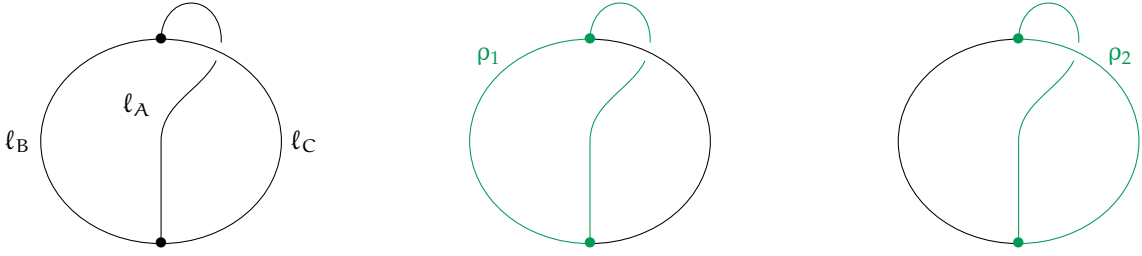
\begin{figure}[t]
	\begin{center}
		\begin{subfigure}[t]{.31\textwidth}
		\centering
		\begin{tikzpicture}[scale=.8]	
			\draw[line width=10pt,white] (0,0) ellipse (2cm and 1.7cm);
			\draw[line width=10pt,white] (0,1.7) to[out=90,in=180] (.5,2.3) to[out=0,in=90] (1,1.7) to[out=-90,in=90] (0,0) -- (0,-.1);

			\draw (0,1.7) to[out=90,in=180] (.5,2.3) to[out=0,in=90] (1,1.7) to[out=-90,in=90] (0,0) -- (0,-1.7);

			\draw [line width=7pt,white] ($(0,0) + (60:2cm and 1.7cm)$) arc (60:80:2cm and 1.7cm);

			\draw (0,0) ellipse (2cm and 1.7cm);
			\node at (0,1.7) {$\bullet$};
			\node at (0,-1.7) {$\bullet$};

			\node at (-2,0) [left] {$\ell_B$};
			\node at (2,0) [right] {$\ell_C$}; 
			\node at (0,.6) [left] {$\ell_A$};
		\end{tikzpicture}
		\end{subfigure}
		\hfill
		\begin{subfigure}[t]{.31\textwidth}
		\centering
		\begin{tikzpicture}[scale=.8]
			\draw[line width=10pt,white] (0,0) ellipse (2cm and 1.7cm);
			\draw[line width=10pt,white] (0,1.7) to[out=90,in=180] (.5,2.3) to[out=0,in=90] (1,1.7) to[out=-90,in=90] (0,0) -- (0,-.1);

			\draw[ForestGreen] (0,1.7) to[out=90,in=180] (.5,2.3) to[out=0,in=90] (1,1.7) to[out=-90,in=90] (0,0) -- (0,-1.7);

			\draw [line width=7pt,white] ($(0,0) + (60:2cm and 1.7cm)$) arc (60:80:2cm and 1.7cm);

			\draw[ForestGreen] (0,1.7) arc (90:270:2cm and 1.7cm);
			\draw (0,-1.7) arc (-90:90:2cm and 1.7cm);
			\node[ForestGreen] at (0,1.7) {$\bullet$};
			\node[ForestGreen] at (0,-1.7) {$\bullet$};

			\node[ForestGreen] at (140:2.2) {$\rho_1$};
		\end{tikzpicture}
		\end{subfigure}
		\hfill
		\begin{subfigure}[t]{.31\textwidth}
		\centering
		\begin{tikzpicture}[scale=.8]
			\draw[line width=10pt,white] (0,0) ellipse (2cm and 1.7cm);
			\draw[line width=10pt,white] (0,1.7) to[out=90,in=180] (.5,2.3) to[out=0,in=90] (1,1.7) to[out=-90,in=90] (0,0) -- (0,-.1);

			\draw[ForestGreen] (0,1.7) to[out=90,in=180] (.5,2.3) to[out=0,in=90] (1,1.7) to[out=-90,in=90] (0,0) -- (0,-1.7);

			\draw [line width=7pt,white] ($(0,0) + (60:2cm and 1.7cm)$) arc (60:80:2cm and 1.7cm);

			\draw (0,1.7) arc (90:270:2cm and 1.7cm);
			\draw[ForestGreen] (0,-1.7) arc (-90:90:2cm and 1.7cm);
			\node[ForestGreen] at (0,1.7) {$\bullet$};
			\node[ForestGreen] at (0,-1.7) {$\bullet$};

			\node[ForestGreen] at (40:2.2) {$\rho_2$};
		\end{tikzpicture}
		\end{subfigure}
	\end{center}
	\caption{The top-dimensional cell of $\mathcal{M}_{1,1}^{\textup{comb}}(L)$ parametrised by edge lengths $(\ell_A,\ell_B,\ell_C)$, together with two essential simple loops $\rho_1$ and $\rho_2$.}
	\label{fig:cell:11}
\end{figure}

\begin{equation}
\label{B11combform} \begin{split}
	\mathscr{B}_{1,1}^{{\rm comb}}(\ell_A,\ell_B,\ell_C)
	& =
	\frac{1}{2}\,\frac{1}{(\ell_A + \ell_B)(\ell_A + \ell_C)} + \frac{1}{2}\,\frac{1}{(\ell_A + \ell_B)(\ell_B + \ell_C)} + \frac{1}{2}\,\frac{1}{(\ell_A + \ell_C)(\ell_B + \ell_C)} \\
	& =
	\frac{L}{2} \frac{1}{(\ell_A + \ell_B)(\ell_B + \ell_C)(\ell_C + \ell_A)}.
\end{split}
\end{equation}
Besides
\[
\begin{split}
	\mathscr{B}_{1,1}^{{\rm comb},\bullet}(\ell_A,\ell_B,\ell_C)
	& =
	\int_{\RR_{+}^3} \dd x_{A}\dd x_{B} \dd x_{C} \,
		\mathbf{1}_{x_A(\ell_B + \ell_C) + x_B(\ell_C + \ell_A) + x_C(\ell_A + \ell_B) \leq 1} \\
	& = \frac{1}{6(\ell_A + \ell_B)(\ell_B + \ell_C)(\ell_C + \ell_A)}
	=
	\frac{2!}{3!} \, \frac{\mathscr{B}_{1,1}^{{\rm comb}}(\ell_A,\ell_B,\ell_C)}{L}
\end{split}
\]
as expected from Lemma~\ref{lem:MF:MFbullet}.

\medskip

Let us now integrate over the moduli space (see Equation~\eqref{eqn:integral:functions}). We recall that $\#\Aut(G) = 6$, and the Kontsevich measure on $\mathfrak{Z}_G(L)$ is $\dd\mu_{\textup{K}} = \dd \ell_A \dd \ell_B$. Expressing $\ell_C = \frac{L}{2} - \ell_A - \ell_B$ and performing the change of variable $(\ell_A,\ell_B) = \frac{L}{2}(a,b)$, we can compute
\[
\begin{split}
	\int_{\mathcal{M}_{1,1}^{{\rm comb}}(L)} \mathscr{B}_{1,1}^{{\rm comb}} \, \dd\mu_{\textup{K}}
	& =
	\frac{1}{6} \int_{\,\,\substack{0 < \ell_A,\ell_B < L/2 \\ \ell_A + \ell_B < L/2}}
		\; \frac{L}{2} \frac{\dd\ell_A \dd\ell_B}{(\ell_A + \ell_B)(\frac{L}{2} - \ell_A)(\frac{L}{2} - \ell_B)} \\
	& =
	\frac{1}{6}\, \int_{\,\,\substack{0 < a,b < 1 \\ a + b < 1}}
		\; \frac{\dd a\,\dd b}{(a + b)(1 - a)(1 - b)} \\
	& =
	- \frac{1}{3} \int_{0}^{1} \frac{\ln(a)}{1 - a^2} \, \dd a \\
	& =
	\frac{\Li_2(1) - \Li_2(-1)}{6} = \frac{\pi^2}{24}.
\end{split}
\]
As expected from \eqref{MVnorm}, this value coincides with $\int_{\mathcal{M}_{1,1}(L)} \mathscr{B}_{1,1}\dd\mu_{{\rm WP}} = \frac{\pi^2}{24}$ found \textit{e.g.} in \cite{ABCDGLW19}.

\medskip

Let us look at the integral of the $s$-th power for $s > 1$
\[
	\int_{\mathcal{M}_{1,1}^{{\rm comb}}(L)} \mathscr{B}^{{\rm comb}}_{1,1}\,\dd\mu_{\textup{K}} = \frac{(L/2)^{1 - s}}{6}\,\mathbb{B}(s),\qquad \mathbb{B}(s) \coloneqq \int_{\substack{a,b \geq 0 \\ a + b \leq 1}} \frac{\dd a\,\dd b}{\big((a + b)(1 - a)(1 - b)\big)^{s}}.
\]
By elementary means we shall prove that it is finite if and only if $s < 2$, and more precisely

\begin{prop}\label{lem11s}
	We have $\mathbb{B}(s) \sim \frac{3}{2 - s}$ when $s \rightarrow 2^-$.
\end{prop}
\begin{proof}
	Let $D = \{(a,b) \in \mathbb{R}_{\geq 0}^2\,\,|\,\,a + b \leq 1\}$ be the $2$-simplex. If $s = 2$, we shall see that the non-integrability comes from the divergence of the integrand at the vertices of $D$, \textit{i.e.} $(a,b) = (0,0)$, $(a,b) = (1,0)$ and $(a,b) = (0,1)$. We decompose the domain of integration, introducing
	\[
		D_{00} = \big\{(a,b) \in D\,\,\big|\,\,a + b \leq \tfrac{1}{2}\big\}, \qquad D_{10} = \big\{(a,b) \in D\,\,\big|\,\,a \geq \tfrac{1}{2}\big\},\qquad D_{01} = \big\{(a,b) \in D\,\,\big|\,\,b \geq \tfrac{1}{2}\big\},
	\]
	and $\tilde{D} = D \setminus \big(D_{00} \cup D_{10} \cup D_{11}\big)$. We analyse separately the contributions of these domains to the integral, with obvious notations:
	\[
		\mathbb{B}(s) = \mathbb{B}_{00}(s) + \mathbb{B}_{10}(s) + \mathbb{B}_{01}(s) + \tilde{\mathbb{B}}(s).
	\]
	The integrand being a continuous function on $\tilde{D}$, $\tilde{\mathbb{B}}(s)$ remains bounded when $s \rightarrow 2$. For the first three contributions, the idea is to choose coordinates transforming the domain into a square and which include a coordinate $c$ measuring the distance to the vertex, then split the integrand into a contribution coming solely from the vanishing factor in the denominator, and a remainder which will remain bounded when $s$ approaches $2$.

	\medskip

	We start with $\mathbb{B}_{00}(s)$. With the change of variable $(c,u) = (a + b,\frac{a}{a + b})$, we find:
	\begin{equation}
	\begin{split}
		\mathbb{B}_{00}(s)  & = \int_{0}^{\frac{1}{2}} \dd c\,c^{1 - s} \int_{0}^{1} \frac{\dd u}{\big((1 - cu)(1 - c + cu)\big)^s} \\
		& = \int_{0}^{\frac{1}{2}} \dd c\,c^{1 - s} + \int_{0}^{\frac{1}{2}} \dd c\,c^{1 - s} \int_{0}^{1} \dd u\Bigg(\frac{1}{\big((1 - cu)(1 - c + cu)\big)^s} - 1\Bigg) \\
		& = \frac{(1/2)^{2 - s}}{2 - s} + \int_{0}^{\frac{1}{2}} \dd c\,c^{1 - s} O(c)  \\
		& \mathop{=}_{s \rightarrow 2} \frac{1}{2 - s} + O(1),
	\end{split}
	\end{equation}
	where the $O(c)$ is uniform for $c \in [0,\frac{1}{2}]$ and $s \in (0,2)$, and we observed $(\tfrac{1}{2})^{2 - s} = 1 + O(2 - s)$ when $s \rightarrow 2$. For $\mathbb{B}_{10}(s)$, we perform the change of variable $(c,u) = \big(1 - a,\frac{b}{1 - a}\big)$ and get
	\[
		\mathbb{B}_{10}(s) = \int_{0}^{\frac{1}{2}} \dd c \,c^{1 - s} \int_{0}^{1} \frac{\dd u}{\big((1- cu)(1 - c + cu)\big)^s} = \mathbb{B}_{00}(s) .
	\]
	Exchanging the role of $a$ and $b$ we also have $\mathbb{B}_{01}(s) = \mathbb{B}_{00}(s)$, hence the result. 
\end{proof}

There is no simple expression for $\mathbb{B}(s)$, but the expression can be transformed in various ways. For instance, with the change of variable $(c,v)=(a+b,\frac{a}{a+b})$ sending $(a,b) \in D$ to $(c,v) \in (0,1)^2$:
\[
	\mathbb{B}(s) = \int_{0}^{1} \frac{c\,\dd c}{\big(c(1-  c)\big)^{s}} \int_{0}^{1} \frac{\dd v}{\big(1 + \frac{c^2}{1 - c} v(1 - v)\big)^{s}}.
\]
By symmetry $v \mapsto 1 - v$, we can restrict the integration to $v \in [0,\frac{1}{2}]$ while multiplying the result by $2$. We then set $y = \frac{c}{2 - c}$ and $x = 1 - 2v$, obtaining 
\[
	\mathbb{B}(s) = 2^{2- s}  \int_{(0,1)^2} \dd x \dd y\, (1+y)^{3(s - 1)} y^{1 - s} (1 - y^2x^2)^{-s} 
\]
as announced in the introduction.

\medskip

Proposition~\ref{lem11s} tells us that the behaviour of $\mathscr{B}_{1,1}^{{\rm comb}}$ already deviates from the one of $\mathscr{B}_{1,1}$, as the latter has a finite square-norm for the Weil--Petersson measure. This simple example shows that $\mathscr{B}_{g,n}^{{\rm comb}}$ has non-trivial integrability properties. The purpose of the next section is to analyse them systematically.

\section{Integrability\texorpdfstring{ of $\mathscr{B}_{\Sigma}^{{\rm comb}}$}{}}
\label{sec:integrability}

\subsection{Geometry of the cells in \texorpdfstring{$\mathcal{T}_{\Sigma}^{{\rm comb}}(L)$}{the combinatorial Teichm\"uller space}}
\label{sec:cells:tcomb}

As a preparation, we study the geometry of the cells $\mathfrak{Z}_{G}(L)$ of $\mathcal{M}_{\Sigma}^{{\rm comb}}(L)$, and in particular we shall characterise the tangent cone at the vertices of the cells.

\begin{defn}\label{nonresdef}
	We say that $L \in \RR_{+}^n$ is \emph{non-resonant} if for any non-zero map $\epsilon \colon \{1,\ldots,n\} \rightarrow \{-1,0,1\}$, we have
	\[
		\sum_{i = 1}^n \epsilon_i L_i \neq 0.
	\]
\end{defn}

\begin{defn}
	Let $G$ be a trivalent ribbon graph with $n$ boundary components and let $S \subseteq E_{G}$. We let $G^{\ast}_S$ the subgraph of the dual graph $G^{\ast}$ in which we keep only the duals of edges from $S$. We call a subset $S \subseteq E_G$ a \emph{support set of $G$} if
	\begin{itemize}
		\item it has $n$ elements,
		\item each face of $G$ contains at least an edge in $S$,
		\item each connected component of $G^\ast_S$ contains a unique cycle which has odd length.
	\end{itemize}
\end{defn}

\begin{figure}
	\centering
	\begin{subfigure}[t]{.6\textwidth}
	\centering
		\begin{tikzpicture}[x=1pt,y=1pt,scale=.6]
			\draw(128, 640) circle[radius=64];
			\draw(336, 640) circle[radius=48];
			\draw(128, 640) circle[radius=32];
			\draw(64, 640) -- (96, 640);
			\draw(192, 640) -- (288, 640);
			\node at (64, 640) {$\bullet$};
			\node at (96, 640) {$\bullet$};
			\node at (192, 640) {$\bullet$};
			\node at (288, 640) {$\bullet$};
			\node at (240, 640) [above] {$e_5$};
			\node at (160, 640) [right] {$e_1$};
			\node at (80, 640) [above] {$e_2$};
			\node at (128, 704) [above] {$e_3$};
			\node at (128, 576) [below] {$e_4$};
			\node at (336, 688) [above] {$e_6$};
		\end{tikzpicture}
	\end{subfigure}
	\\
	\begin{subfigure}[t]{.48\textwidth}
	\centering
		\begin{tikzpicture}[x=1pt,y=1pt,scale=.6]
			\draw [dotted] (128, 640) circle[radius=64];
			\draw [thick] (336, 640) circle[radius=48];
			\draw [thick] (128, 640) circle[radius=32];
			\draw [thick] (64, 640)  -- (96, 640);
			\draw [thick] (192, 640)  -- (288, 640);
			\draw [RoyalBlue] (240, 720)  .. controls (256, 560) and (344, 552) .. (384, 592)  .. controls (424, 632) and (416, 720) .. (240, 720);
			\draw [RoyalBlue] (176, 640)  .. controls (128, 560) and (80, 600) .. (80, 640)  .. controls (80, 680) and (128, 720) .. (176, 640);
			\draw [RoyalBlue, dotted] (176, 640)  .. controls (128, 528) and (48, 584) .. (48, 640)  .. controls (48, 696) and (128, 752) .. (240, 720);
			\draw [RoyalBlue] (336, 640)  -- (240, 720);
			\draw [RoyalBlue, dotted] (240, 720)  -- (176, 640);
			\draw [RoyalBlue] (176, 640)  -- (128, 640);

			\node at (64, 640) {$\bullet$};
			\node at (96, 640) {$\bullet$};
			\node at (192, 640) {$\bullet$};
			\node at (288, 640) {$\bullet$};

			\node [white] at (240, 720) {$\bullet$};
			\node [white] at (176, 640) {$\bullet$};
			\node [white] at (128, 640) {$\bullet$};
			\node [white] at (336, 640) {$\bullet$};
			\node at (240, 720) {$\circ$};
			\node at (176, 640) {$\circ$};
			\node at (128, 640) {$\circ$};
			\node at (336, 640) {$\circ$};
		\end{tikzpicture}
	\end{subfigure}
	\begin{subfigure}[t]{.48\textwidth}
	\centering
		\begin{tikzpicture}[x=1pt,y=1pt,scale=.6]
			\draw [thick] (336, 640) circle[radius=48];
			\draw [thick] (128, 640) circle[radius=32];
			\draw [thick] (64, 640)  -- (96, 640);
			\draw [dotted] (192, 640)  -- (288, 640);
			\draw [RoyalBlue,dotted] (240, 720)  .. controls (256, 560) and (344, 552) .. (384, 592)  .. controls (424, 632) and (416, 720) .. (240, 720);
			\draw [RoyalBlue] (176, 640)  .. controls (128, 560) and (80, 600) .. (80, 640)  .. controls (80, 680) and (128, 720) .. (176, 640);
			\draw [RoyalBlue,dotted] (176, 640)  .. controls (128, 528) and (48, 584) .. (48, 640)  .. controls (48, 696) and (128, 752) .. (240, 720);
			\draw [RoyalBlue] (336, 640)  -- (240, 720);
			\draw [RoyalBlue] (240, 720)  -- (176, 640);
			\draw [RoyalBlue] (176, 640)  -- (128, 640);
			\draw [thick] (64, 640)  arc[start angle=180, end angle=360, x radius=64, y radius=-64];
			\draw [dotted] (192, 640)  arc[start angle=0, end angle=180, x radius=64, y radius=-64];

			\node at (64, 640) {$\bullet$};
			\node at (96, 640) {$\bullet$};
			\node at (192, 640) {$\bullet$};
			\node at (288, 640) {$\bullet$};

			\node [white] at (240, 720) {$\bullet$};
			\node [white] at (176, 640) {$\bullet$};
			\node [white] at (128, 640) {$\bullet$};
			\node [white] at (336, 640) {$\bullet$};
			\node at (240, 720) {$\circ$};
			\node at (176, 640) {$\circ$};
			\node at (128, 640) {$\circ$};
			\node at (336, 640) {$\circ$};
		\end{tikzpicture}
	\end{subfigure}
	\caption{On top, a ribbon graph $G$ of type $(0,4)$. On the bottom, two examples of support sets of $G$: the black (resp. blue) edges are the ones in the support set (resp. dual support set), and the black (resp. blue) dotted edges are the ones not in the support set (resp. dual support set). The graph $G$ has five different support sets $S_1 = \{e_1,e_2,e_5,e_6\}$ (on the left), $S_2 = \{e_1,e_2,e_3,e_6\}$ (on the right), $S_3 = \{e_1,e_2,e_4,e_6\}$, $S_4 = \{e_1,e_3,e_5,e_6\}$, and $S_5 = \{e_1,e_4,e_5,e_6\}$.}
	\label{fig:supp:set}
\end{figure}
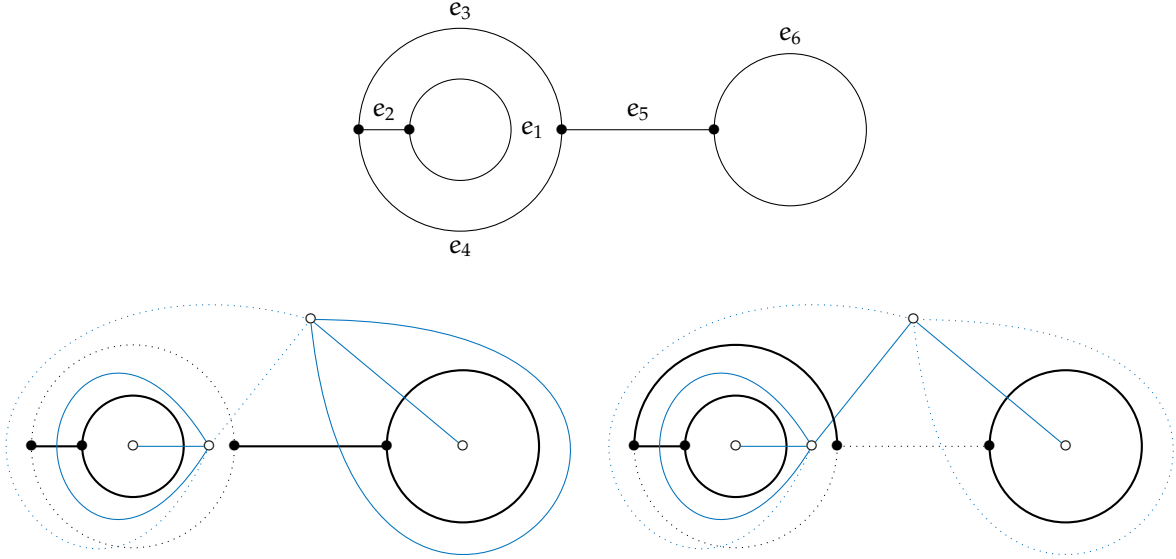

See Figure~\ref{fig:supp:set} for some examples of support sets.

\begin{defn}
	Let $G$ be a trivalent ribbon graph. For $L \in \RR_+^n$ and $\lambda$ a point of the cell closure $\overline{\mathfrak{Z}}_G(L)$ we define
	\[
		E[\lambda] \coloneqq \Set{e \in E_G | \lambda_e = 0}.
	\]
\end{defn}

\begin{lem} \label{lem:vertices:and:support:sets}
	Let $G$ be a trivalent ribbon graph of genus $g$ with $n$ faces.
	\begin{enumerate}
		\item[(A)] Let $L \in \RR_+^n$ be non-resonant and $\mathfrak{Z}_{G}(L)$ be a top-dimensional cell of the combinatorial moduli space $\mathcal{M}_{g,n}^{\textup{comb}}(L)$. If $\lambda = (\lambda_e)_{e \in E_{G}}$ is a vertex of the cell closure $\overline{\mathfrak{Z}}_{G}(L) \subset \RR_{+}^{E_{G}}$, then $E \setminus E[\lambda]$ is a support set.
		\item[(B)] Conversely, let $S \subset E_G$  a support set for $G$. Then there exists a non-resonant $L \in \RR_+^n$ and a vertex $\lambda$ of the cell closure $\overline{\mathfrak{Z}}_G(L)$ such that $S = E \setminus E[\lambda]$.
	\end{enumerate}
\end{lem}

\begin{lem}\label{lem:tangent:cone}
	Let $L \in \RR_+^n$ be non-resonant and $\mathfrak{Z}_G(L)$ be a top-dimensional cell of the combinatorial moduli space $\mathcal{M}_{g,n}^{\textup{comb}}(L)$. Then the tangent cones at any vertex of the cell closure $\overline{\mathfrak{Z}}_{G}(L)$ are simplicial. Furthermore, at a given vertex $\lambda$ the rays $r^{(e)}$ of the tangent cone are indexed by the edges $e \in E[\lambda]$ in such a way that
	\begin{equation} 
		\forall e' \in E[\lambda],
		\qquad
		r_{e'}^{(e)} = \delta_{e,e'}.
	\end{equation}
\end{lem}

\begin{proof}[Proof of Lemma~\ref{lem:vertices:and:support:sets}]
	The closure of the polytope is determined by inequalities $\ell_e \geq 0$ for each $e \in E_{G}$ and $n$ equalities of the form
	\[
		\sum_{e \in E_G^{(i)}} a_{i,e} \, \ell_{e} = L_i,
		\qquad
		i \in \set{1,\ldots,n},
	\]
	where $E_G^{(i)}$ is the set of edges around the $i$-th face and $a_{i,e} \in \set{1,2}$ is the multiplicity of the edge $e$ around this face. Now, for an arbitrary $S \subseteq E_{G}$, consider the inhomogeneous linear system of equations in the variables $(\ell_e)_{e \in E_{G}}$
	\begin{equation}\label{eqn:linear:system}
		\begin{cases}
			\ell_e = 0										& \text{for $e \in E_{G} \setminus S$},\\[.2ex]
			\sum_{e \in E_G^{(i)}} a_{i,e} \ell_{e} = L_i	& \text{for $i \in \set{1,\ldots,n}$}. 
	\end{cases}
	\end{equation}
    We claim that
    \begin{enumerate}
        \item the system~\eqref{eqn:linear:system} is invertible in $(\ell_e)_{e \in E_G}$ if and only if $S$ is a support set,
        \item if $L_i$ is non-resonant and $S$ is a support set then the solution of the system is such that $\ell_e > 0$ for $e \in S$.
    \end{enumerate}
    Let us prove the first claim. The matrix associated to the family of equations $\sum_{e \in E_G^{(i)}} a_{i,e} \ell_{e} = L_i$ is the incidence matrix of the graph $G^\ast_S$. In order for the incidence matrix to be invertible there must be as many edges as vertices in each connected component of $G^\ast_S$, hence a unique cycle. Next, degree one vertices does not play any role in the invertibility (the edge length $\ell_e$ adjacent to a the vertex dual to the $i$-th face must be set to $\ell_e = L_i$). Hence one can get rid of the tree part of the graph. Finally the incidence matrix of a cycle is invertible if and only if it has odd length. Indeed if the cycle is even then the alternating vector $(1,-1,1,-1,\ldots,1,-1)$ belongs to the kernel. Whereas if the cycle is odd, the alternating vector $(1,-1,1,-1,\ldots,1)$ is mapped to twice a basis vector and the matrix is invertible by cyclic symmetry. This concludes the proof that $S$ must be a support set.

    \medskip

    Now let us prove the second claim. Let $(L_i)_{i \in \{1,\ldots,n\}} \in \RR^n$ and $(\ell_e)_{e \in E_G}$ be the corresponding solution in~\eqref{eqn:linear:system}. Assume that for $e_0 \in S$ we have $\ell_{e_0} = 0$. Then $G^\ast_{S \setminus \{e_0\}}$ contains at least one tree component. Let $S'$ be the vertices of a tree component of $G^\ast_{S \setminus \{e_0\}}$ and $S' = S'_1 \sqcup S'_2$ a bipartition of $S'$ (\textit{i.e.} vertices in $S'_1$ are only adjacent to $S'_2$). Then $\sum_{i \in S'_1} L_i = \sum_{i \in S'_2} L_i$ and hence $L_i$ is resonant. This concludes the proof of the second claim.
			
	\medskip

    We turn to the proof of the first part (A) of the lemma. Assume that $\lambda$ is a vertex and $S \coloneqq \{e \in E_G | \lambda_e > 0\}$ is such that the system~\eqref{eqn:linear:system} admits a unique solution. Necessarily $\# S \leq n$. If $S$ is not contained in a support set then the graph $G^\ast_S$ contains an even cycle and the solution of~\eqref{eqn:linear:system} is not unique. Let us suppose by contradiction that $\#S < n$ and let $S' \supset S$ be a support set. Then $\lambda$ is a solution of the system~\eqref{eqn:linear:system} with the subset of edges $S'$. It contradicts our second claim that states that $\lambda_e$ would be positive for all $e \in S'$.

	\medskip

	For the converse --- part (B) of the lemma --- pick a support set and a positive vector $(\ell_e)_{e \in S}$. Because the system is bijective there is no further inequality $\ell_e \geq 0$ that can be set to an equality $\ell_e = 0$. In other words, completing the vector $(\ell_e)_{e \in S}$ with zeros, we obtain a vertex. Now if the positive values are generic enough the associated face lengths $L_i$ are non-resonant.
\end{proof}

\begin{proof}[Proof of Lemma~\ref{lem:tangent:cone}]
   	Let $L_i$ be non-resonant. Let $\lambda = (\lambda_e)_{e \in E_G}$ be a vertex of $\overline{\mathfrak{Z}}_G(L)$ and $S[\lambda] = \{e \in E_G| \lambda_e > 0\}$. By Lemma~\ref{lem:vertices:and:support:sets} $S$ is a support set. 
	The invertibility of the homogeneous linear system underlying~\eqref{eqn:linear:system} shows that the projection map from the tangent space
	\[
		T_{\lambda}\overline{\mathfrak{Z}}_{G}(L)
		=
		\bigcap_{i=1}^{n} \Big\{\ell \in \RR^{E_{G}} \,\,\Big|\,\,  \sum_{e \in E_G^{(i)}} a_{i,e} \ell_{e} = 0 \Big\}
	\]
	to $\RR^{E_G\setminus S[\lambda]}$, is an isomorphism. Then, the preimage of the canonical basis gives a basis of $T_{\lambda}\overline{\mathfrak{Z}}_{G}(L)$ that are rays of the tangent cone at $\lambda$, proving the last part of the lemma. 
\end{proof}

It will be useful for the study of integrability of $\mathscr{B}_{g,n}^{{\rm comb}}$ to cover $\mathfrak{Z}_G(L)$ by neighbourhoods of the vertices.

\begin{lem}\label{lem:cover}
	Let $\mathfrak{Z}_G(L)$ be a top-dimensional cell, and denote by $\Lambda_G(L)$ the set of vertices of its closure. There exists $\epsilon \in (0,1)$, depending only on $g,n$ and $L$, such that
	\begin{equation}
	\label{nun1un}	\mathfrak{Z}_{G}(L) = \bigcup_{\lambda \in \Lambda_G(L)} U_{G,L,\lambda},
		\qquad
		U_{G,L,\lambda} = \Set{ \ell \in \mathfrak{Z}_{G}(L) \,\, | \,\, \forall e \in S[\lambda] \quad \ell_{e} > \epsilon }.
	\end{equation}
\end{lem} 

\begin{proof}
	Let $\ell \in \overline{\mathfrak{Z}}_{G}(L)$. Since $\overline{\mathfrak{Z}}_{G}(L)$ is a polytope, there exists $t \in [0,1]^{\Lambda_{G}(L)}$ such that
	\[
		\ell = \sum_{\lambda \in \Lambda_G(L)} t_{\lambda} \, \lambda,
		\qquad
		\sum_{\lambda \in \Lambda_G(L)} t_{\lambda} = 1.
	\]
	In particular, there exists $\lambda_0 \in \Lambda_G(L)$ such that $t_{\lambda_0} \geq \frac{1}{\#\Lambda_G(L)}$. So, for any $e \in S[\lambda_0]$, we have
	\[
		\ell_{e} \geq \frac{\min_{e \in S[\lambda]} \lambda_e}{\#\Lambda_G(L)}.
	\]
	As vertices are characterised by their support set (which are certain subsets of $E_G$ of cardinality $n$), $\#\Lambda_G(L)$ is bounded by a constant $c$ depending only on $g,n$. For fixed $L \in \RR_{+}^n$, let $c' > 0$ (depending on $g,n$ and $L$) be the minimum of $\lambda_e$ over $e \in S[\lambda]$, $\lambda \in \Lambda_G(L)$ and $G$ trivalent ribbon graphs of type $(g,n)$. Equation~\eqref{nun1un} holds with $\epsilon \leq \frac{c'}{c}$, in particular we can take $\epsilon < 1$.
\end{proof}

\subsection{Main result}
\label{sec:main:result}

The proof of the main result Theorem~\ref{main:thm} will be decomposed in two intermediate results which we now state, and which are established in the next subsections. They require the following extension of the definition of ${\rm MF}_{\Sigma}^{\bullet}$ from Section~\ref{sec:param:mf} to unstable surfaces. When $\Sigma$ is a topological cylinder, \textit{i.e.} has type $(0,2)$, we set ${\rm MF}_{\Sigma}^{\bullet} = \mathbb{R}_{\geq 0}$, consisting of the real non-negative multiple of a boundary-homotopic curve. In that case the dimension is not given by $6g - 6 + 3n = 0$ but rather
\[
	\dim {\rm MF}_{\Sigma}^{\bullet} = 1. 
\]
When $\Sigma$ is a topological disk, \textit{i.e.}  has type $(0,1)$, we set ${\rm MF}_{\Sigma}^{\bullet} = \{0\}$, so that the dimension is $0$. And, if $\Sigma$ is a union of connected surfaces $(\Sigma_i)_i$, we set ${\rm MF}_{\Sigma} = \prod_{i} {\rm MF}_{\Sigma_i}^{\bullet}$.

\begin{prop} \label{p:integrability:formula}
	Let $g,n$ with $2g-2+n > 0$, $L \in \mathbb{R}_+^n$ be non-resonant, $G$ a trivalent ribbon graph of type $(g,n)$ and $\lambda$ a vertex of the cell closure $\overline{\mathfrak{Z}}_G(L) \subset \mathcal{M}_{\Sigma}^{{\rm comb}}(L)$. Let $\epsilon$ and $U_{G,L,\lambda}$ as in Lemma~\ref{lem:cover}. Then the integral
	\[
		\int_{U_{G,L,\lambda}} \big(\mathscr{B}^{{\rm comb}}_{g,n}\big)^s\dd\mu_{{\rm K}}
	\]
	converges if and only if
	\[
		s < \min_{\substack{E' \subseteq E[\lambda]\\ E' \not= \emptyset}} \hat{s}(G_{|E'}),
	\]
	where $E[\lambda]$ is the subset of edges of $G$ that have length $0$ at $\lambda$, and for any ribbon graph $\Gamma$ with underlying surface $\Sigma_{\Gamma}$ we defined:
	\begin{equation}\label{SGE}
	    \hat{s}(\Gamma) =  \frac{\# E_{\Gamma}}{\dim {\rm MF}_{\Sigma_{\Gamma}}^{\bullet}}.
	\end{equation}
\end{prop}

\begin{prop} \label{p:worse:divergent:subgraph}
	Let $g,n$ with $2g-2+n > 0$ and $(g,n) \neq (0,3)$, and $L \in \mathbb{R}_+^n$ non-resonant. The minimum $s_{g,n}^*$ of $\hat{s}(G_{|E'})$ over trivalent ribbon graphs $G$ of type $(g,n)$, over vertices $\lambda \in \Lambda_G(L)$ and non-empty subsets of edges $E' \subseteq E[\lambda]$, is given by
	\[
		s^*_{g,n} =
		\begin{dcases}
			2 & \text{if $g = 0$ and $n \in \set{4,5}$}  \\[2pt]
			\frac{4}{3} + \frac{2}{3}\frac{1}{\lfloor n/2 \rfloor - 2} & \text{if $g = 0$ and $n \geq 6$,} \\[2pt]
			2 & \textit{if $(g,n) = (1,1)$} \\[2pt]
			\frac{4}{3} & \text{if $g = 1$ and $n \geq 2$,} \\[2pt]
			1 + \frac{1}{3(2g - 3)} & \text{if $g \geq 2$ and $n = 1$,} \\[2pt]
			1 + \frac{1}{3(2g - 1)} & \text{if $g \geq 2$ and $n \geq 2$.}
		\end{dcases}
	\]
\end{prop}

\begin{proof}[Proof of Theorem~\ref{main:thm} assuming Propositions~\ref{p:integrability:formula} and~\ref{p:worse:divergent:subgraph}]
	By Lemma~\ref{lem:cover}, the space $\mathcal{M}_{g,n}^{\textup{comb}}(L)$ is covered by the finitely many open sets $U_{G,L,\lambda}$. Hence the integral of $(\mathscr{B}_{g,n}^{\rm comb})^s$ over $\mathcal{M}_{g,n}^{\textup{comb}}(L)$ diverges if and only if the integral over at least one $U_{G,L,\lambda}$ diverges. Now Proposition~\ref{p:integrability:formula} reformulates the divergence over $U_{G,L,\lambda}$ in terms of subgraphs and Proposition~\ref{p:worse:divergent:subgraph}	provides the smallest exponent $s^*_{g,n}$ above which one of the integrals is diverging.
\end{proof}

\subsection{Local integrability: proof of Proposition~\ref{p:integrability:formula}}
\label{sec:local:integrability}

Our starting point to prove Proposition~\ref{p:integrability:formula} is Proposition~\ref{prop:volume:rational:fnct}, writing $\mathscr{B}_{g,n}^{{\rm comb}}(\bm{G})$ as a linear combination of elementary rational functions. We first show that it suffices to analyse the integrability of these elementary rational functions. Then, we rely on Theorem~\ref{thm:sigma} proved in Appendix~\ref{app:integrability:lemma} to analyse the indices of convergence of the latters.

\medskip

For a fixed $G \in \mathcal{R}^{\textup{triv}}_{g,n}$, $L \in \RR_{+}^n$ and $s \in \RR_{+}$, we consider the integral of the $s$-th power of $\mathscr{B}_{g,n}^{{\rm comb}}$ over $U_{G,L,\lambda}$
\begin{equation}\label{eqn:integral:B:comb}
	\mathbb{I}_{G,L,\lambda}(s) =
			\int_{U_{G,L,\lambda}} \big(\mathscr{B}_{g,n}^{{\rm comb}}\big)^{s} \, \dd\mu_{\textup{K}}
	\in (0,+\infty].
\end{equation}
We will study its convergence by comparison with more elementary integrals, defined as follows.

\begin{defn}\label{def:theta}
	Let $G$ be a trivalent ribbon graph. For $\lambda$ a vertex of a cell closure $\overline{\mathfrak{Z}}_G(L)$. We define the linear map $\theta : \RR^{E[\lambda]} \to \RR^{E_{G}}$ as follows. Given $x \in \RR^{E[\lambda]}$, the vector $\theta(x) \in \mathbb{R}^{E_G}$ is the unique solution of the linear system of equations for $\ell = (\ell_e)_{e \in E_G}$ --- see the proof of Lemma~\ref{lem:vertices:and:support:sets}:
	\begin{equation} 
		\begin{cases}
			\ell_e = x_e & \text{if $e \in E[\lambda]$}, \\
			\sum_{e \in E_{G}^{(i)}} a_{i,e} \ell_e = L_i & i \in \set{1,\ldots,n}.
		\end{cases} 
	\end{equation}
\end{defn}

Given a ribbon graph $G$, a vertex $\lambda$ of $G$ and an elementary simplex $t$ (see Section~\ref{sec:explicit:bcomb}) we define the \emph{elementary integral} $\mathbb{J}_{G,L,\lambda,t}(s)$ as
\begin{equation}\label{eqn:J:lambda:Delta}
	\mathbb{J}_{G,L,\lambda,t}(s)
	\coloneqq
	\int_{(0,1]^{E[\lambda]}} \frac{\prod_{e \in E[\lambda]} \dd \ell_e}{\prod_{\rho \in R(t)[\lambda]}\big(\sum_{e \in E[\lambda]} \rho_{e} \ell_e\big)^{s}}
\in (0,+\infty],
\end{equation}
where $R(t)[\lambda]$ is the subset of rays of $R(t)$ vanishing at $\lambda$, \textit{i.e.} the subset of curves in $R(t)$ supported in $E[\lambda]$.

\begin{lem}\label{lem:int:iff:ele:int}
	Let $s > 0$. The integral $\mathbb{I}_{G,L,\lambda}(s)$ in~\eqref{eqn:integral:B:comb} converges if and only if for any elementary simplex $t$ the integral $\mathbb{J}_{G,L,\lambda,t}(s)$ in~\eqref{eqn:J:lambda:Delta} converges.
\end{lem}

\begin{proof}
	We first notice that, if $s \in (0,1)$, we can write
	\begin{equation}
		\big(\mathscr{B}_{g,n}^{{\rm comb}}(\bm{G})\big)^{s}
		\leq
		\max\big\{1,\mathscr{B}_{g,n}^{{\rm comb}}(\bm{G})\big\},
	\end{equation}
	so we can assume $s \geq 1$. Let $T_G$ be as in Proposition~\ref{prop:volume:rational:fnct} a simplicial decomposition of $Z_G$. We obtain from Proposition~\ref{prop:volume:rational:fnct} that
	\begin{equation}\label{eq:Bconv}
		\big(\mathscr{B}_{g,n}^{{\rm comb}}(\bm{G})\big)^{s} \leq c_1 \sum_{t \in T_G} \frac{1}{\prod_{\rho \in R(t)} \bigl( \ell_{\bm{G}}(\rho) \bigr)^{s}},
	\end{equation}
	where
	\[
		c_1 = (\# T_G)^{s-1} \cdot \left( \frac{\max_{t \in T_G} (\det t)^{-1}}{(6g-6+2n)!} \right)^s .
	\]
	We now integrate the inequality~\eqref{eq:Bconv} over $U_{G,L,\lambda}$. Integrating over the cell $\mathfrak{Z}_G(L)$ instead of the orbicell $\mathfrak{Z}_G(L)/\Aut(G)$ we find
	\begin{equation}
	\mathbb{I}_{G,L,\lambda}(s) = \int_{U_{G,L,\lambda}}  \big(\mathscr{B}_{g,n}^{{\rm comb}}\big)^{s} d \mu_{\rm K}
		\leq
		c_1 \sum_{t \in  T} 
			\int_{U_{G,L,\lambda}} \frac{\dd\mu_{{\rm K}}(\bm{G})}{\prod_{\rho \in R(t)} \bigl( \ell_{\bm{G}}(\rho) \bigr)^{s}}.
	\end{equation}
	From the definition of $U_{G,L,\lambda}$ in Lemma~\ref{lem:cover} we see that assuming $\epsilon<1$ (otherwise we can take $\epsilon=1$ in the following equation)
	\begin{equation}\label{eq:rm:pos:rays}
		\int_{U_{G,L,\lambda}}
			\frac{\dd\mu_{\textup{K}}(\bm{G})}{\prod_{\rho \in R(t)} \left(\ell_{\bm{G}}(\rho)\right)^s}
		\leq
		\frac{1}{\epsilon^{s (6g-6+2n)}} \int_{U_{G,L,\lambda}}
			\frac{\dd\mu_{\textup{K}}(\bm{G})}{\prod_{\rho \in R(t)[\lambda]} \left(\ell_{\bm{G}}(\rho)\right)^s}.
	\end{equation}
	Observe there exists $c_2 > 0$ such that $U_{G,L,\lambda} \subset \theta\big((0,c_2)^{E[\lambda]}\big)$. Besides, for the vertex $\lambda$, the Kontsevich measure on $\mathfrak{Z}_G(L)$ is the restriction onto $\mathfrak{Z}_G(L)$ of the pushforward via $\theta$ of a measure of the form
	\begin{equation}\label{eqn:Kont:measure:powers:2}
		2^{k} \prod_{e \in E[\lambda]} \dd \ell_e
	\end{equation}
	for some $k \in \ZZ$ that is bounded in absolute value by a constant depending only on $g$ and $n$, see~\cite{Kontsevich}.  Therefore,
	\begin{equation}
	\begin{split}\label{eqn:upper:bound:Ign}
	 \mathbb{I}_{G,L,\lambda}(s)
		& \leq c_3
			\sum_{t \in T_G}
			\int_{(0,c_2)^{E[\lambda]}}
				\frac{\prod_{e \in E[\lambda]} \dd \ell_e}{\prod_{\rho \in R(t)[\lambda]} \big(\sum_{e \in E_{G}} \rho_{e} \theta_{e}(\ell)\big)^{s}} \\
	\\
		& \leq c_4 \sum_{t \in T_G}  \mathbb{J}_{G,L,\lambda,t}(s), 
	\end{split} 
	\end{equation} 
	for some constant $c_4$ depending on $\epsilon$ and $s$. We have used homogeneity of the integrand to get the last line of~\eqref{eqn:upper:bound:Ign} as for $e\in E[\lambda]$ we have $\theta_{e}(\ell)=\ell_{e}$ and $R(t)[\lambda]$ is always a linear combination of edges in $E[\lambda]$. We deduce that the convergence of all elementary integrals imply the one of $\big(\mathscr{B}_{g,n}^{{\rm comb}}(\bm{G})\big)^{s}$ over $U_{G,L,\lambda}$.
	  
	\medskip
	  
	We then search for a lower bound for \eqref{eqn:integral:B:comb}. We get it by taking into consideration only the $s$-th power of the contribution of a single elementary simplex $t \subseteq Z_G$. It suffices to integrate this contribution over one cell $\mathfrak{Z}_G(L)$ instead of an orbicell. For the lower bound we can also integrate over a single set $U_{G,L,\lambda}$ of the cover, and in fact replace it by a set of the form $\theta\big((0,c_5)^{E[\lambda]}\big)$ which is strictly contained in $U_{G,L,\lambda}$ for a $c_5 > 0$ chosen small enough, depending only on $g,n,L$. Recalling \eqref{eqn:Kont:measure:powers:2} for the Kontsevich measure against which we integrate on $U_{G,L,\lambda}$ and using again homogeneity, we find there exists $c_6 > 0$ depending only on $g,n,L$ such that
	\begin{equation}\label{eqn:lower:bound:Ign}
		c_6	\max_{\lambda \in \Lambda_{G}(L)}
			\max_{t \in T_G} \,
				\mathbb{J}_{G,L,\lambda,t}(s)
		\leq
		\mathbb{I}_{G,L,\lambda}(s).
	\end{equation}
	By the two inequalities \eqref{eqn:upper:bound:Ign}-\eqref{eqn:lower:bound:Ign} we obtain the claim.
\end{proof}

\begin{lem}
	The integral $\mathbb{I}_{G,L,\lambda}(s)$ converges if and only if
	\[
		s < \min_{\substack{E' \subseteq E[\lambda] \\ E' \not= \emptyset}} \hat{s}(G|_{E'}),
	\]
	where $\hat{s}(G_{|E'})$ is given by \eqref{SGE}.
\end{lem}

\begin{proof}
	Let $s$ be strictly smaller that the minimum. We shall prove that $\mathbb{I}_{G,L,\lambda}(s)$ converges. By Lemma~\ref{lem:int:iff:ele:int} it suffices to prove that all elementary integrals $\mathbb{J}_{G,L,\lambda,t}(s)$ converge. By Theorem~\ref{thm:sigma} from Appendix~\ref{app:integrability:lemma} it suffices to show that
	\begin{equation} \label{eq:ineq:converges}
		\frac{1}{s}
		<
		\max_{\substack{E' \subseteq  E[\lambda] \\ E \neq \emptyset}}
		\frac{\# \{\rho \in R(t)[\lambda]\,\,|\,\, {\rm supp}\,\,\rho \subseteq E'\}}{\# E'}.
	\end{equation}
	where ${\rm supp}\,\rho$ is the set of edges involved in the ray $\rho$. As the rays contained in $E'$ must be linearly independent, there are at  most $\dim {\rm MF}^{\bullet}_{\Sigma'}$ of them, where $\Sigma'$ is the surface underlying $G_{|E'}$. This is also true when $\Sigma'$ has unstable components, thanks to our special definition. In other words, the right-hand side in~\eqref{eq:ineq:converges} is smaller than $1/\min \hat{s}(G_{|E'})$. 
	 
	\medskip

	We shall now prove that for $s = \min \hat{s}(G_{|E'})$ the integral $\mathbb{I}_{G,L,\lambda}(s)$ diverges. Again by Lemma~\ref{lem:int:iff:ele:int} it suffices to exhibit an elementary simplex $t$ in $Z_{G}$ such that the associated elementary integral $\mathbb{J}_{G,L,\lambda,t}(s)$ diverges. For this purpose let $E'$ be such that $\hat{s}(G_{|E'}) = s$ and $\Sigma'$ the geometric realisation of $G_{|E'}$. By definition of $\hat{s}(G_{|E'})$ we can find $\dim {\rm MF}^\bullet_{\Sigma'}$ independent rays supported in $E' \subseteq E[\lambda]$. This subset of rays can be completed into an elementary simplex of ${\rm MF}_{\Sigma}$ ($\Sigma$ is the original surface of type $(g,n)$), by including curves whose length remain bounded away from $0$ at $\lambda$. By  the last part of Theorem~\ref{thm:sigma}, the integral over this elementary simplex diverges.
\end{proof}

\subsection{Identifying the worst diverging subgraph: proof of Proposition~\ref{p:worse:divergent:subgraph}}
\label{sec:worst:divergence}

Let $g \geq 0$ and $n > 0$ such that $2g - 2 + n > 0$ and $(g,n) \neq (0,3)$, and $L \in \mathbb{R}_+^n$ non-resonant. We want to compute the minimum $s_{g,n}^*$ of $\hat{s}(\Gamma)$ defined in \eqref{SGE}, over trivalent ribbon graphs $G$ of type $(g,n)$, over vertices of $\lambda$ of $\overline{\mathfrak{Z}}_{G}(L)$, and over non-empty subgraphs $\Gamma \subseteq G_{|E[\lambda]}$. Let $\hat{s}_{\lambda}^*$ be that minimum for fixed $G,\lambda$. To prove Proposition \ref{p:worse:divergent:subgraph}, we first reduce the computation of $s_{g,n}^*$ to the problem of finding the \textit{worst diverging relevant subgraph}. The study of the various integrability ranges as $g$ and $n$ vary is cut into pieces in Lemmata \ref{g0starlem} to \ref{lem:g2:n1}. We start by the following elementary observation. 

\begin{lem} \label{leun}
	For fixed $G$ and $\lambda$, $\hat{s}_{\lambda}$ is also the minimum of $\hat{s}(\Gamma)$ over non-empty connected subgraph $\Gamma \subseteq G_{|E[\lambda]}$ without univalent vertices (if there are no such subgraphs, $\hat{s}_{\lambda} = +\infty$). For such a $\Gamma$, denoting $(g_{\Gamma},n_{\Gamma})$ its type and $v_{\Gamma}^{(2)}$ its number of bivalent vertices, we must have $v_{\Gamma}^{(2)} \geq n_{\Gamma}$ and
	\begin{equation} \label{shatexp}
		\hat{s}(\Gamma) = \begin{dcases} v_{\Gamma}^{(2)} & {\rm if}\,\,(g_{\Gamma},n_{\Gamma}) = (0,2), \\[2pt] 1 + \frac{v_{\Gamma}^{(2)}}{6g_{\Gamma} - 6 + 3n_{\Gamma}} & {\rm otherwise}. \end{dcases}
	\end{equation}
\end{lem}

\begin{proof}
	If $G_{|E[\lambda]}$ is a forest of trees, so must be $\Gamma$ for any non-empty subgraph $\Gamma \subset G_{|E[\lambda]}$, and so $|\Gamma|$ is a union of topological disks. Accordingly, ${\rm MF}_{|\Gamma|}^{\bullet}$ has dimension $0$, leading to $\hat{s}(\Gamma) = +\infty$, hence $\hat{s}_{\lambda}^* = +\infty$. Now assume it is not the case, and let $\Gamma$ realising the equality $\hat{s}(\Gamma) = \hat{s}^*_{\lambda}$.

	\medskip

	Assume that $\Gamma$ has a univalent vertex with incident edge $e$. Then $e$ cannot be the only edge of $\Gamma$, otherwise  $|\Gamma|$ would have type $(0,1)$ and this is already ruled out. Thus, $\Gamma' = \Gamma \setminus \{e\}$ is non-empty. As $|\Gamma'|$ and $|\Gamma|$ are homeomorphic, we have $\dim {\rm MF}^{\bullet}_{|\Gamma'|} = \dim {\rm MF}_{|\Gamma|}^{\bullet}$ but $\#E_{\Gamma'} < \# E_{\Gamma}$, therefore $\hat{s}(\Gamma') < \hat{s}(\Gamma)$, contradicting minimality. Therefore, $\Gamma'$ cannot contain a univalent vertex.

	\medskip

	Now assume that $\Gamma$ is not connected. Denote $(\Gamma_i)_{i = 1}^k$ its connected components, $d_i = \dim {\rm MF}_{|\Gamma_i|}^{\bullet}$ and $e_i = \# E_{\Gamma_i}$. We have
	\[
		\hat{s}(\Gamma) = \frac{\sum_{i = 1}^k e_i}{\sum_{i = 1}^k d_i},\qquad\qquad \hat{s}(\Gamma_i) = \frac{e_i}{d_i}.
	\]
	Up to relabelling we can assume that $\frac{e_1}{d_1} \leq \frac{e_i}{d_i}$ for any $i \in \{1,\ldots,k\}$. This can be written $e_1 d_i \leq e_i d_1$ and summing over $i$ we deduce that $\hat{s}(\Gamma_1) = \frac{e_1}{d_1} \leq \hat{s}(\Gamma)$. Therefore, $\Gamma_1$ is a connected and minimising subgraph.

	\medskip

	By Lemma~\ref{lem:vertices:and:support:sets}, no face of $G$ is bordered by edges only in $E[\lambda]$. Hence, faces of $G_{|E[\lambda]}$ cannot be faces of $G$. In particular, around each face of $G_{|E[\lambda]}$ there should be at least one vertex which is incident to an edge in $S[\lambda] = E_{G} \setminus E[\lambda]$ (thus having positive length at $\lambda$) and pointing towards this face. As this vertex is trivalent in $G$, it must have valency $1$ or $2$ in $G_{|E[\lambda]}$. A connected minimising subgraph $\Gamma \subseteq G_{|E[\lambda]}$ is obtained by erasing further edges from $G_{|E[\lambda]}$, and as we know that $\Gamma$ cannot contain univalent vertices, we deduce that the erasing procedure will create at least one bivalent vertex per face of $\Gamma$, \textit{i.e.} $v_{\Gamma}^{(2)} \geq n_{\Gamma}$.

	\medskip

	Now let $\Gamma$ be an arbitrary non-empty connected ribbon graph without univalent vertices, with vertices of valency $2$ (their number is denoted $v_{\Gamma}^{(2)}$) or $3$. If $\Gamma$ has type $(0,2)$, all vertices must be bivalent and be aligned on a circle separating the two faces, therefore $\hat{s}(\Gamma) = v_{\Gamma}^{(2)}$. If $\Gamma$ has type $(g_{\Gamma},n_{\Gamma}) \neq (0,2)$, all bivalent vertices must be incident to two distinct edges. If we erase the bivalent vertices, we obtain a trivalent ribbon graph of the same type, hence having exactly $6g_{\Gamma} - 6 + 3n_{\Gamma}$ edges. Coming back to $\Gamma$ we obtain
	\[
		\# E_{\Gamma} = 6g_{\Gamma} - 6 + 3n_{\Gamma} + v_{\Gamma}^{(2)}.
	\]
	Together with ${\rm MF}_{|\Gamma|}^{\bullet}$ is  $6g_{\Gamma} - 6 + 3n_{\Gamma}$, we obtain the desired formula.
\end{proof}

\begin{defn}
	A graph $\Gamma \subseteq G$ is \textit{relevant} if it is connected, has no univalent vertices, has at least as many bivalent vertices as faces. We say that $\Gamma$ is a \textit{vanishing subgraph} if $\Gamma \subseteq G_{|E[\lambda]}$ for some trivalent ribbon graph $G$, some non-resonant $L$ and some vertex $\lambda$ of $\overline{\mathfrak{Z}}_{G}(L)$.
\end{defn}
  
Our strategy to compute $s^*_{g,n}$ will consist in exhibiting certain relevant subgraphs $\Gamma_{h,k}$ of type $(h,k)$, that we can realise as vanishing subgraphs in a ribbon graph $G_{g,n}$ of type $(g,n)$. We will see that at least one such subgraph exist for each $(g,n)$, which by Lemma~\ref{leun} implies that $s^*_{g,n} < +\infty$ and only relevant subgraphs have to be discussed. We will then justify that our examples of subgraphs provide the minimal value of $\hat{s}$ for fixed $(g,n)$, thus giving access to $s^*_{g,n}$ with help of \eqref{shatexp}.

\medskip

If $h = 0$ and $k \geq 2$ or $h = 1$ and $k \geq 1$, we introduce  $\Gamma_{h,k}$ as described in Figure~\ref{fig:subgraphs}. It appears as a vanishing subgraph in a trivalent ribbon graph $G_{h,2k}$. Since $(g_{\Gamma_{0,k}}, n_{\Gamma_{0,k}}, v^{(2)}_{\Gamma_{0,k}})=(0,k,k)$ and $(g_{\Gamma_{1,k}}, n_{\Gamma_{1,k}}, v^{(2)}_{\Gamma_{1,k}})=(1,k,k)$, we have from \eqref{shatexp}
\begin{equation}\label{sok8}
	\hat{s}(\Gamma_{0,k})
	=
	\begin{dcases}
		\,\,2 & \text{if $k = 2$} \\[2pt]
		\frac{4k - 6}{3k - 6} & \text{if $k \geq 3$}
		\end{dcases},\qquad\qquad\qquad \hat{s}(\Gamma_{1,k}) = \frac{4}{3}.
\end{equation}
If in the above $G_{h,2k}$ we apply the substitution of Figure~\ref{fig:subgraph:plus}, we obtain another ribbon graph $G_{h,2k + 1}$ containing $\Gamma_{h,k}$ as a vanishing subgraph. For $h \geq 2$, it will be sufficient to consider the graphs $\Gamma_{h,1}$ of Figure~\ref{fig:subgraphs:tilde}, which can be realised as vanishing subgraph of $G_{h,k}$ for any $k \geq 2$. They have $1$ bivalent vertex, genus $h$ and $1$ face, hence
\[
	\hat{s}(\Gamma_{h,1}) = \frac{6h - 2}{6h - 3}.
\]
Note that setting $h = 1$ in this formula gives the value $\frac{4}{3}$, which matches the value of $\hat{s}(\Gamma_{1,1})$. This squares with the fact that the construction of Figure~\ref{fig:subgraphs:tilde} in the case $h = 1$ gives the same result as $\Gamma_{1,1}$ described in Figure~\ref{fig:subgraphs}.

\begin{figure}[h!]
\begin{center}
	\begin{tikzpicture}[x=1pt,y=1pt,scale=1.3]
		\draw(399.6172, 759.468) -- (399.5722, 751.957) -- (399.5722, 751.957) -- (399.5722, 751.957);
		\draw[shift={(192, 768)}, xscale=1.5, YellowOrange](0, 0) -- (128, 0) -- (128, 0) -- (128, 0);
		\draw[YellowOrange](400, 768) circle[radius=16];
		\draw[YellowOrange](176, 768) circle[radius=16];
		\draw[YellowOrange](224, 768) -- (224, 784);
		\draw[YellowOrange](224, 800) circle[radius=16];
		\draw[YellowOrange](272, 768) -- (272, 784);
		\draw[YellowOrange](272, 800) circle[radius=16];
		\draw[YellowOrange](352, 768) -- (352, 784);
		\draw[YellowOrange](352, 800) circle[radius=16];
		\draw(224, 816) -- (224, 808);
		\draw(272, 816) -- (272, 808);
		\draw(352, 816) -- (352, 808);
		\draw(160, 768) -- (168, 768) -- (168, 768);
		\draw[shift={(399.693, 751.665)}, xscale=0.0078, yscale=13.6989](0, 0) -- (-8, 0);
		\draw(304, 768) -- (304, 752);
		\draw(304, 744) circle[radius=8];
		\draw(176, 768) circle[radius=8];
		\draw(400, 768) circle[radius=8];
		\draw(224, 800) circle[radius=8];
		\draw(272, 800) circle[radius=8];
		\draw(352, 800) circle[radius=8];

		\node[YellowOrange] at (192, 768) {$\bullet$};
		\node[YellowOrange] at (224, 784) {$\bullet$};
		\node[YellowOrange] at (272, 768) {$\bullet$};
		\node[YellowOrange] at (224, 768) {$\bullet$};
		\node at (352, 808) {$\bullet$};
		\node at (272, 808) {$\bullet$};
		\node at (224, 808) {$\bullet$};
		\node[YellowOrange] at (352, 768) {$\bullet$};
		\node[YellowOrange] at (384, 768) {$\bullet$};
		\node at (168, 768) {$\bullet$};
		\node at (399.7234, 759.6492) {$\bullet$};
		\node[YellowOrange] at (272, 784) {$\bullet$};
		\node[YellowOrange] at (352, 784) {$\bullet$};
		\node at (304, 752) {$\bullet$};
		\node[white] at (399.6923, 751.6935) {$\bullet$};
		\node[white] at (304, 768) {$\bullet$};
		\node[white] at (160, 768) {$\bullet$};
		\node[white] at (272, 816) {$\bullet$};
		\node[white] at (224, 816) {$\bullet$};
		\node[white] at (352, 816) {$\bullet$};
		\node[YellowOrange] at (399.6923, 751.6935) {$\circ$};
		\node[YellowOrange] at (304, 768) {$\circ$};
		\node[YellowOrange] at (160, 768) {$\circ$};
		\node[YellowOrange] at (272, 816) {$\circ$};
		\node[YellowOrange] at (224, 816) {$\circ$};
		\node[YellowOrange] at (352, 816) {$\circ$};

		\node at (312, 800) {$\cdots$};

		\node at (176.14, 768.102) {\tiny${1}$};
		\node at (176, 756) {\tiny${2}$};
		\node at (224.127, 800.357) {\tiny$3$};
		\node at (224, 789) {\tiny$4$};
		\node at (272.025, 799.734) {\tiny$5$};
		\node at (272.025, 789) {\tiny$6$};
		\node at (351.871, 799.991) {\tiny${2k - 5}$};
		\node at (352, 789) {\tiny${2k - 4}$};
		\node at (399.841, 768.05) {\tiny${2k -3}$};
		\node at (400, 779) {\tiny${2k - 2}$};
		\node at (304, 744) {\tiny${2k -1}$};
		\node at (248, 744) {\tiny${2k}$};

		\node at (96, 768) {$\textcolor{YellowOrange}{\Gamma_{0,k}} \subseteq G_{0,2k}$};

		\draw[shift={(191.787, 655.798)}, xscale=1.5, YellowOrange] (0, 0) -- (128, 0) -- (128, 0) -- (128, 0);
		\draw[YellowOrange](399.7869, 655.798) circle[radius=16];
		\draw[YellowOrange](175.7869, 655.798) circle[radius=16];
		\draw[YellowOrange](223.7869, 655.798) -- (223.7869, 671.798);
		\draw[YellowOrange](271.7869, 655.798) -- (271.7869, 671.798);
		\draw[YellowOrange](351.7869, 655.798) -- (351.7869, 671.798);
		\draw[YellowOrange](351.7869, 687.798) circle[radius=16];
		\draw(223.7869, 687.798) circle[radius=8];
		\draw(271.7869, 687.798) circle[radius=8];
		\draw(223.7869, 703.798) -- (223.7869, 695.798);
		\draw(271.7869, 703.798) -- (271.7869, 695.798);
		\draw(351.7869, 703.798) -- (351.7869, 695.798);
		\draw(303.7869, 655.798) -- (303.7869, 639.798);
		\draw(303.7869, 631.798) circle[radius=8];
		\draw[YellowOrange](164.947, 644.0296) .. controls (172, 648) and (173.9984, 643.8997) .. (176.3353, 640.7927);
		\draw[YellowOrange](177.696, 639.0364) .. controls (179.295, 636.681) and (185.156, 635.2145) .. (187.063, 644.342);
		\draw(399.8753, 647.833) -- (399.8303, 640.322) -- (399.8303, 640.322) -- (399.8303, 640.322);
		\draw[shift={(399.951, 640.035)}, xscale=0.0078, yscale=13.6989](0, 0) -- (-8, 0);
		\draw[YellowOrange](223.7869, 687.798) circle[radius=16];
		\draw[YellowOrange](271.7869, 687.798) circle[radius=16];
		\draw(351.787, 687.798) circle[radius=8];
		\draw(399.787, 655.798) circle[radius=8];

		\node at (303.7869, 639.798) {$\bullet$};
		\node[YellowOrange] at (187.063, 644.342) {$\bullet$};
		\node[YellowOrange] at (164.947, 644.0296) {$\bullet$};
		\node at (399.9816, 648.015) {$\bullet$};
		\node[YellowOrange] at (191.7869, 655.798) {$\bullet$};
		\node[YellowOrange] at (223.7869, 671.798) {$\bullet$};
		\node[YellowOrange] at (271.7869, 655.798) {$\bullet$};
		\node[YellowOrange] at (223.7869, 655.798) {$\bullet$};
		\node at (351.7869, 695.798) {$\bullet$};
		\node at (271.7869, 695.798) {$\bullet$};
		\node at (223.7869, 695.798) {$\bullet$};
		\node[YellowOrange] at (351.7869, 655.798) {$\bullet$};
		\node[YellowOrange] at (383.7869, 655.798) {$\bullet$};
		\node[YellowOrange] at (271.7869, 671.798) {$\bullet$};
		\node[YellowOrange] at (351.7869, 671.798) {$\bullet$};
		\node[white] at (271.7869, 703.798) {$\bullet$};
		\node[white] at (223.7869, 703.798) {$\bullet$};
		\node[white] at (351.7869, 703.798) {$\bullet$};
		\node[white] at (399.9505, 640.059) {$\bullet$};
		\node[white] at (303.7869, 655.798) {$\bullet$};
		\node[YellowOrange] at (271.7869, 703.798) {$\circ$};
		\node[YellowOrange] at (223.7869, 703.798) {$\circ$};
		\node[YellowOrange] at (351.7869, 703.798) {$\circ$};
		\node[YellowOrange] at (399.9505, 640.059) {$\circ$};
		\node[YellowOrange] at (303.7869, 655.798) {$\circ$};

		\node at (312, 688) {$\cdots$};

		\node at (224.133, 687.678) {\tiny$1$};
		\node at (224, 677) {\tiny$2$};
		\node at (271.954, 687.83) {\tiny$3$};
		\node at (272, 677) {\tiny$4$};
		\node at (352.083, 687.859) {\tiny${2k - 5}$};
		\node at (352, 677) {\tiny${2k - 4}$};
		\node at (400.001, 655.864) {\tiny${2k -3}$};
		\node at (400, 667) {\tiny${2k - 2}$};
		\node at (303.932, 631.98) {\tiny${2k - 1}$};
		\node at (248, 632) {\tiny${2k}$};

		\node at (96.121, 655.871) {$\textcolor{YellowOrange}{\Gamma_{1,k}} \subseteq G_{1,2k}$};

	\end{tikzpicture}
	\caption{\label{fig:subgraphs}
		In orange: the graphs $\Gamma_{0,k}$ (for $k \geq 2$) and $\Gamma_{1,k}$ (for $k \geq 1$), emphasizing the bivalent vertices $\circ$. In black: the graph $G_{0,n}$ in which they are realised as a vanishing subgraph. The vertex $\lambda$ of $\overline{\mathfrak{Z}}_{G_{g,n}}(L)$ is identified by assigning to the black edges the length necessary to make up for the fixed perimeters $L$, and zero lengths to the orange edges. The only inequality imposed in the picture is $L_{2i - 1} < L_{2i}$ for all $i \in \{1,\ldots,k\}$. For pairwise distinct (\textit{a fortiori}, non-resonant) boundary lengths, these inequalities can always be satisfied up to relabelling the faces.} 
\end{center}
\end{figure}
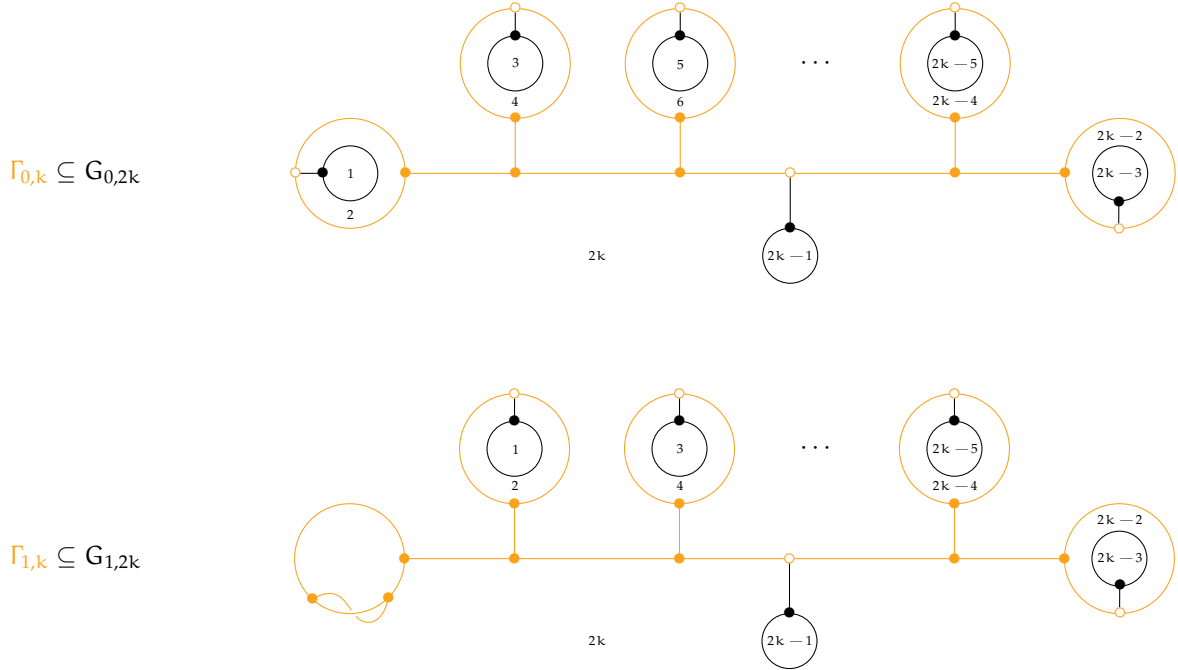

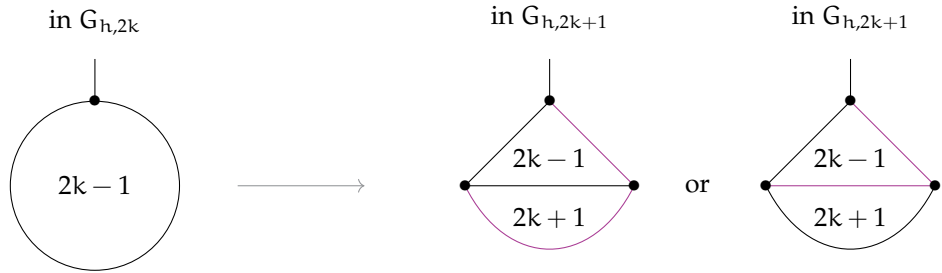
\begin{figure}[h!]
\begin{center}
	\begin{tikzpicture}[x=1pt,y=1pt,scale=1]
		\draw(128, 320) circle[radius=32];
		\draw(128, 368) -- (128, 352);
		\node at (128, 352) {$\bullet$};
		\node at (128, 320) {${2k - 1}$};
		\node at (128, 382) {in $G_{h,2k}$};

		\draw[Gray,->](182, 320) -- (230, 320);

		\draw[Mulberry](268, 320) .. controls (284, 288) and (316, 288) .. (332, 320);
		\draw[Mulberry](300, 352) -- (332, 320);
		\draw(268, 320) -- (300, 352);
		\draw(268, 320) -- (332, 320);
		\draw(300, 368) -- (300, 352);
		\node at (300, 352) {$\bullet$};
		\node at (268, 320) {$\bullet$};
		\node at (332, 320) {$\bullet$};
		\node at (300, 330) {${2k - 1}$};
		\node at (300, 308) {${2k + 1}$};
		\node at (300, 383.413) {in $G_{h,2k + 1}$};

		\node at (356,320) {or};

		\begin{scope}[xshift=4cm]
			\draw(268, 320) .. controls (284, 288) and (316, 288) .. (332, 320);
			\draw[Mulberry](300, 352) -- (332, 320);
			\draw(268, 320) -- (300, 352);
			\draw[Mulberry](268, 320) -- (332, 320);
			\draw(300, 368) -- (300, 352);
			\node at (300, 352) {$\bullet$};
			\node at (268, 320) {$\bullet$};
			\node at (332, 320) {$\bullet$};
			\node at (300, 330) {${2k - 1}$};
			\node at (300, 308) {${2k + 1}$};
			\node at (300, 383.413) {in $G_{h,2k + 1}$};
		\end{scope}
	\end{tikzpicture}
	\caption{\label{fig:subgraph:plus} Two substitutions of the lollipop with inner perimeter $L_{2k - 1}$ to obtain $G_{h,2k + 1}$ from $G_{h,2k}$ --- the subgraph $\Gamma_{h,k}$ is unchanged. The vertex of $\overline{\mathfrak{Z}}_{G_{h,2k + 1}}(L)$ is identified by assigning zero lengths to the purple edge --- on top of the edges of $\Gamma_{h,k}$ --- and other edge lengths in order to realise the boundary perimeters $L$. The structure of the rest of the graph still imposes $L_{2i - 1} < L_{2i}$ for $i \in \{1,\ldots,k\}$. Besides, we can consider the first substitution when  $L_{2k} < L_{2k + 1} + L_{2k - 1}$ is satisfied, while the second substitution is possible for $L_{2k} > L_{2k + 1} + L_{2k - 1}$. For non-resonant $L$, up to relabelling of the faces, we can always achieve one of these two sets of inequalities.}
\end{center}
\end{figure} 

\begin{figure}[h!]
\begin{center}
	\begin{tikzpicture}[x=1pt,y=1pt,scale=1]
		\draw[YellowOrange] (270.377, 327.713) circle[radius=17.4435];
		\draw (338.3077, 326.879)  -- (353.3997, 326.879);
		\draw[shift={(288.896, 327.231)}, xscale=-2.0767] (0, 0)  -- (-16, 0);
		\draw[Mulberry] (388.0136, 326.553) circle[radius=65.9697];
		\draw[Mulberry] (386.8745, 327.463) ellipse[x radius=49.2139, y radius=48.4322];
		\draw (435.8762, 326.689)  -- (453.5402, 326.634);
		\draw[Mulberry] (385.7907, 326.803) circle[radius=32.381];
		\draw (385.3277, 327.23) circle[radius=15.4976];
		\draw (400.8177, 326.369)  -- (418.6837, 326.369);
		\draw[YellowOrange] (258.982, 314.884)  .. controls (263.3947, 317.7307) and (267.3187, 316.8903) .. (270.754, 312.363);
		\draw[YellowOrange] (272.175, 309.689)  .. controls (275.741, 304.457) and (281.435, 306.804) .. (281.244, 313.755);
		\draw[YellowOrange] (237.8566, 327.035)  -- (252.9486, 327.035);
		\draw[YellowOrange] (221.1216, 327.5891) circle[radius=17.4435];
		\draw[YellowOrange] (209.7266, 314.6612)  .. controls (214.1393, 317.5079) and (218.0633, 316.6675) .. (221.4986, 312.1402);
		\draw[YellowOrange] (222.9196, 309.4662)  .. controls (226.4856, 304.2342) and (232.1796, 306.5812) .. (231.9886, 313.5322);
		\draw[YellowOrange] (188.925, 327.974)  -- (204.017, 327.974);
		\draw[YellowOrange] (142.883, 328.3968) circle[radius=17.4435];
		\draw[YellowOrange] (131.488, 315.4688)  .. controls (135.9007, 318.3155) and (139.8247, 317.4752) .. (143.26, 312.9478);
		\draw[YellowOrange] (144.681, 310.2738)  .. controls (148.247, 305.0418) and (153.941, 307.3888) .. (153.75, 314.3398);
		\draw[YellowOrange] (160.66, 328.143)  -- (175.752, 328.143);

		\node[YellowOrange] at (259.2288, 314.066) {$\bullet$};
		\node[YellowOrange] at (280.8606, 313.384) {$\bullet$};
		\node[YellowOrange] at (203.6023, 327.817) {$\bullet$};
		\node[YellowOrange] at (238.7634, 327.258) {$\bullet$};
		\node[white] at (288.2053, 327.259) {$\bullet$};
		\node[YellowOrange] at (288.2053, 327.259) {$\circ$};
		\node[white] at (337.2967, 327.016) {$\bullet$};
		\node[white] at (353.2967, 326.553) {$\bullet$};
		\node[white] at (417.2967, 326.553) {$\bullet$};
		\node[white] at (436.4432, 326.697) {$\bullet$};
		\node[white] at (453.9872, 326.697) {$\bullet$};
		\node[Mulberry] at (337.2967, 327.016) {$\circ$};
		\node[Mulberry] at (353.2967, 326.553) {$\circ$};
		\node[Mulberry] at (417.2967, 326.553) {$\circ$};
		\node[Mulberry] at (436.4432, 326.697) {$\circ$};
		\node[Mulberry] at (453.9872, 326.697) {$\circ$};
		\node at (401.0417, 326.404) {$\bullet$};
		\node[white] at (321.981, 327.247) {$\bullet$};
		\node[Mulberry] at (321.981, 327.247) {$\circ$};
		\node[YellowOrange] at (209.9734, 313.843) {$\bullet$};
		\node[YellowOrange] at (231.6052, 313.161) {$\bullet$};
		\node[YellowOrange] at (252.9467, 327.035) {$\bullet$};
		\node[YellowOrange] at (160.5245, 328.066) {$\bullet$};

		\node[YellowOrange] at (131.7345, 314.651) {$\bullet$};
		\node[YellowOrange] at (153.3663, 313.969) {$\bullet$};

		\node[YellowOrange] at (183.25, 327.7) {$\cdots$};

		\node at (304.282, 352.619) {$1$};
		\node at (385.7907, 326.803) {$2$};
		\node at ($(385.7907, 326.803) + (45:24)$) {$3$};
		\node at ($(385.7907, 326.803) + (45:42)$) {\rotatebox{45}{$\cdots$}};
		\node at ($(385.7907, 326.803) + (45:60)$) {$k$};

		\node at (70, 327) {$\textcolor{YellowOrange}{\Gamma_{h,1}} \subseteq G_{h,k}$};

		\draw [
		    YellowOrange,
		    decoration={
		        brace,
		        mirror,
		        raise=0.2cm,
		        amplitude=10pt
		    },
		    decorate
		] (128, 310) -- (284, 310) 
		node [pos=0.5,anchor=north,yshift=-0.55cm] {\small{$h$ times}}; 
	\end{tikzpicture}
	\caption{\label{fig:subgraphs:tilde}
		For $h \geq 2$, the ribbon graph $\Gamma_{h,1}$ (in orange) realised as a vanishing subgraph of a ribbon graph $G_{h,k}$ (for $k \geq 2$). The vertex $\lambda$ of $\overline{\mathfrak{Z}}_{G_{h,k}}(L)$ corresponds to assigning zero lengths to the edges in purple and in orange, and positive lengths to the black edges making up for the boundary perimeters $L$. Note that $L_{1}$ is the perimeter of the face obtained by travelling along all handles. The only inequality imposed by the picture is $L_{3} > L_2$, which for non-resonant $L$ can always be imposed up to relabelling the faces.} 
\end{center}
\end{figure}
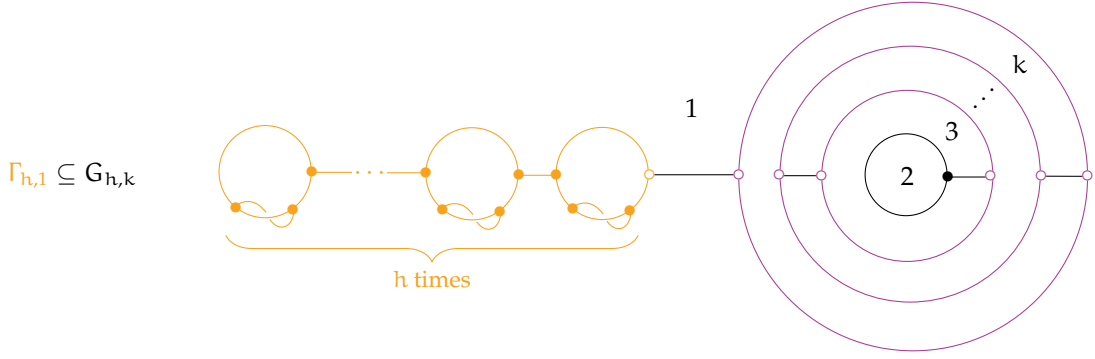

\begin{lem} \label{g0starlem}
	For non-resonant $L$, we have $s_{0,4}^* = s_{0,5}^* = 2$ and $s^*_{0,n} = \frac{4}{3} + \frac{2}{3 \lfloor n/2 \rfloor - 6}$ for $n \geq 6$.
\end{lem}

\begin{proof}
	Let $G$ be a trivalent ribbon graph of genus $g = 0$ with $n$ faces, $L \in \mathbb{R}_+^n$ non-resonant, and $\lambda$ a vertex in $\overline{\mathfrak{Z}}_{G}(L)$. By Lemma~\ref{leun}, it is enough to discuss relevant subgraphs $\Gamma \subseteq G_{|E[\lambda]}$. Recall that $\# E[\lambda] = 6g - 6 + 2n = 2n - 6$ here.

	\medskip

	If $n = 4$, we have $\# E[\lambda] = 2$. The only relevant subgraph of $G_{|E[\lambda]}$ that can occur is $\Gamma_{0,2}$, and it has $\hat{s}(\Gamma_{0,2}) = 2$. It is a vanishing subgraph in $G_{0,4}$, hence $s^*_{0,4} = 2$.

	\medskip

	If $n = 5$, we have $\#E[\lambda] = 3$. A relevant subgraph of $G_{|E[\lambda]}$ is either a $\Gamma_{0,2}$, or a graph with $3$ bivalent vertices on a circle separating two faces. The former yields $\hat{s} = 2$ and does occur as a vanishing subgraph in $G_{0,5}$ (see Figure~\ref{fig:subgraph:plus}), while the latter has $\hat{s} = 3$. Hence $s^*_{0,5} = 2$.

	\medskip

	We now assume $n \geq 6$. Let $\Gamma$ be a relevant subgraph of $G_{|E[\lambda]}$. It must have genus $0$, and $v^{(2)}_{\Gamma} \geq n_{\Gamma}$ by definition of relevance. If $\Gamma$ is not a $\Gamma_{0,2}$, we must have $n_{\Gamma} \geq 3$, and
	\[
		\hat{s}(\Gamma)
		=
		1 + \frac{v^{(2)}_{\Gamma}}{3n_{\Gamma} - 6}
		\geq
		\frac{4n_{\Gamma} - 6}{3n_{\Gamma} - 6}
		=
		\frac{4}{3} + \frac{2}{3n_{\Gamma} - 6},
	\]
	with equality if and only if $v^{(2)}_{\Gamma} = n_{\Gamma}$. The right-hand side is a decreasing function of $n_{\Gamma}$. We claim (and justify at the end) that $n \geq 2n_{\Gamma}$, and deduce that
	\begin{equation}\label{eqstar9n}
		s^*_{0,n} \geq \frac{4}{3} + \frac{2}{3 \lfloor \frac{n}{2} \rfloor - 6},
	\end{equation}
	and note that the right-hand side is strictly smaller than $2$. If $n$ is even, $\Gamma_{0,\frac{n}{2}}$ does occur as a vanishing subgraph of $G_{0,n}$, thus saturating the inequality \eqref{eqstar9n}. Hence $s_{0,n}^* = \frac{4}{3} + \frac{2}{3 \lfloor n/2 \rfloor - 6}$. If $n$ is odd, $\Gamma_{0,\frac{n - 1}{2}}$ occurs as a vanishing subgraph in $G_{0,n}$ (obtained from $G_{0,n-1}$ with the substitution of Figure~\ref{fig:subgraph:plus}), achieving the same result.

	\medskip

	It remains to justify that $G$ has at least twice as many faces as $\Gamma$, \textit{i.e.} $n \geq 2n_{\Gamma}$. It suffices to show it for the subgraph $\Gamma_{\lambda} \coloneqq G_{|E[\lambda]}$, since $\Gamma \subseteq G_{|E[\lambda]}$. We recall that, as edges in $E[\lambda]$ have zero lengths at $\lambda$, faces of $\Gamma_{\lambda}$ cannot be faces of $G$. In the proof of Lemma~\ref{leun}, we showed that for each face $\mathfrak{f}$ of $\Gamma_{\lambda}$ one can choose a bivalent vertex $v(\mathfrak{f})$ which came from a trivalent vertex in $G$ and so that the edge $e(\mathfrak{f})$ incident to $v(\mathfrak{f})$  in $G$ but not in $\Gamma_{\lambda}$ points towards $\mathfrak{f}$. We consider the graph $G' \subseteq G$ obtained by taking the union of $\Gamma_{\lambda}$ with the connected components of $G \setminus \Gamma_{\lambda}$ containing $\set{e(\mathfrak{f}) | \mathfrak{f} \in F_{\Gamma_{\lambda}}}$. Since $G'$ has genus $0$, inside each face of $\Gamma_{\lambda}$ one should find at least two faces of $G'$. Hence $n \geq n_{G'} \geq 2n_{\Gamma_{\lambda}}$.
\end{proof} 

\begin{lem}
	We have $s_{1,1}^* = 2$, and for $n \geq 2$ and $L$ non-resonant, $s_{1,n}^* = \frac{4}{3}$.
\end{lem}

\begin{proof}
	The case $(g,n) = (1,1)$ has already been treated by hand in Proposition~\ref{lem11s}, leading to $s_{1,1}^* = 2$. We now assume $g = 1$ and $n \geq 2$, and let $\Gamma$ a relevant subgraph of $G_{|E[\lambda]}$. If it has genus $0$, the proof of Lemma~\ref{g0starlem} shows that $\hat{s}(\Gamma) > \frac{4}{3}$. If it has genus $1$, using again $v_{\Gamma}^{(2)}\geq n_{\Gamma}$, we have $\hat{s}(\Gamma) \geq \frac{4}{3}$ with equality if and only if $v_{\Gamma}^{(2)} = n_{\Gamma}$. The graph $\Gamma_{1,\frac{n}{2}}$ if $n$ is even, or $\Gamma_{1,\frac{n - 1}{2}}$ if $n$ is odd, is a vanishing subgraph of $G_{1,n}$ and achieves the equality. Hence $s_{1,n}^* = \frac{4}{3}$.
\end{proof}
 
\begin{lem}
	For $g,n \geq 2$ and $L$ non-resonant, we have $s_{g,n}^* = 1 + \frac{1}{3(2g - 1)}$. 
\end{lem}

\begin{proof}
	Let $\Gamma$ be a relevant subgraph of some $G_{|E[\lambda]}$ with $G$ of type $(g,n)$. Then $\Gamma$ has genus $g_{\Gamma} \in \{0,\ldots,g\}$ and $v_{\Gamma}^{(2)} \geq n_{\Gamma}$. If $g_{\Gamma} \in \{0,1\}$, we know from the proof of the previous lemmata that $\hat{s}(\Gamma) \geq \frac{4}{3}$. If $g_{\Gamma} \geq 2$, using $v_{\Gamma}^{(2)} \geq n_{\Gamma}$, we get from \eqref{shatexp} the lower bound
	\[
		\hat{s}(\Gamma) \geq \frac{6g_{\Gamma}  - 6 + 4n_{\Gamma}}{6g_{\Gamma} - 6 + 3n_{\Gamma}}.
	\]
	The right-hand side is an increasing function of $n_{\Gamma}$ and a decreasing function of $g_{\Gamma}$. We can therefore lower bound it by its value at $(g_{\Gamma},n_{\Gamma}) = (g,1)$, which is
	\[
		\hat{s}(\Gamma) \geq \frac{6g - 2}{6g - 3} = 1 + \frac{1}{3(2g - 1)}.
	\]
	Equalities are achieved with the graph $\Gamma_{g,1}$, which is realised as a vanishing subgraph of $G_{g,n}$ (see Figure~\ref{fig:subgraphs:tilde}). Hence $s^*_{g,n} = 1 + \frac{1}{3(2g - 1)}$.
\end{proof}
 
\begin{lem}\label{lem:g2:n1}
	For $g \geq 2$, $n = 1$ and $L$ non-resonant, we have $s_{g,1}^* = 1 + \frac{1}{3(2g - 3)}$.
\end{lem}

\begin{proof}
	Since $n = 1$, $G_{|E[\lambda]}$ is obtained from $G$ by removing a single edge. Either it is connected, has genus $g - 1$ and $2$ faces, or it has two connected components $\Gamma^{(i)}$ of genus $g^{(i)} < g$ with $1$ face, such that $g^{(1)} + g^{(2)} = g$. Both cases can be realised. In view of the proofs of previous lemmata, the connected situation gives a smaller value of $\hat{s}$. So, $s^*_{g,1} = \hat{s}(\Gamma_{g - 1,1})$, which takes the value $1 + \frac{1}{3(2g - 3)}$.
\end{proof}

\appendix

\section{An integrability lemma}
\label{app:integrability:lemma}

The aim of this appendix is to prove a general result about integrability of rational functions, see Theorem~\ref{thm:sigma} below. It is used in Section~\ref{sec:local:integrability} to prove Proposition~\ref{p:integrability:formula}.

\medskip

We consider a polynomial
\begin{equation}\label{eq:F:poly}
	P(x_1, \ldots, x_e) = \prod_{i=1}^d P_i(x_1, \ldots, x_e)
\end{equation}
which is the product of non-zero $d$ linear forms in the $e$ variables $x_1 \ldots, x_e$, having non-negative coefficients. We are interested in determining the values of $s > 0$ for which the integral
\begin{equation}\label{eq:F:integral}
	\mathbb{I}(P; s) \coloneqq \int_{(0,1]^e} \frac{\dd x_1 \cdots \dd x_e}{P(x_1, \ldots, x_e)^s},
\end{equation}
converges. Let us define
\[
	\hat{s}(P) \coloneqq \sup \left\{ t > 0\,\,\bigg|\,\, \int_{(0,1]^e} \frac{\dd x_1 \cdots \dd x_e}{P(x_1, \ldots, x_e)^t} < +\infty\right\}.
\]
Let $A = (A_{ij})_{\substack{1 \leq i \leq d \\ 1 \leq j \leq e}}$ be the variables/linear forms adjacency matrix, that is
\[
	A_{ij} = \left\{ \begin{array}{ll}
	1 & \text{if the coefficient of $x_j$ in $P_i$ is non-zero}, \\
	0 & \text{otherwise.}
	\end{array} \right.
\]
Clearly $\hat{s}(P)$ only depends on $A$.

\begin{thm} \label{thm:sigma}
	In the previous situation, we have
	\[
		\frac{1}{\hat{s}(P)} = \max_{\substack{J \subseteq \{1, \ldots, e\} \\ J \not= \emptyset }}
		\frac{1}{\# J} \sum_{i=1}^d \min_{j \in J} A_{ij}.
	\]
	Moreover, the integral $\mathbb{I}(P; \hat{s}(P))$ in \eqref{eq:F:integral} diverges.
\end{thm}

The proof has two steps. We first identify $\frac{1}{\hat{s}(P)}$ as the solution of a $\min$-$\max$ problem. This is Lemma~\ref{lem:sigma-as-min-max} below. This lemma is just a special case of the well-known fact that such an exponent of convergence can be read on the Newton polytope of the denominator. We give a proof to be self-contained, but we refer for instance to \cite{BFP14} for the general theory, after transforming the integrals on the cube to integral on the octant by the change of variable $x_j = \frac{y_j}{1 + y_j}$. In a second step, we linearise the optimisation problem and analyse the solution of the equivalent $\max$-$\min$ problem, which leads to the formula in Theorem~\ref{thm:sigma}. As we are not aware of an earlier reference for the latter, we felt the need to include its proof.

\begin{lem}\label{lem:sigma-as-min-max}
	In the previous situation, we have 
	\begin{equation} \label{eq:s:formula}
		\frac{1}{\hat{s}(P)} = \min_{\alpha \in \mathcal{A}} \max_{1 \leq j \leq e} \sum_{i=1}^d \alpha_{ij},
	\end{equation}
	where $\mathcal{A}$ is the compact set
	\[
		\mathcal{A} = \bigg\{(\alpha_{ij})_{\substack{1 \leq i \leq d  \\ 1 \leq j \leq e}} \,\,\bigg|\,\,\forall i,j \quad 0 \leq \alpha_{ij} \leq A_{ij} \qquad {\rm and}\qquad  \forall i\quad \sum_{j = 1}^e \alpha_{ij} \geq 1\bigg\}.
	\]
	Moreover, the integral $\mathbb{I}(P; \hat{s}(P))$ diverges.
\end{lem}

\begin{proof}
	By elementary inequalities, it suffices to consider $P(x_1,\ldots,x_e) \coloneqq \prod_{i=1}^d \left( \sum_{j=1}^e A_{ij} x_j \right)$.
	Let $\alpha \in \mathcal{A}$ and $M_i = \sum_{j = 1}^e \alpha_{ij} \geq 1$.
	Using the concavity of the logarithm and the fact that $x_j \in [0,1]$, we have:
	\[
		\forall i \in \{1,\ldots,d\},\qquad \sum_{j = 1}^e A_{ij}x_j \geq \sum_{j=1}^e \frac{\alpha_{ij}}{M_i} x_j \geq \prod_{j=1}^e x_j^{\frac{\alpha_{ij}}{M_i}} \geq \prod_{j = 1}^e x_j^{\alpha_{ij}}.
	\]
	Taking the product over $i$ we get
	\[
		P(x_1,\ldots,x_e)^s
		\geq \bigg( \prod_{j=1}^e x_j^{\sum_{i=1}^d \alpha_{ij}} \bigg)^s
		\geq (x_1 x_2 \cdots x_e)^{s t_{\alpha}},\qquad t_{\alpha} \coloneqq \max_{1 \leq j \leq e} \sum_{i=1}^d \alpha_{ij}.
	\]
	Now the integral
	\[
		\int_{(0,1]^e} \frac{\dd x_1\cdots \dd x_e}{(x_1 x_2 \cdots x_e)^{s t_{\alpha}}}
		= \left( \int_{0}^1 \frac{\dd x}{x^{st_{\alpha}}} \right)^e
	\]
	converges if and only if $st_{\alpha} < 1$. Since $\alpha \in \mathcal{A}$ was arbitrary, we obtain $\hat{s}(P) \geq \max_{\alpha} \frac{1}{t_{\alpha}}$, in other words the upper bound
	\begin{equation} \label{PnuinA}
		\frac{1}{\hat{s}(P)} \leq \frac{1}{\check{s}(P)} \coloneqq \min_{\alpha \in \mathcal{A}}\max_{1 \leq j \leq e} \sum_{i=1}^d \alpha_{ij}.
	\end{equation}

	\medskip

	Now, we want to show that the integral diverges for $s = \check{s}$. In a first step, we reformulate the minimising problem. Let
	\[
		\mathcal{P} = \bigg\{\Big(\sum_{i = 1}^d \alpha_{ij}\Big)_{1 \leq j \leq e} \,\,\bigg|\,\,\forall i \quad \sum_{j = 1}^d A_{ij}\alpha_{ij} \geq 1\bigg\}.
	\]
	We observe that $\check{s} = \min_{p \in \mathcal{P}'} ||p||_{\infty}$ for the polytope
	\[
		\mathcal{P}' = \bigg\{ \Big(\sum_{i = 1}^d \alpha_{ij}\Big)_{1 \leq j \leq e} \,\,\bigg|\,\,\forall i,j \quad 0 \leq \alpha_{ij} \leq A_{ij} \qquad {\rm and}\qquad \forall i\quad \sum_{j = 1}^e \alpha_{ij} \geq 1\bigg\}.
	\]
	Since $\mathcal{P} = \mathcal{P}' + \big(\mathbb{R}_{\geq 0}\big)^{e}$, it is clear that $\check{s} = \min_{p \in \mathcal{P}} ||p||_{\infty}$. Then, we claim that $\check{s}$ is the minimum $s$ such that $p(s) \coloneqq (\frac{1}{s},\ldots,\frac{1}{s})$ belongs to $\mathcal{P}$ (actually, to a facet of $\mathcal{P}$). Indeed, if $p(s) \in \mathcal{P}$, there exists $(\alpha_{ij})_{i,j}$ such that $\frac{1}{s} = \sum_{i = 1}^d \alpha_{ij}$ for any $j \in \{1,\ldots,e\}$. Since all those values are equal, we also have $\frac{1}{s} = \max_{j} \sum_{i = 1}^d \alpha_{ij}$. The minimal such $s$ is thus given by $\check{s}$.

	\medskip

	In a second step, we are going to upper bound $P(x_1,\ldots,x_e)$ by the power of a single variable, but the choice of the variable and the power will be optimised depending on the region in $(0,1]^e$, leading to a diverging integral for $s = \check{s}$. For this purpose we choose a supporting hyperplane $\mathcal{H}$ for the polytope $\mathcal{P}$ at $(\frac{1}{\check{s}},\ldots,\frac{1}{\check{s}})$. Let $J_{\infty} \subset \{1,\ldots,e\}$ the set of indices $j$ such that $\mathcal{H}$ does not intersect the $j$-th axis. For $j \notin J_{\infty}$, we denote $h_j$ the coordinate of the intersection. Denoting $w_1,\ldots,w_e$ the canonical basis of $\mathbb{R}^e$, we have
	\[
		\mathcal{H} \cap \mathbb{R}_{\geq 0}^e = \bigg\{\sum_{j \notin J_{\infty}} t_j h_j w_j + \sum_{j \in J_{\infty}} t_j w_j \,\,\bigg|\,\,\forall j \quad t_j \geq 0 \quad {\rm and} \quad \sum_{j \notin J_{\infty}} t_j = 1\bigg\}.
	\]
	Let us define
	\[
		Q\,:\,
		\begin{array}{lll}
			(0,1]^e & \longrightarrow & \RR \\
			\,\,\,\,x & \longmapsto & \max\big\{x^u\,\,|\,\,u \in \mathcal{H} \cap \mathbb{R}_{\geq 0}^e\big\}  = \max\big\{x_j^{h_j}\,\,|\,\,1 \leq j \leq e \  {\rm and} \  j \not\in J_{\infty} \big\}
		\end{array}.
	\]
	The equality holds because the linear form $u \mapsto \sum_{j = 1}^e u_j \ln(x_j)$ necessarily reaches its maximum at the vertices which are on the axes. As $\mathcal{H}$ is a supporting hyperplane of $\mathcal{P}$ at $p$, for each monomial $x_1^{u_1}\cdots x_e^{u_e}$ appearing in $P$ we have $x_1^{u_1} \cdots x_e^{u_e} \leq Q(x)$. For $x \in (0,1]^e$, we deduce that $P(x) \leq P(1) \cdot Q(x)$ and hence
	\[ 
		\mathbb{I}(P; \check{s}) \geq \frac{1}{P(1)^{\check{s}}} \int_{(0,1]^e} \frac{\dd x_1 \cdots \dd x_e}{Q(x_1, \ldots, x_e)^{\check{s}}}.
	\]
	We now prove that the integral on the right-hand side diverges. Let $j_0 \in \{1, \ldots, e\} \setminus J_{\infty}$ be an index for which $h_{j_0}$ is minimal, and set $D \coloneqq \Set{x \in (0,1]^e |\forall j \quad x_j \leq x_{j_0}}$. By minimality of $h_{j_0}$:
	\[
		D \subseteq \Set{x \in (0,1]^e\ | \forall j \quad (x_j)^{h_j} \leq (x_{j_0})^{h_{j_0}}}.
	\]
	Hence $Q(x_1, \ldots, x_e) = x_{j_0}^{h_{j_0}} $ on $D$ and
	\begin{equation} \label{inufsgfginf}
		\int_{(0,1]^e} \frac{\dd x_1 \cdots \dd x_e}{Q(x_1, \ldots, x_e)^{\check{s}}} \geq \int_{D} \frac{\dd x_1 \cdots \dd x_e}{Q(x_1, \ldots, x_e)^{\check{s}}} = \int_D \frac{\dd x_1 \cdots \dd x_e}{x_{j_0}^{h_{j_0}\check{s}}} = \int_0^1 \frac{\dd x}{x^{h_{j_0}\check{s} - e + 1}}.
	\end{equation}
	Computing the sum of the coordinates of  $(\frac{1}{\check{s}},\ldots,\frac{1}{\check{s}}) \in \mathcal{H} \cap \mathbb{R}_{\geq 0}^e$, we find that $\frac{e}{\check{s}} \geq \sum_{j \notin J_{\infty}} t_j h_j$. Since $h_{j_0}$ was minimal, this implies $\frac{e}{\check{s}} \geq h_{j_0}$. So: $h_{j_0}\check{s} - e + 1 \leq 1$, implying that the integral \eqref{inufsgfginf} diverges. We conclude that $\mathbb{I}(P; \check{s})$ diverges, and in particular that we have equality in \eqref{PnuinA}.
\end{proof}

\begin{rem}
	The polytope $\mathcal{P}'$ that appears in the proof is the Newton polytope of the polynomial $P$. $\mathcal{P}$ is obtained from $\mathcal{P}'$ by adding the positive rays of $\RR^e$.
\end{rem}

\begin{proof}[Proof of Theorem~\ref{thm:sigma}]
	Consider the simplex $\boldsymbol{\Delta} = \big\{\beta \in \mathbb{R}^e_{\geq 0}\,\, \big|\,\, \sum_{j = 1}^e \beta_j = 1\}$. Since extrema in $\boldsymbol{\Delta}$ of linear forms are reached at vertices, we have
	\[
		\frac{1}{\hat{s}(P)}
		= \min_{\alpha \in \mathcal{A}} \max_{1 \leq j \leq e} \sum_{i=1}^d \alpha_{ij}
		= \min_{\alpha \in \mathcal{A}} \max_{\beta \in \boldsymbol{\Delta}} \sum_{j=1}^e \sum_{i=1}^d  \alpha_{ij}\beta_j . \\
	\]
	Since $\mathcal{A}$ and $\boldsymbol{\Delta}$ are compact convex sets, the min-max principle of von Neumann (see \textit{e.g.} \cite{Sim95}) implies
	\[
		\frac{1}{\hat{s}(P)} = \max_{\beta \in \boldsymbol{\Delta}} \min_{\alpha \in \mathcal{A}} \sum_{j=1}^e \sum_{i=1}^d \alpha_{ij}\beta_j.
	\]
	The minimum over $\alpha$ being reached at the vertices of $\mathcal{A}$, we obtain
	\[
		\frac{1}{\hat{s}(P)}  = \max_{\beta \in \boldsymbol{\Delta}} m(\beta),\qquad {\rm where}\qquad m(\beta) = \sum_{i = 1}^d \min_{1 \leq j \leq e} A_{ij}\beta_j.
	\]
	We claim that the maximum is reached by a vector $\beta$ whose non-vanishing entries are all equal, in other words by a vector of the form
	\[
		\beta_j =
		\begin{dcases}
		\frac{1}{\# J} & \text{if $j \in J$}, \\[2pt]
		0 & \text{otherwise}.
		\end{dcases}
	\]
	for some non-empty subset $J \subseteq \{1,\ldots,e\}$. The thesis is an immediate consequence of this claim.

	\medskip 

	To justify the claim, let us take a maximiser  $\beta$ of $m$ such that the (non-empty) set $\Pi(\beta) = \Set{\beta_j | 1 \leq j \leq e}\setminus \{0\}$ has minimal cardinality. Assume that $\# \Pi(\beta) > 1$. We can pick $b,c \in \Pi(\beta)$ such that $0 < b < c$, and define $J_b = \Set{j \in \{1,\ldots,e\} | \beta_j = b}$ and likewise $J_c$. The sets $J_b$ and $J_c$ are disjoint and non-empty. Given $t \in \mathbb{R}$, we define a vector $\beta^t$ by
	\[
		\beta^t_j = \begin{dcases}
			b + \frac{t}{\# J_b}& \text{if $j \in J_b$}, \\[2pt]
			c - \frac{t}{\# J_c} & \text{if $j \in J_c$}, \\[2pt] 
			\beta_j & \text{otherwise}.
		\end{dcases}
	\]
	For small $t$, $\beta^t$ remains in the simplex since the sum of coordinates of $\beta - \beta^t$ is zero. Furthermore
	\[
		\Pi(\beta^t) = \big(\Pi(\beta) \setminus \{b,c\}\big) \cup \big\{b + \tfrac{t}{\# J_b},c - \tfrac{t}{\# J_c}\big\}.
	\]
	Now define $I_b = \Set{i \in \{1,\ldots,d\} | \min_{1 \leq j \leq e} A_{ij}\beta_j = b}$ and likewise $I_c$. If $\frac{\# I_b}{\# J_b} \neq \frac{\# I_c}{\# J_c}$ then for $t \neq 0$ small enough  we have $m(\beta) \neq m(\beta^t)$. For small $t$, we see that $m(\beta^t)$ is linear in $t$ and non-constant. Given that the sign of $t$ is arbitrary, this contradicts the fact that $\beta$ maximises $m$. Therefore, we must have $\frac{\# I_b}{\# J_b} = \frac{\# I_c}{\# J_c}$, which implies that $m(\beta) = m(\beta^t)$ for $t$ small enough. Now let $t_0$ be the smallest positive $t$ such that $b + \frac{t}{\# J_b}$ or $c - \frac{t}{\# J_c}$ is an element of $\Pi(\beta) \cup \{0\}$. Then $\# \Pi(\beta^{t_0}) < \# \Pi(\beta)$, but by continuity $m(\beta) = m(\beta^{t_0})$, so $\beta^{t_0}$ is a maximiser of $m$, leading to a contradiction with the minimality of $\#\Pi(\beta)$. Returning to the start of the argument, we conclude that $\# \Pi(\beta) = 1$, as desired.
\end{proof}

\section{Discrete integration}
\label{app:discrete:integration}

\subsection{Principle}

Unlike $\mathcal{M}_{g,n}(L)$, the combinatorial moduli space $\mathcal{M}_{g,n}^{{\rm comb}}(L)$ admits an integral structure $\mathcal{M}_{g,n}^{{\rm comb},\mathbb{Z}}(L)$ consisting of those metric ribbon graphs with integral edge lengths. Since each edge is bounded by two (possibly the same) faces, $\mathcal{M}_{g,n}^{{\rm comb},\mathbb{Z}}(L)$ is empty unless $L \in \mathbb{Z}_+$ and $\sum_{i = 1}^n L_i$ is even. We assume this condition throughout this section. Since $\mathcal{M}_{g,n}^{{\rm comb},\mathbb{Z}}(L)$ is finite, we can consider the discrete integration of $\mathscr{B}_{g,n}^{{\rm comb}}$, \textit{i.e.}
\begin{equation} \label{NgnLS}
	\mathscr{N}_{g,n}(L;s) = \sum_{\mathbf{G} \in \mathcal{M}_{g,n}^{{\rm comb},\mathbb{Z}}(L)} \frac{\big(\mathscr{B}_{g,n}^{{\rm comb}}(\mathbf{G})\big)^{s}}{\#{\rm Aut}(\mathbf{G})}.
\end{equation}
\textit{which is well-defined for any $s \in \mathbb{C}$}. We may also rescale the integral structure by a factor $k > 0$ and perform a sum over the set $\mathcal{M}_{g,n}^{{\rm comb},\mathbb{Z}/k}(L)$ of metric ribbon graphs whose edge lengths are integral multiples of $1/k$. The definition of combinatorial lengths functions and $\mathscr{B}_{\Sigma}^{{\rm comb}}$ makes clear that:
\[
	\forall \mathbf{G} \in \mathcal{M}_{g,n}^{{\rm comb}}(L),\qquad \mathscr{B}_{g,n}^{{\rm comb}}(k^{-1}\mathbf{G}) = k^{6g - 6 + 2n} \mathscr{B}_{g,n}^{{\rm comb}}(\mathbf{G})
\]
where $k^{-1}\mathbf{G}$ is the metric ribbon graph $\mathbf{G}$ in which all edge lengths are multiplied by $k^{-1}$. Recall from \cite{Kontsevich} that the Kontsevich measure is essentially a Lebesgue measure:
\[
	\dd\mu_{{\rm K}} \prod_{i = 1}^n \dd L_i = 2^{2g - 2 + n} \prod_{e \in E_G} \dd \ell_e
\]
and that the sublattice of $\mathbb{Z}^n$ where $\sum_{i = 1}^n L_i$ is even, has index $2$. It follows by definition of the Riemann integral, that for $s$ in the range of integrability of $\mathscr{B}_{g,n}^{{\rm comb}}$, we must have
\begin{equation}\label{limk}
	\lim_{\substack{k \rightarrow \infty \\  k \in \mathbb{Z}_+}} k^{(s - 1)(6g - 6 + 2n)} \mathscr{N}_{g,n}(kL;s)
	=
	2^{-(2g - 3 + n)} \int_{\mathcal{M}_{g,n}^{{\rm comb}}(L)} \big(\mathscr{B}^{{\rm comb}}_{g,n}\big)^{s} \dd \mu_{{\rm K}}.
\end{equation}
Note that the contributions of the cells of positive codimension vanish in the limit.

\subsection{Motivation and elementary results}
\label{appB}

Although we do not undertake a systematic study here, we see both geometric and arithmetic reasons why the study of the discrete integration of $\mathscr{B}_{g,n}^{{\rm comb}}$ is an interesting problem.

\medskip

Let us start by recalling the picture in the hyperbolic world. For  a stable punctured surface $\Sigma$, we have the isomorphisms of measured spaces (the maps may be ill-defined on negligible sets):
\begin{equation}\label{Qetun}
	\big(\mathcal{Q}\mathcal{T}_{\Sigma},\mu_{{\rm MV}}'\big) \longrightarrow \big({\rm MF}_{\Sigma} \times {\rm MF}_{\Sigma},\mu_{{\rm Th}} \otimes \mu_{{\rm Th}}\big) \longleftarrow \big(\mathcal{T}_{\Sigma} \times {\rm MF}_{\Sigma},\mu_{{\rm WP}} \otimes \mu_{{\rm Th}} \big)
\end{equation}
where $\Sigma$ is a punctured surface, $\mathcal{Q}\mathcal{T}_{\Sigma}$ the bundle of meromorphic quadratic differentials on $\Sigma$ with simple poles at the punctures, and $\mu_{{\rm MV}}'$ is a suitably normalised Masur--Veech measure. The first morphism consists in taking the horizontal and vertical trajectories of the differential, the second morphism in taking the horocyclic foliation associated to a hyperbolic structure. The Thurston measure $\mu_{{\rm Th}}$ comes from asymptotic of lattice points counting, so one can consider discretised versions of $\int_{\mathcal{M}_{g,n}(0)} \mathscr{B}_{g,n}\,\dd\mu_{{\rm WP}}$ by summing over lattice points instead of integrating, and obtain the integrals by studying the asymptotics of such sums.
\begin{itemize}
	\item[(i)] Lattice points in ${\rm MF}_{\Sigma} \times {\rm MF}_{\Sigma}$ that fill the surface $\Sigma$ are square-tiled surfaces. Their enumeration was studied in~\cite{DGZZ19} in order to compute Masur--Veech volumes, and it enjoys quasi-modularity properties~\cite{Engel21,EO01,EOP08}.

	\item[(ii)] Performing lattice sums along ${\rm MF}_{\Sigma}$ and integration along $\mathcal{M}_{g,n}$ lead to statistics of multicurves for random hyperbolic surfaces. They were studied in \cite{ABCDGLW19,GRpaper,Mirzaergo}.
\end{itemize}
The appearance of even zeta values (Bernoulli numbers) in $\int_{\mathcal{M}_{g,n}} \mathscr{B}_{g,n}\dd\mu_{{\rm WP}}$ can be understood from (i) or (ii), and both (i) and (ii) can be computed by topological recursion.

\medskip

In the combinatorial world, one has to use bordered surfaces of fixed boundary lengths $L$ and consider differentials with double poles, but there is a similar diagram. In fact, $\mathcal{T}_{\Sigma}^{{\rm comb}}$ can already be realised \cite{WKarticle} as a subset of a space of measured foliations ${\rm MF}_{\Sigma}'$ (differing from ${\rm MF}_{\Sigma}$ by the choice of boundary behavior), replacing the right part of \eqref{Qetun}, and it is equipped with Kontsevich measure which also comes from asymptotics of lattice point counts. We therefore have three discretised versions of $\int_{\mathcal{M}_{g,n}(L)} \mathscr{B}_{\Sigma}^{{\rm comb}} \dd \mu_{{\rm K}}$:
\begin{itemize}
	\item[(I)] Lattice sums along ${\rm MF}_{\Sigma}$ and integration along $\mathcal{M}_{g,n}^{{\rm comb}}(L)$ leads to statistics of multicurves for random combinatorial surfaces. They were studied in \cite{WKarticle}.

	\item[(II)] Integration along ${\rm MF}_{\Sigma}$ and lattice sums along $\mathcal{M}_{g,n}^{{\rm comb}}(L)$ leads to $\mathscr{N}_{g,n}(L;s = 1)$ defined in \eqref{NgnLS}.

	\item[(III)] Sums over lattice points in $\mathcal{M}_{g,n}^{{\rm comb}}(L) \times {\rm MF}_{\Sigma}$. The latter are ordered pairs $(\mathbf{G},\gamma)$ where $\mathbf{G}$ is a metric ribbon graph and $\gamma$ is a multicurve, which in view of Lemma~\ref{lem:MF:parametrisation} can be identified with an integral point in $Z_G$.  
\end{itemize}
Here we will not discuss (III) and content ourselves with elementary facts about (II). An explicit evaluation can be carried out for the  $(1,1)$ case.

\begin{prop}\label{N11L1}
	For $L \in 2\mathbb{Z}_+$, we have:
	\[
		\mathscr{N}_{1,1}(L;1)
		=
		\sum_{\mathbf{G} \in \mathcal{M}_{1,1}^{{\rm comb},\mathbb{Z}}(L)} \frac{\mathscr{B}_{1,1}^{{\rm comb}}(\mathbf{G})}{\#{\rm Aut}(\mathbf{G})}
		=
		\frac{1}{4} \sum_{k = 1}^{\frac{L}{2} - 1} \frac{1}{k^2}.
	\]
	In generating series form:
	\[
		\sum_{L > 0} \mathscr{N}_{1,1}(L;1)z^{\frac{L}{2}}
		=
		\frac{1}{4}\,\frac{z \, {\rm Li}_2(z)}{1 - z}\,.
	\]
\end{prop}

We know from \eqref{limk} with \eqref{MVnorm} that:
\[
	\lim_{\substack{k \rightarrow \infty \\  k \in \mathbb{Z}_+}} \mathscr{N}_{1,1}(kL;1) = 2^{-(2g - 3 + n)} \int_{\mathcal{M}_{1,1}^{{\rm comb}}(L)} \mathscr{B}^{{\rm comb}}_{1,1} \dd \mu_{{\rm K}} = \frac{\zeta(2)}{4} = \frac{\pi^2}{24}.
\] 
This indeed agrees with the formula for $\mathscr{N}_{1,1}(L;1)$, and we see that it involves truncations of the series defining $\zeta(2)$. For general $(g,n)$, we can give the following formula which performs the lattice sum over $\mathcal{M}_{g,n}^{{\rm comb}}$.

\begin{prop}\label{igbuegbg}
	For each ribbon graph of type $(g,n)$, fix a simplicial decomposition $T_G$ of $Z_G$. We have
	\begin{equation*}
	\begin{split}
		& \quad \sum_{\substack{L_1,\ldots,L_n > 0 \\ L_1 + \cdots + L_n\,\,{\rm even}}} \mathscr{N}_{g,n}(L_1,\ldots,L_n;1) \prod_{i = 1}^n z_i^{L_i} \\
		& = \sum_{G \in \mathcal{R}_{g,n}} \frac{1}{\# {\rm Aut}(G)} \sum_{t \in T_G} \frac{1}{(6g - 6 + 2n)! \det(t)} \int_{[0,1]^{R(t)}} \prod_{\rho \in R(t)} \frac{\dd x_{\rho}}{x_{\rho}} \prod_{e \in E_G} \left(\frac{\prod_{i = 1}^n z_i^{P_{i,e}} \prod_{\rho \in R(t)} x_{\rho}^{\overline{P}_{\rho,e}}}{1 - \prod_{i = 1}^n z_i^{P_{i,e}} \prod_{\rho \in R(t)} x_{\rho}^{\overline{P}_{\rho,e}}}\right),
	\end{split}
	\end{equation*}
	where $P_{i,e}$ (resp. $\overline{P}_{\rho,e}$) is the number of times the edge $e$ appears along the $i$-th boundary face (resp. the ray $\rho$) --- counted with multiplicity.
\end{prop}

This suggests that the general $(g,n)$ case could have interesting arithmetics, possibly in relation with polylogarithms, and we know \textit{a priori} that it should make appear truncations of even zeta values.

\medskip 

Another way to study the integrability property of $\big(\mathscr{B}_{g,n}^{{\rm comb}}\big)^s$ is to study the result of the discrete integration $\mathscr{N}_{g,n}(kL;s)$ for large integral $k$. The non-integrability cases will be detected by an anomalous scaling of this function, \textit{i.e.} a growth faster than $k^{-(s - 1)(6g - 6 + 2n)}$ when $k \rightarrow \infty$. This can also be read from the dominant singularity of the generating series $N_{g,n}(z;L,s) = \sum_{k > 0}  \mathscr{N}_{g,n}(kL;s)\,z^{k}$. Namely, we expect logarithmic singularities for $N_{g,n}(z;L,s)$ when $s = s^*_{g,n}$, which will correspond to the appearance of logarithms in the large $k \rightarrow \infty$ asymptotics of $\mathscr{N}_{g,n}(kL;s^*_{g,n})$. We do not venture in a systematic singularity analysis, but give for $(g,n) = (1,1)$ evidence of the logarithmic behavior by an elementary argument.

\begin{prop}\label{lemauntr}
	There exists $c_2 > c_1 > 0$, such that
	\[
		\forall L \in 2\mathbb{Z}_+,\qquad c_1 \frac{\ln L}{L^2} \leq \mathscr{N}_{1,1}(L;s = 2)  \leq c_2\,\frac{\ln L}{L^2}.
	\]
\end{prop}

The three propositions will be proved in the next two subsections.

\subsection{The \texorpdfstring{$(1,1)$}{(1,1)} case}

\begin{proof}[Proof of Proposition~\ref{N11L1}]
	Our starting point is \eqref{B11combform} for $\mathscr{B}_{1,1}^{{\rm comb}}$, which yields for $L \in 2\mathbb{Z}_+$:
	\[
		\mathscr{N}_{1,1}(L;1) = \frac{1}{4} S_{\frac{L}{2}} + \frac{1}{4}T_{\frac{L}{2}},
	\]
	with 
	\[
		S_\ell \coloneqq \sum_{\substack{a + b + c = \ell \\ a,b,c > 0}} \frac{1}{(a + b)(b + c)},
		\qquad
		T_{\ell} \coloneqq \sum_{\substack{a + b = \ell \\ a,b > 0}} \frac{1}{ab}.
	\]
	Here the first terms corresponds to integer points in the top-dimensional cell, while the second sum counts for the codimension-$1$ cell $\mathfrak{Z}_{G'}(L) = \set{ (a,b) \in \RR_+^2 | a + b = \tfrac{L}{2} }$ associated to the unique $4$-valent ribbon graph $G'$ of type $(1,1)$, whose automorphism group is $\ZZ_4$. We can simplify the second sum as
	\[
		T_{\frac{L}{2}} = \sum_{k = 1}^{\frac{L}{2} - 1} \frac{1}{k(\frac{L}{2} - k)} = \frac{1}{\frac{L}{2}} \sum_{k = 1}^{\frac{L}{2}-1} \Big(\frac{1}{k} + \frac{1}{\frac{L}{2} - k} \Big)= \frac{4}{L} \sum_{k = 1}^{\frac{L}{2}-1} \frac{1}{k}.
	\]
	For the record, its generating series is
	\begin{equation} \label{Teqn}
		T(z) = \sum_{\ell > 0} T_{\ell}\,z^{\ell} = \sum_{a,b > 0} \frac{z^{a + b}}{ab} = \ln^2(1 - z).
	\end{equation}
	The first sum could be evaluated by direct manipulations, but we prefer a generating series approach, as it can be adapted (Section~\ref{Gndis}) in any topology. We introduce 
	\begin{equation}
		S(z) = \sum_{\ell > 0} S_{\ell}\,z^{\ell}.
	\end{equation}
	We also introduce the refined generating series
	\[
		S(z_1,z_2,z_3) = \sum_{a,b,c > 0} \frac{z_1^{a + b}z_2^{b + c}z_3^{c + a}}{(a + b)(b + c)},\qquad S(z) = S(\sqrt{z},\sqrt{z},\sqrt{z}).
	\] 
	The advantage is that taking derivatives with respect to $z_1$ and $z_2$, we can decouple the summation variables and recognise geometric series. Indeed, for  $z_1,z_2,z_3 \in [0,1)$ 
	\[
		z_1z_2\partial_{z_1}\partial_{z_2}S(z_1,z_2,z_3) = \frac{z_1^2z_2^2z_3^2}{(1 - z_1z_2)(1 - z_1z_3)(1 - z_2z_3)}.
	\]
	Since $S(z_1,z_2,z)$ vanishes when $z_1 = 0$ or $z_2 = 0$, we get by integration
	\begin{equation}
	\begin{split}\label{S123}
		S(z_1,z_2,z_3)
		& =
		\int_{0}^{z_1} \int_{0}^{z_2} \frac{x_1x_2 \, z_3^2 \,\dd x_1\dd x_2}{(1 - x_1x_2)(1 - z_3x_1)(1 - z_3x_2)} \\
		& =
		\int_{[0,1]^2} \frac{z_1^2z_2^2z_3^2 \, y_1y_2 \, \dd y_1\dd y_2}{(1 - z_1z_2 y_1y_2)(1 - z_1z_3 y_1)(1 - z_2z_3 y_2)}.
	\end{split}
	\end{equation}
	Hence:
	\[
	\begin{split}
		S(z)
		& =
		\int_{0}^{1}\int_{0}^{1} \frac{z^3\,y_1y_2\,\dd y_1\dd y_2}{(1 - zy_1y_2)(1 - zy_1)(1 - zy_2)}
		=
		\int_{0}^{1} z\,\dd y_2 \bigg(\frac{\ln\big(\frac{1 - z y_2}{1 - z}\big)}{(1 - y_2)(1 - z y_2)} + \frac{\ln(1 - z)}{1 - z y_2}\bigg) \\
		& =
		\bigg(\frac{z}{1 - z} \int_{0}^{\frac{z}{1 - z}} \dd u\, \frac{\ln(1 + u)}{u(1 + u)}\bigg) - \ln^2(1-z)
		=
		\frac{z\,{\rm Li}_2(z)}{1 - z} - \ln^2(1-z),
	\end{split} 
	\] 
	where we identified the dilogarithm in the last line by differentiating the integral with respect to $z$ and integrating again from $0$ to $z$. We therefore have a simplification with \eqref{Teqn}:
	\[
		\sum_{\ell > 0} \mathscr{N}_{1,1}(2\ell;1) z^{\ell} = \frac{1}{4}\,\frac{z\,{\rm Li}_2(z)}{1 - z}.
	\]
	The expansion in powers of $z$ yields
	\[
		\mathscr{N}_{1,1}(2\ell;1) = \frac{1}{4} \sum_{k = 1}^{\ell - 1} \frac{1}{k^2}.
	\]
\end{proof}

\begin{proof}[ Proof of Proposition~\ref{lemauntr}]
	We have $\mathscr{N}_{1,1}(L;2) =  \frac{1}{6} \tilde{S}_{\frac{L}{2}} + \frac{1}{4}\tilde{T}_{\frac{L}{2}}$, where the first (resp. second) term is the contribution from the top-dimensional (codimension $1$) cell, namely
	\begin{equation*}
	\begin{split}
		\tilde{S}_{\ell} & = \frac{1}{4} \sum_{\substack{a + b + c = \ell \\ a,b,c > 0}} \bigg(\frac{1}{(a+b)(a+c)} + \frac{1}{(a + b)(b + c)} + \frac{1}{(a + c)(b + c)}\bigg)^2,  \\
		\tilde{T}_{\ell} & \coloneqq \sum_{\substack{ a + b = \ell \\ a,b > 0}} \frac{1}{a^2b^2}.
	\end{split}
	\end{equation*}
	This last expression can be evaluated as follows.
	\begin{equation*}
	\begin{split}
		\tilde{T}_{\ell} & = \sum_{k = 1}^{\ell - 1} \frac{1}{k^2(\ell - k)^2} = \sum_{k = 1}^{\ell - 1} \frac{1}{k^2\ell^2} + \frac{1}{(\ell - k)^2\ell^2} + \frac{2}{k\ell^3} + \frac{2}{(\ell - k)\ell^3} \\
		& = \frac{2}{\ell^2} \sum_{k = 1}^{\ell - 1} \frac{1}{k^2} + \frac{4}{\ell^3} \sum_{k = 1}^{\ell - 1} \frac{1}{k}.
	\end{split}
	\end{equation*}
	Therefore $\tilde{T}_{\ell} = O(\ell^{-2})$ when $\ell \rightarrow \infty$. Let us transform the first expression:
	\begin{equation*}
	\begin{split}
		\tilde{S}_{\ell} & = \frac{1}{4} \sum_{\substack{a + b + c = \ell \\ a,b,c > 0}} \bigg(\frac{1}{(a+b)(a+c)} + \frac{1}{(a + b)(b + c)} + \frac{1}{(a + c)(b + c)}\bigg)^2 \\
		& = \sum_{\substack{a + b + c = \ell \\ a,b,c > 0}}  \frac{3}{4}\,\frac{1}{(a + b)^2(b + c)^2} + \frac{3}{2}\,\frac{1}{(a + b)(b + c)(a + c)^2}.
	\end{split}
	\end{equation*}
	Given a triple of integer summing up to $\ell$, at least one of them is $\geq \ell/3$. Therefore
	\begin{equation*}
	\begin{split}
		\tilde{S}_{\ell} & \leq \sum_{\substack{a + b + c = \ell \\ a,b,c > 0}} \frac{3}{4}\bigg(\frac{2}{(\ell/3)^2(a + b)^2} + \frac{1}{(\ell/3)^4}\bigg) + \frac{3}{2}\bigg(\frac{2}{(\ell/3)^3(a + b)} + \frac{1}{(\ell/3)^2(a + c)^2}\bigg) \\
		& \leq \frac{3}{4}\bigg(\frac{54\ln \ell}{\ell^2} + \frac{81(\ell^2/2)}{\ell^4}\bigg) + \frac{3}{2}\bigg(\frac{54\ln \ell}{\ell^3} + \frac{27\ln \ell}{\ell^2}\bigg) \leq \frac{c_2'\ln \ell}{\ell^2},
	\end{split}
	\end{equation*}
	provided we choose $c_2' > 81$ and $\ell$ large enough. Since any linear factor in the denominators is $\leq \ell$, we get
	\begin{equation*}
	\begin{split}
		\tilde{S}_{\ell} & \geq  \sum_{\substack{a + b + c = \ell \\ a,b,c > 0}}\frac{\frac{3}{4} + \frac{3}{2}}{\ell^2(a + b)^2} \geq c_1'\,\frac{\ln \ell}{\ell^2},
	\end{split} 
	\end{equation*}
	provided we choose $c_1' \in (0,9)$ and $\ell$ large enough. To get both inequalities we relabeled summation indices to collect terms. 
\end{proof}

\subsection{An integral formula for arbitrary \texorpdfstring{$(g,n)$}{(g,n)}: proof of Proposition~\ref{igbuegbg}}
\label{Gndis}

We want to compute
\begin{equation} \label{Ngnzn}
\begin{split}
	& \sum_{\substack{L_1,\ldots,L_n > 0 \\ L_1 + \cdots + L_n\,\,{\rm even}}} \mathscr{N}_{g,n}(L_1,\ldots,L_n;1) \prod_{i = 1}^n z_i^{L_i}  \\
	& \quad = \sum_{G \in \mathcal{R}_{g,n}} \frac{1}{\#{\rm Aut}(G)} \sum_{\ell\,:\;E_G \rightarrow \mathbb{Z}_+} \mathscr{B}_{g,n}^{{\rm comb}}(G,\ell)\,\prod_{\substack{1 \leq i \leq n \\ e \in E_G}} z_i^{P_{i,e}\ell_e} \\
	& \quad = \sum_{G \in \mathcal{R}_{g,n}} \frac{1}{\#{\rm Aut}(G)} \sum_{\ell\,:\;E_G \rightarrow \mathbb{Z}_+}  \sum_{t \in T_G} \frac{1}{d_{g,n}!\det(t)} \,\frac{\prod_{i = 1}^n \prod_{e \in E_G} z_i^{P_{i,e}\ell_e}}{\prod_{\rho \in R(t)} \big(\sum_{e \in E_G} \overline{P}_{\rho,e} \ell_e\big)},
\end{split}
\end{equation}
where $d_{g,n} = 6g - 6 + 2n$. Here $(G,\ell) \in \mathcal{M}_{g,n}^{{\rm comb},\mathbb{Z}}$ is the ribbon graph $G$ equipped with the metric $\ell$, and we have used in the last line Proposition~\ref{prop:volume:rational:fnct}. To handle the sum over $\ell$, we generalise the trick seen in the proof of Proposition~\ref{N11L1}. Namely, for fixed $G$ and $t$, we introduce the easily-computable refined generating series
\[
	\sum_{\ell\,:\,E_G \rightarrow \mathbb{Z}_+} \prod_{e \in E_G} \bigg(\prod_{i  = 1}^n  z_i^{P_{i,e}\ell_e} \prod_{\rho \in R(t)} x_{\rho}^{\overline{P}_{\rho,e}\ell_e}\bigg) = \prod_{e \in E_G} \left(\frac{\prod_{i = 1}^n z_i^{P_{i,e}} \prod_{\rho \in R(t)} x_{\rho}^{\overline{P}_{\rho,e}}}{1 - \prod_{i = 1}^n z_i^{P_{i,e}} \prod_{\rho \in R(t)} x_{\rho}^{\overline{P}_{\rho,e}}}\right)
\]
and we observe that multiplying it by $\prod_{\rho \in R(t)} \frac{\dd x_{\rho}}{x_{\rho}}$ and integrating over $x_{\rho} \in [0,1]^{R(t)}$ yields the sums
\[
	\sum_{\ell\,:\,E_G \rightarrow \mathbb{Z}_+} \frac{\prod_{i = 1}^n \prod_{e \in E_G} z_i^{P_{i,e}\ell_e}}{\prod_{\rho \in R(t)} \big(\sum_{e \in E_G} \overline{P}_{\rho,e} \ell_e\big)}
\]
which appear in \eqref{Ngnzn}. \hfill $\blacksquare$

\section{Index of notations}
\label{sec:notation}
%
{\small
\renewcommand{\arraystretch}{1.4}
\begin{center}
\begin{tabulary}{\textwidth}{c|c|L}
$G$, $\bm{G}$
	& Sec.~\ref{sec:tcomb}
	& Ribbon graph, metric ribbon graph
	\\
$\Aut(G) \supset \Aut(\bm{G})$
	& Sec.~\ref{sec:tcomb}
	& Their automorphism groups
	\\	
$(\bm{G},f)$, $\GG$
	& Sec.~\ref{sec:tcomb}
	& Combinatorial marking, combinatorial structure
	\\
$M_{\Sigma}$
	& Sec.~\ref{sec:tcomb}
	& Set of multicurves of $\Sigma$ (excluding boundary-parallel components)
	\\
$M_{\Sigma}^{\bullet}$
	& Sec.~\ref{sec:param:mf}
	& Set of multicurves of $\Sigma$ (including boundary-parallel components)
	\\
$\mathfrak{Z}_{G}(L)$
	& Eqns.~\eqref{eqn:cell:moduli}-\eqref{eqn:cell:Teich}
	& The space $\RR_{+}^{E_{G}}$ of all possible metrics on a ribbon graph $G$
	\\
$\mathcal{T}_{\Sigma}^{\textup{comb}}(L)$
	& Eqn.~\eqref{eqn:comb:Teich}
	& Combinatorial Teichm\"uller space of $\Sigma$ with fixed boundary lengths $L \in \RR_+^n$
	\\
$\mathcal{M}_{g,n}^{\textup{comb}}(L)$, $\mu_{\textup{K}}$
	& Eqns.~\eqref{eqn:comb:moduli}-\eqref{eqn:integral:functions}
	& Combinatorial moduli space of metric ribbon graphs of type $(g,n)$ with fixed boundary lengths $L \in \RR_+^n$ and associated Kontsevich measure
	\\
${\rm MF}_{\Sigma}$, $\mu_{\textup{Th}}$
	& Sec.~\ref{sec:param:mf}
	& Measured foliations on $\Sigma$ with boundary components being singular leaves and the associated Thurston measure
	\\
${\rm MF}_{\Sigma}^{\bullet}$, $\mu_{\textup{Th}}^{\bullet}$
	& Sec.~\ref{sec:param:mf}
	& Measured foliations on $\Sigma$ with boundary components being leaves (both singular or smooth) and the associated Thurston measure
	\\
$\mathscr{B}_{\Sigma}^{\textup{comb}} \colon \mathcal{T}_{\Sigma}^{\textup{comb}}(L) \to \RR_+$
	& Eqn.~\eqref{eqn:Bcomb}
	& Thurston volume of the unit ball of measured foliations with respect to the combinatorial length function on $\mathcal{T}_{\Sigma}^{\textup{comb}}(L)$
	\\
$\mathscr{B}_{g,n}^{\textup{comb}} \colon \mathcal{M}_{g,n}^{\textup{comb}}(L) \to \RR_+$
	& Prop.~\ref{prop:Bcomb:properties}
	& Associated function on the moduli space
	\\
$\mathfrak{m}_{[G,f]} \colon {\rm MF}_{\Sigma}^{\bullet} \to \RR_{\geq 0}^{E_G}$
	& Eqn.~\eqref{eqn:int:numb:map}
	& Intersection of the embedded ribbon graph $[G,f]$ and a given foliation $\mathcal{F}$
	\\
$Z_G^{\bullet}$
	& Lem.~\eqref{lem:MF:parametrisation}
	& Image of ${\rm MF}_{\Sigma}^{\bullet}$ via $\mathfrak{m}_{[G,f]}$
	\\
$Z_G$
	& Lem.~\eqref{lem:MF:parametrisation}
	& Image of ${\rm MF}_{\Sigma}$ via $\mathfrak{m}_{[G,f]}$
	\\
$t \in T_G$
	& Sec.~\ref{sec:explicit:bcomb}
	& A cone $t$ in a simplicial decomposition $T_G$ of $Z_G$
	\\
$R(t)$
	& Sec.~\ref{sec:explicit:bcomb}
	& Set of generators of the extremal rays of the simplicial cone $t$
	\\
\end{tabulary}
\end{center}
}


\bibliographystyle{amsplain}
\bibliography{BibliBcomb}

\end{document}